\documentclass[12pt, a4paper]{article}

\usepackage{amssymb,amsmath,amsthm,bbm}
\usepackage{bigfoot}
\usepackage[all]{xy}

\usepackage{amssymb,amsmath,amsthm}
\usepackage[mathscr]{eucal}%
\usepackage[all]{xy}
\usepackage{lmodern}
\usepackage[T1]{fontenc}
\usepackage[utf8]{inputenc}
\usepackage{tikz-cd}
\usepackage[shortlabels]{enumitem}
\usepackage{relsize}
\usepackage{geometry}
\usepackage{mathtools}
\usepackage{adjustbox}
\usepackage{lscape}
\usepackage{diagbox}
\usepackage{slashbox}
\usepackage{bm}
\usepackage{caption}
\usepackage{float}

\usepackage{stmaryrd}
\usepackage{mathrsfs}  
\usepackage{multirow}
\usepackage[hyperfootnotes=false]{hyperref}
\usetikzlibrary{positioning}
\usepackage{tikz-3dplot}
\usepackage{xifthen}
\usetikzlibrary{quotes,arrows.meta}
\tdplotsetmaincoords{60}{125}
\tdplotsetrotatedcoords{0}{0}{0}
\usepackage{xcolor}
\definecolor{mygreen}{rgb}{0.16,.55,0.0}
\usepackage[nottoc]{tocbibind}

\renewcommand{\theenumi}{{\upshape(\roman{enumi})}}

\renewcommand{\theenumii}{{\upshape(\alph{enumii})}}

\theoremstyle{plain}
\newtheorem{prop}{Proposition}[subsection]

\newtheorem{lem}[prop]{Lemma}
\newtheorem{lm}[prop]{Lemma}
\newtheorem{thm}[prop]{Theorem}
\newtheorem{cor}[prop]{Corollary}
\theoremstyle{definition}
\newtheorem{definit}[prop]{Definition}
\newtheorem{ex}[prop]{Example}
\newtheorem{rem}[prop]{Remark}
\newtheorem{ques}[prop]{Question}

\theoremstyle{plain} 
\newtheorem{prop1}{Proposition}[section]

\newtheorem{lem1}[prop1]{Lemma}

\newtheorem{thm1}[prop1]{Theorem}
\newtheorem{cor1}[prop1]{Corollary}
\theoremstyle{definition}

\newtheorem{ques1}[prop1]{Question}

\theoremstyle{plain}

\theoremstyle{definition}

\theoremstyle{plain}

\theoremstyle{definition}

\def\varddots{\mathinner{\raise7pt\vbox{\kern3pt\hbox{.}}\mkern1mu\smash{\raise4pt\hbox{.}}\mkern1mu\smash{\raise1pt\hbox{.}}}}

\newcommand{\vplim}{\varprojlim}

\DeclareMathAlphabet{\mathpzc}{OT1}{pzc}{m}{it}

\DeclarePairedDelimiter{\vabs}{\lvert}{\rvert}

\DeclareMathOperator{\Hom}{Hom}
\DeclareMathOperator{\cont}{cont}
\DeclareMathOperator{\Ind}{Ind}
\DeclareMathOperator{\cInd}{c-Ind}
\DeclareMathOperator{\Res}{Res}

\DeclareMathOperator{\GL}{GL}

\DeclareMathOperator{\Gal}{Gal}

\DeclareMathOperator{\im}{im}

\DeclareMathOperator{\Proj}{Proj}
\DeclareMathOperator{\Inj}{Inj}

\DeclareMathOperator{\St}{St}

\DeclareMathOperator{\rad}{rad}

\DeclareMathOperator{\gr}{gr}

\newcommand{\m}{\mathfrak{m}}

\newcommand{\Z}{{\mathbb Z}}

\newcommand{\cC}{\mathcal{C}}

\newcommand{\QK}[1]{I_{1}}

\newcommand{\Qpf}{\mathbb{Q}_{p^f}}
\newcommand{\Qp}{\mathbb{Q}_{p}}
\newcommand{\Zpf}{\mathbb{Z}_{p^f}}
\newcommand{\Zp}{\mathbb{Z}_{p}}
\newcommand{\F}{\mathbb{F}}
\newcommand{\Fp}{\mathbb{F}_{p}}
\newcommand{\Fpf}{\mathbb{F}_{p^f}}
\newcommand{\Fq}{\mathbb{F}_{q}}

\newcommand{\Qpbar}{\overline{\mathbb{Q}}_p}

\newcommand{\Fpbar}{\overline{\mathbb{F}}_p}

\newcommand{\val}{\mathrm{val}}

\newcommand{\Q}{\mathbb{Q}}

\newcommand{\R}{\mathbb{R}}

\newcommand{\wt}[1]{\widetilde{#1}}

\newcommand{\onto}{\twoheadrightarrow}
\newcommand{\into}{\hookrightarrow}
\newcommand{\congto}{\xrightarrow{\,\sim\,}}

\newcommand{\cF}{\mathcal{F}}

\newcommand{\cO}{\mathcal{O}}
\newcommand{\et}{\acute{\mathrm{e}}\mathrm{t}}

\DeclareMathOperator{\ind}{ind}

\newcommand{\Av}{{\mathbb A}_{F^+}^{\infty,v}}
\newcommand{\Avo}{{\mathbb A}_{F^+}^{\infty,v_0}}
\newcommand{\Ap}{{\mathbb A}_{F^+}^{\infty,p}}
\newcommand{\A}{{\mathbb A}_{F^+}^{\infty}}

\newcommand{\cI}{\mathcal{I}}

\newcommand{\rbar}{\overline{r}}
\newcommand{\rhobar}{\overline{\rho}}

\newcommand{\gK}{{\Gal}(\Qpbar/K)}

\newcommand{\matr}[4]{\begin{pmatrix}{#1}&{#2}\\ {#3}&{#4}\end{pmatrix}}
\newcommand{\smatr}[4]{\bigl(\begin{smallmatrix} {#1}& {#2}\\ {#3}&{#4}\end{smallmatrix}\bigr)}%
\newcommand{\smat}[1]{\left( \begin{smallmatrix} #1 \end{smallmatrix} \right)}

\newcommand{\un}[1]{\underline{#1}}

\newcommand{\ra}{\rightarrow}

\newcommand{\ppar}[1]{(\mkern-3mu(#1)\mkern-3mu)}
\newcommand{\bbra}[1]{\llbracket #1\rrbracket}

\newcommand{\Florian}[1]{\textcolor{magenta}{\makebox[0pt][l]{\textcolor{white}{@FH}}#1}}

\newcommand{\mar}[1]{\marginpar{\raggedright\tiny @#1}}

\newcommand{\fm}{\mathfrak{m}}

\DeclareMathOperator{\id}{id}

\DeclareMathOperator{\HOM}{HOM}

\newcommand{\cts}{_{\mathrm{cts}}}

\newcommand{\xto}[1][]{\xrightarrow{#1}}
\newcommand{\simto}{\xto[\sim]} %

\renewcommand{\subset}{\subseteq}
\renewcommand{\supset}{\supseteq}
\renewcommand{\simeq}{\cong}

\newcommand{\vp}{\varphi}
\newcommand{\ok}{\cO_K}
\newcommand{\oks}{\ok\s}

\newcommand{\zp}{\Z_p}

\newcommand{\rhob}{\overline{\rho}} 
\renewcommand{\o}[1]{\overline{#1}}
\newcommand{\s}{^\times}
\newcommand{\wh}[1]{\widehat{#1}}

\newcommand{\fq}{\F_q}
\newcommand{\zlt}{Z_\mathrm{LT}}
\newcommand{\zok}{Z_{\ok}}
\newcommand{\zokg}{\zok^{\mathrm{gen}}}
\DeclareMathOperator{\res}{res}
\newcommand{\cG}{\mathcal{G}}
\newcommand{\plim}{\varprojlim}
\newcommand{\ilim}{\varinjlim}   
\newcommand{\zltg}{\zlt^{\mathrm{gen}}}

\newcommand{\LT}{\mathrm{LT}}
\newcommand{\Perf}{\mathrm{Perf}}
\newcommand{\perf}{^\mathrm{perf}}
\newcommand{\pr}{\mathrm{pr}}

\DeclareMathOperator{\Spa}{\mathrm{Spa}}

\renewcommand{\ss}{\mathrm{ss}}

\newcommand{\proet}{\mathrm{proet}}
\newcommand{\qproet}{\mathrm{qproet}}
\newcommand{\qproetp}{\mathrm{qproet}'}
\newcommand{\an}{\mathrm{an}}
 
\theoremstyle{plain}

\definecolor{olive}{rgb}{0.5, 0.5, 0.0}%

\title{To be or not to be local\footnote{The title is borrowed from Alain Aspect, Nature 446, 2007, 866--867 (in a different context!).}}

\DeclareNewFootnote{AAffil}[alph]

\usepackage{etoolbox}
\makeatletter
\patchcmd\maketitle{\def\@makefnmark{\rlap{\@textsuperscript{\normalfont\@thefnmark}}}}{}{}{}
\makeatother

\makeatletter
\def\thanksAAffil#1{%
  \footnotemarkAAffil\protected@xdef\@thanks{\@thanks%
        \protect\footnotetextAAffil[\the \c@footnoteAAffil]{#1}}%
}
\def\thanksANote#1{%
  \footnotemarkANote%
  \protected@xdef\@thanks{\@thanks%
        \protect\footnotetextANote[\the \c@footnoteANote]{#1}}%
}
\makeatother

\author{Christophe Breuil\thanksAAffil{CNRS, B\^atiment 307, Facult\'e d'Orsay, Universit\'e Paris-Saclay, 91405 Orsay Cedex, France}
\and
Florian Herzig\thanksAAffil{Dept.\ of Math., Univ.\ of Toronto, 40 St.\ George St., Toronto, ON M5S 2E4, Canada}
\and
Yongquan Hu\thanksAAffil{Morningside Center of Math., No.\ 55, Zhongguancun East Road, Beijing, 100190, China}
\and
Karol Kozio\l\thanksAAffil{CUNY Baruch College, 55 Lexington Ave., New York, NY 10010, USA}
\and
Stefano Morra\thanksAAffil{LAGA, 99 Avenue Jean Baptiste Cl\'ement, 93430 Villetaneuse, France}
\and
Benjamin Schraen\thanksAAffil{Universit\'e Claude-Bernard-Lyon-I, Institut Camille Jordan, 69622 Villeurbanne, France}
\and
Sug Woo Shin\thanksAAffil{UC Berkeley, Berkeley, CA 94720, USA / KIAS, Seoul 02455, Republic of Korea}}

\date{ }

\begin{document} 

\maketitle

\setcounter{tocdepth}{3}

\vspace{-8ex}

\begin{abstract}
Let $p$ be a prime number and $K$ a finite unramified extension of $\Qp$. For a smooth representation $\pi$ of $\GL_2(K)$ occurring in some Hecke eigenspace of the mod $p$ cohomology of a Shimura curve, we explore different strategies (inspired by the case $K=\Qp$) to attack the locality question: does $\pi$ depend only on the underlying $2$-dimensional representation $\rhobar$ of ${\rm Gal}(\overline K/K)$? In particular when $[K:\Qp]=2$, crucially using perfectoid geometry, we associate to $\rhobar$ an infinite-dimensional mod $p$ smooth representation of $\smatr{K\s}{K}{0}{1}$ which we hope is the restriction to $\smatr{K\s}{K}{0}{1}$ of the (irreducible) supersingular subquotient of $\pi$.
\end{abstract}

\tableofcontents

\section{Introduction}\label{sec:intro}

Let $p$ be a prime number and $K$ a finite extension of $\Q_p$. In the classical formulation of the (hoped for) mod $p$ Langlands correspondence for $\GL_2(K)$, one would like to relate two-dimensional continuous representations $\rhob$ of $\Gal(\overline{K}/K)$ in characteristic $p$ to certain finite length infinite-dimensional smooth representations $\pi$ of $\GL_2(K)$ in characteristic $p$. The case $K=\Q_p$ has been known for years (\cite{Colmez}, \cite{emerton-local-global}, \cite{CDP}), but the case $K\neq\Q_p$ has been challenging specialists for a quarter of a century, at least in part due to a lack of classification of irreducible smooth representations of $\GL_2(K)$ in characteristic $p$ (when $K\ne \Qp$). 

Inspired by $K=\Qp$, it is natural to first look for the representations of $\GL_2(K)$ in the mod $p$ Betti cohomology $H^\bullet$ of towers of Shimura varieties, or more precisely in the Hecke-isotypic subspaces of $H^d$, where $d$ is the dimension of the relevant Shimura varieties and ``Hecke-isotypic'' means with respect to modular globalizations $\rbar$ of the Galois representations $\rhob$. When $d\in \{0,1\}$ it is known in many cases that the isomorphism class of such a Hecke-isotypic subspace $\pi$ determines the one of $\rhob$ (\cite{BD}, \cite{ScholzeLT}). When moreover $K$ is unramified and under several technical hypothesis (on $p$, $\rhob$, $\rbar$), it was recently proven that $\pi$ has finite length (\cite{BHHMS4}, \cite{lucrezia}). In general, $\pi$ could depend on more than $\rhob$ (for instance because of multiplicities that involve places outside of $p$). But under a multiplicity one assumption on $\pi$ (that can be satisfied in practice), $\rhob\mapsto \pi$ is a good candidate for a mod $p$ Langlands correspondence for $\GL_2(K)$, or at least would be if we knew the answer to the following crucial question.
\begin{ques1}\label{hard}
Does $\pi$ depend only on $\rhob$?
\end{ques1}
Question \ref{hard}, which is called the locality issue, is natural because everything that can be calculated about $\pi$ up till now (Serre weights, invariants under $1+p{\text M}_2(\cO_K)$, etc.)\ depends only on $\rhob$. But over the years it has proved to be extremely difficult: its answer is yes when $K=\Qp$ but so far there is no example of a single $\pi$ when $K\ne \Qp$ for which the answer is known. In this work, we do not answer Question \ref{hard} but we present a tentative approach to this locality issue (\S~\ref{perfII}), as well as two failed attempts so far (\S~\ref{perf0}, \S~\ref{perfI}) but whose intermediate results may be of use in the future.

Let us temporarily return to $K=\Qp$ and assume that $\rhobar:\Gal(\overline{\Qp}/\Qp)\rightarrow \GL_2(\F)$ is irreducible, or equivalently $\pi$ is supersingular, which is the most interesting case (in this introduction $\F$ is a sufficiently large finite extension of $\Fp$ so that all representations are over $\F$-vector spaces). Let $D(\rhobar)$ be Fontaine's cyclotomic $(\varphi,\Zp^\times)$-module over $\F\ppar \Zp\cong \F\ppar T$ associated to $\rhobar$ and recall there is a canonical surjection $\psi:D(\rhobar)\twoheadrightarrow D(\rhobar)$ with $\psi\circ \varphi = \id$. Colmez found in \cite[\S~IV.4]{Colmez} several ways to relate $D(\rhobar)$ to the restriction of $\pi$ to $P:=\smat{\Qp^\times & \Qp\\ 0 & 1}$ (up to twist, say). One way is an equivariant injection $\pi^\vee\hookrightarrow \vplim_{\psi}D(\rhob)$ (where $\pi^\vee$ is the $\F$-linear dual of $\pi$) which induces an isomorphism
\begin{equation}\label{eq:intro1}
\pi^\vee\vert_P\buildrel\sim\over\longrightarrow \vplim_{\psi}D(\rhob)^\natural,
\end{equation}
where $D(\rhob)^\natural$ is the unique minimal nonzero closed $\F\bbra{T}$-submodule of $D(\rhob)$ which is $\Zp^\times$-stable and on which $\psi$ is surjective. In particular $\pi^\vee\vert_P$ depends only on $\rhobar$ and with \cite{paskunas-restriction} it follows that $\pi$ depends only on $\rhobar$. Another way, which can also be deduced from the injection $\pi^\vee\hookrightarrow \vplim_{\psi}D(\rhob)$, is an equivariant surjection (again stated here up to twist):
\begin{equation}\label{eq:intro2}
D_\infty(\rhobar):=\F\ppar{T^{p^{-\infty}}}\otimes_{\F\ppar T}D(\rhob)\twoheadrightarrow \pi\vert_P.
\end{equation}
A few years ago we began investigating what happens to (\ref{eq:intro1}) and (\ref{eq:intro2}) when $\Qp$ is replaced by an unramified extension. Unsurprisingly we discovered a situation \emph{far more complicated} than for $\Qp$. In this work we present our findings~so~far.

Assume from now on that $K$ is unramified and let $f:=[K:\Qp]$, $q:=p^f$. The Iwasawa algebra $\F\bbra{\cO_K}$ {can be written} as $\F\bbra{Y_{\sigma},\ \sigma\!:\!\Fq\hookrightarrow \F}$ for $Y_{\sigma}:= \sum_{\lambda\in \Fq^\times}\sigma(\lambda)^{-1}[\lambda]\in \F\bbra{\cO_K}$, where $[\lambda]\in \cO_K\s$ is the multiplicative representative of $\lambda$ (seen in $\F\bbra{\cO_K}$). We define the ring $A$ as the completion of $\F\bbra{\cO_K}[1/Y_{\sigma},\ \sigma\!:\!\Fq\hookrightarrow \F]$ for the $(Y_{\sigma})_{\sigma}$-adic topology. Alternatively $A$ is isomorphic to the Tate algebra $\F\ppar{Y_{\sigma}}\langle (Y_{\sigma'}/Y_{\sigma})^{\pm 1},\sigma'\ne \sigma\rangle$ for any choice of $\sigma$. It is endowed with an $\F$-linear Frobenius $\varphi$ coming from the multiplication by $p$ on $\cO_K$ and with a commuting continuous action of $\cO_K^\times$ also coming from multiplication on $\cO_K$. An \'etale $(\varphi,\cO_K^\times)$-module over $A$ is by definition a finite free $A$-module $D$ endowed with a semi-linear Frobenius $\varphi$ whose image generates $D$ and a commuting continuous semi-linear action of $\cO_K^\times$. Any \'etale $(\varphi,\cO_K^\times)$-module $D$ is also endowed with a surjection $\psi:D\twoheadrightarrow D$ such that $\psi\circ \varphi=\id$.

Let $\rhob$ and $\pi$ be as in Question \ref{hard} (we always assume $\rhob$ sufficiently generic). In \cite{BHHMS3} we associated to $\rhobar$ an \'etale $(\varphi,\cO_K^\times)$-module $D_A^{\otimes}(\rhobar)$ of rank $2^f$ over $A$. In \cite{BHHMS2} we also defined
\begin{equation}\label{eq:intro3}
D_A(\pi):= A\wh\otimes_{\F\bbra{\cO_K}}(\pi^\vee),
\end{equation}
where the completion is for the tensor product filtration (with $\m_{I_1}$-adic filtration on $\pi^\vee$, where $I_1$ is the pro-$p$ Iwahori defined in \S~\ref{strengthening}). In \cite{BHHMS3} and \cite{YW2} it is shown by explicit computations that $D_A(\pi)$ and $D_A^{\otimes}(\rhobar)$ are isomorphic as \'etale $(\varphi,\cO_K^\times)$-modules (up to twist). Together with (\ref{eq:intro3}) we obtain a morphism $\pi^\vee\rightarrow D_A^{\otimes}(\rhobar)$ and denote by $D_A(\pi)^\natural$ its image, which is a nonzero $\cO_K^\times$-stable closed $\F\bbra{\cO_K}$-submodule of $D_A^{\otimes}(\rhobar)$ on which $\psi$ is surjective. In the rest of this introduction, in order to avoid principal series (which are not at all difficult to handle but which behave slightly differently from supersingular representations), we replace $\pi$ by its maximal subquotient which has only (finitely many) supersingular constituents (this subquotient exists by the main results of \cite{BHHMS4}). There is an explicit subquotient $D_A^{\otimes}(\rhobar)^\ss$ of $D_A^{\otimes}(\rhobar)$ such that $D_A(\pi)\cong D_A^{\otimes}(\rhobar)^\ss$ (see (\ref{form}), (\ref{piglobal})).

\begin{thm1}
Let $\rhob$, $\pi$ be as above and recall $P = \smatr{K\s}{K}{0}{1} \subset \GL_2(K)$.
\begin{enumerate}
\item
We have an isomorphism $\pi^\vee\vert_{P}\buildrel\sim\over\longrightarrow \vplim_{\psi}D_A(\pi)^\natural$.
\item
Assume $\pi$ is irreducible (which happens when $f=2$ or when $\rhobar$ is irreducible). Then $D_A(\pi)^\natural$ is a minimal closed $\F\bbra{\cO_K}$-submodule of $D_A^{\otimes}(\rhobar)^\ss$ which is $[\Fq^\times]$-stable and on which $\psi$ is surjective. Moreover $\pi$ depends only on $\rhobar$ if and only if $D_A(\pi)^\natural$, as a $([\Fq^\times],\psi)$-module over $\F\bbra{\cO_K}$, depends only on $\rhobar$.
\end{enumerate}
\end{thm1}

Part (i) is easy (Corollary \ref{limpsiiso}) and part (ii) follows from part (i) and a strengthening of results of Pa{\v{s}}k{\=u}nas (Theorem \ref{localforQ} and Remark \ref{thm:Q-irred+}(ii)); see Corollary \ref{localfail1} and the paragraph below Remark \ref{rem:nonunique}. In fact, in most cases, using a further strengthening of \cite{paskunas-restriction} (Theorem \ref{thm:Q-irred}) one can even forget the $[\Fq^\times]$-action and prove that $D_A(\pi)^\natural$ is a minimal closed $\F\bbra{\cO_K}$-submodule of $D_A^{\otimes}(\rhobar)^\ss$ on which $\psi$ is surjective (Corollary \ref{naturalirr}). Even better, at least when $f=2$ and $\rhobar$ is reducible, we do not need the surjectivity of $\psi$, that is, $D_A(\pi)^\natural$ is a minimal closed $\F\bbra{\cO_K}$-submodule of $D_A^{\otimes}(\rhobar)^\ss$ preserved by $\psi$ (Lemma \ref{stronger}).

Since $D_A^{\otimes}(\rhobar)^\ss$ is very simple when $f=2$ and $\rhobar$ is reducible (see (\ref{ssummand}) with (\ref{piglobal})), and $\psi$ is also not very complicated (see (\ref{psi}) with (\ref{basechange})), we were hoping to find ways to identify $D_A(\pi)^\natural$ inside $D_A^{\otimes}(\rhobar)^\ss$ in that case, even though we knew that $D_A(\pi)^\natural$ is not of finite type over $\F\bbra{\cO_K}$ (Proposition \ref{notfinite}). However, we could not find \emph{one single nonzero} element in $D_A(\pi)^\natural$ (see Lemma \ref{flippant}). Also, contrary to the case $f=1$, there can exist several minimal nonzero closed $\F\bbra{\cO_K}$-submodules of $D_A^{\otimes}(\rhobar)^\ss$ which are $\cO_K^\times$-stable on which $\psi$ is surjective (we give examples in Remark \ref{rem:nonunique} but we suspect this could be a general phenomenon).

As it seemed impossible to directly recognize $D_A(\pi)^\natural$, we tried a perfectoid approach (\S~\ref{perfI}). When $f=1$, $D_\infty(\rhobar)$ in (\ref{eq:intro2}) is a dense subspace of $\vplim_{\psi}D(\rhob)$ and the intersection $X:=D_\infty(\rhobar)\cap \vplim_{\psi}D(\rhob)^\natural$ is again dense in $\vplim_{\psi}D(\rhob)^\natural\cong \pi^\vee\vert_P$ (using (\ref{eq:intro1})). That is, $\pi^\vee\vert_P$ can also be described as the closure of $X$ in $\vplim_{\psi}D(\rhob)$. When $f>1$, though we do not know what $D_A(\pi)^\natural$ is, there is a possible geometric interpretation of (the analog of) $X$ in terms of perfectoid geometry which depends only on $\rhob$~as~we~explain~now.

Replacing $A$ by its perfectoid version $A_\infty:= \F\ppar{Y_{\sigma}^{1/p^{\infty}}}\left\langle (Y_{\sigma'}/Y_{\sigma})^{\pm {1/p^{\infty}}},\sigma'\ne \sigma\right\rangle$ (a perfectoid Tate algebra over the perfectoid field $\F\ppar{Y_{\sigma}^{1/p^{\infty}}}$ for any $\sigma$) we set $D_{A_\infty}^{\otimes}(\rhobar):=A_\infty \otimes_A D_A^{\otimes}(\rhobar)$ and likewise for $D_{A_\infty}^{\otimes}(\rhobar)^\ss$. One easily checks there is a canonical $P$-equivariant injection $D_{A_\infty}^{\otimes}(\rhobar)^\ss\hookrightarrow \vplim_{\psi}D_A^{\otimes}(\rhobar)^\ss$ (see Proposition \ref{embedding}). Just as for $f=1$ we define $X:=D_{A_\infty}^{\otimes}(\rhobar)^\ss\cap \vplim_{\psi}D_A(\pi)^\natural$ (see (\ref{square})) which is an $\F\bbra{\smat{1 & K\\ 0 & 1}}\cong \F\bbra{Y_\sigma^{p^{-\infty}},\sigma}$-submodule of $D_{A_\infty}^{\otimes}(\rhobar)^\ss$ preserved by $P$ (Lemma \ref{lem:action-extends}).

Fix an embedding $\sigma_0:\Fq\hookrightarrow \F$ (the choice of which does not matter) and set $\sigma_i:= \sigma_0\circ\varphi^i$ for $i\in \Z$. As in \cite{BHHMS3} we consider the two perfectoid spaces 
\begin{eqnarray*}
Z_{\LT}&:= &\underbrace{\Spa\big(\F\ppar{T_{\sigma_0}^{1/p^\infty}},\F\bbra{T_{\sigma_0}^{1/p^\infty}}\big)\times_{\Spa(\F)}\cdots \times_{\Spa(\F)}\Spa\big(\F\ppar{T_{\sigma_0}^{1/p^\infty}},\F\bbra{T_{\sigma_0}^{1/p^\infty}}\big)}_{f\mathrm{\ times}}\\
Z_{\cO_K}&:= &\Spa\Big(\F\bbra{Y_{\sigma_0}^{1/p^\infty},\dots,Y_{\sigma_{f-1}}^{1/p^\infty}},\F\bbra{Y_{\sigma_0}^{1/p^\infty},\dots,Y_{\sigma_{f-1}}^{1/p^\infty}}\Big)\setminus V(Y_{\sigma_0},\dots, Y_{\sigma_{f-1}}),
\end{eqnarray*}
where $T_{\sigma_0}$ is the variable of the Lubin--Tate formal group over $\ok$ associated to the uniformizer $p$ (we use the embedding $\sigma_0$ to work with $\F$-coefficients). The perfectoid space $Z_{\LT}$ is quasi-Stein (cf.~\S~\ref{sec:space-zlt}), while the perfectoid space $Z_{\cO_K}$ is quasi-compact (cf.~\S~\ref{geometricinter}). Moreover $Z_{\LT}$ is endowed with an action of $(K^\times)^f\rtimes{S}_f$, $Z_{\cO_K}$ is endowed with an action of $K^\times$ and there is a canonical quasipro\'etale surjection of perfectoid spaces $m:Z_{\LT}\longrightarrow Z_{\cO_K}$ such that $m \circ ((a_0,\dots,a_{f-1}),w) = (\prod_i a_i) \circ m$ for $a_i\in K^\times$ and $w\in {S}_f$. (Here $S_f$ is the permutation group in $f$ variables.) Recall also that $\Spa(A_\infty, A_\infty^\circ)$ is an affinoid open subspace of $Z_{\cO_K}$ and that $Z_{\LT}=Z_{\cO_K}=\Spa(A_\infty, A_\infty^\circ)$ if (and only if) $f=1$.

The perfectoid Lubin--Tate $(\varphi_q,\oks)$-module associated to $\rhobar$ (and $\sigma_0$) can be seen as \ the \ space \ of \ global \ sections \ of \ a \ $K^\times$-equivariant \ vector \ bundle \ $\mathcal{V}_{\rhobar}$ \ on \ $\Spa(\F\ppar{T_{\sigma_0}^{1/p^\infty}},\F\bbra{T_{\sigma_0}^{1/p^\infty}})$. We define the following $(K^\times)^f\rtimes{S}_f$-equivariant vector bundle on $Z_{\LT}$:
\[\cF_{\rhobar} := \mathcal{V}_{\rhob}^{(0)}\otimes_{\cO_{Z_{\LT}}}\mathcal{V}_{\rhob}^{(1)}\otimes_{\cO_{Z_{\LT}}}\cdots\otimes_{\cO_{Z_{\LT}}}\mathcal{V}_{\rhob}^{(f-1)},\]
where $\mathcal{V}_{\rhobar}^{(i)}:= \pr_i^*\mathcal{V}_{\rhobar}$ for $i\in \{0,\dots, f-1\}$ with $\pr_i:Z_{\LT}\rightarrow \Spa(\F\ppar{T_{\sigma_0}^{1/p^\infty}},\F\bbra{T_{\sigma_0}^{1/p^\infty}})$ the $i$-th projection. Let $G:= \{(a_0,\dots,a_{f-1})\in (K^\times)^f,\ \prod a_i =1\}{\rtimes S_f}$, then $(m_*\cF_{\rhobar})^G$ is a $K^\times$-equivariant sheaf of $\cO_{Z_{\cO_K}}$\!-modules~on~$Z_{\cO_K}$ and we have by \cite[\S~2.7]{BHHMS3} $H^0(\Spa(A_\infty, A_\infty^\circ),(m_*\cF_{\rhobar})^G)\cong D_{A_\infty}^{\otimes}(\rhobar)$. We give in Definition~\ref{cF} a $(K^\times)^f\rtimes{S}_f$-equivariant subquotient $\cF_{\rhobar}^\ss$ of $\cF_{\rhobar}$ such that 
\[H^0(\Spa(A_\infty, A_\infty^\circ),(m_*\cF_{\rhobar}^\ss)^G)\cong D_{A_\infty}^{\otimes}(\rhobar)^\ss.\]

The sheaf $(m_*\cF_{\rhobar}^\ss)^G$ is very difficult to study, as the map $m$ is not a $G$-torsor when $f>1$ (we strongly suspect $(m_*\cF_{\rhobar}^\ss)^G$ is not locally free in general). However, if we replace the $2$-dimensional $\rhobar$ by a $1$-dimensional character $\chi$ of $\Gal(\overline{K}/K)$, then $\cF _\chi$ (defined similarly as $\cF_{\rhobar}$) is isomorphic to $\cO_{Z_{\LT}}$ as a $G$-equivariant vector bundle, and one can check that $(m_*\cO_{Z_{\LT}})^G\cong \cO_{Z_{\cO_K}}$ (Lemma \ref{FN0}). Using $H^0(Z_{\cO_K},\cO_{Z_{\cO_K}})\cong \F\bbra{Y_\sigma^{p^{-\infty}},\sigma}$ when $f>1$ (Proposition \ref{globalsections}, note that this is wrong when $f=1$!), one easily checks in that case $H^0(Z_{\cO_K},(m_*\mathcal{F}_{\chi})^G)\cong X,$
where $X$ is defined as previously replacing $\pi$ by a certain principal series (see \S~\ref{sec:princ-seri-case}). In general $H^0(Z_{\cO_K},(m_*\cF_{\rhobar}^\ss)^G)$ is an $\F\bbra{Y_\sigma^{p^{-\infty}},\sigma}$-submodule of $D_{A_\infty}^{\otimes}(\rhobar)^\ss$ preserved by $P$. The above case of dimension $1$ makes it a natural candidate for the ``possible geometric interpretation of $X$ depending only on $\rhob$'' alluded to above:

\begin{ques1}[Question \ref{qu:intersec}, Question \ref{density?}]\label{hard2}
Do we have $H^0(Z_{\cO_K},(m_*\cF_{\rhobar}^\ss)^G)\cong X$ inside $D_{A_\infty}^{\otimes}(\rhobar)^\ss$? If so, is this submodule dense in $\varprojlim_\psi \!\!D_A(\pi)^\natural\cong \pi^\vee|_P$?
\end{ques1}

A positive answer to Question \ref{hard2} would be enough to imply that $\pi$ depends only on $\rhobar$. Unfortunately, despite much effort and energy (including the use of computers), we are still unable to even decide whether $X$ or $H^0(Z_{\cO_K},(m_*\cF_{\rhobar}^\ss)^G)$ is zero or nonzero (let alone if they are equal). Note that Fargues conjectures (in a broader context) that $H^0(Z_{\cO_K},(m_*\cF_{\rhobar}^\ss)^G)\ne 0$, see Remark \ref{fargues}.

Not being successful (so far) with the analog of (\ref{eq:intro1}) for $K$, we tried the analog of (\ref{eq:intro2}) (\S~\ref{perfII}). Indeed, it is not difficult to show that there is an equivariant surjection (up to twist) $D_{A_\infty}^{\otimes}(\rhobar)^\ss\twoheadrightarrow \pi\vert_{P}$, see \S~\ref{perfII} (in particular (\ref{surj})). But this still does not prove that $\pi\vert_P$ depends only on $\rhobar$. For that we would again need a geometric interpretation of the surjection $D_{A_\infty}^{\otimes}(\rhobar)^\ss\twoheadrightarrow \pi\vert_P$ depending only on $\rhob$. We were able to discern a candidate when $f=2$ at least, which we explain now.

One can rewrite $H^0(Z_{\cO_K},(m_*\cF_{\rhobar}^\ss)^G)\cong H^0(G, H^0(Z_{\LT},\cF_{\rhobar}^\ss))$, hence it is natural to consider the continuous group cohomology $H^1\cts(G, H^0(Z_{\LT},\cF_{\rhobar}^\ss))$, where $H^0(Z_{\LT},\cF_{\rhobar}^\ss)$ is endowed with its natural Fr\'echet topology. We assume from now on $f=2$. In that case $Z_{\cO_K}=U_0\cup U_1$, where $U_0$, $U_1$ are affinoid perfectoid open subsets defined by $U_0 = \{ |Y_{\sigma_1}| \le |Y_{\sigma_0}| \ne 0 \}$, $U_1 = \{ |Y_{\sigma_0}| \le |Y_{\sigma_1}| \ne 0 \}$ (see (\ref{eq:U_i}) for a more explicit description). We have $U_0 \cap U_1 = \Spa(A_\infty, A_\infty^\circ)$ (see the proof of Lemma \ref{hartogs}). Combining the Mayer--Vietoris sequence for the covering $Z_{\cO_K}=U_0\cup U_1$ with the long exact sequence of continuous group cohomology for $H^0$ and $H^1$, we obtain the following two exact sequences of $P$-representations (see (\ref{connecting}) and (\ref{IR})):
\begin{multline}\label{eq:intro5}
0 \longrightarrow H^0(\zok, (m_* \cF_{\rhobar}^\ss)^G) \longrightarrow H^0(U_0, (m_* \cF_{\rhobar}^\ss)^G) \oplus H^0(U_1, (m_* \cF_{\rhobar}^\ss)^G)\\
\longrightarrow D_{A_\infty}^{\otimes}(\rhobar)^\ss \buildrel \delta \over \longrightarrow H^1_{\mathrm{cts}}(G,H^0(\zlt,\cF_{\rhobar}^\ss))
\end{multline}
\begin{multline}\label{eq:intro6}
0 \longrightarrow H^1_{\mathrm{cts}}\big(G/(p,p^{-1})^{\Z}, H^0(\zlt,\cF_{\rhobar}^\ss)^{(p,p^{-1})^{\Z}}\big)\longrightarrow H^1_{\mathrm{cts}}(G, H^0(\zlt,\cF_{\rhobar}^\ss))\\
\buildrel \res \over \longrightarrow H^1\big((p,p^{-1})^{\Z},H^0(\zlt,\cF_{\rhobar}^\ss)\big)^G.
\end{multline}
By an explicit calculation we first prove (see \S~\ref{sec:smoothness-result}):

\begin{prop1}\label{intro:smooth}
The action of $P$ on $H^1\big((p,p^{-1})^{\Z},H^0(\zlt,\cF_{\rhobar}^\ss)\big)^G$ is smooth.
\end{prop1}

As we look for a $P$-smooth quotient of $D_{A_\infty}^{\otimes}(\rhobar)^\ss$, it is therefore natural to consider the composition
\begin{equation}\label{eq:intro4}
D_{A _\infty}^\otimes(\rhob)^{\ss} \xrightarrow{\ \delta\ } H^1_{\mathrm{cts}}(G,H^0(\zlt,\cF_{\rhobar}^\ss)) \xrightarrow{\ \res\ } H^1\big((p,p^{-1})^{\Z},H^0(\zlt,\cF_{\rhobar}^\ss)\big)^G.
\end{equation}

The following is our main result, see Theorem \ref{main} and Corollary \ref{cor:main}:

\begin{thm1}\label{mainintro}
The image of $D_{A _\infty}^\otimes(\rhob)^{\ss}$ in the composition (\ref{eq:intro4}) is an infinite-di\-men\-sion\-al smooth representation of $P$ over $\F$ which is $Y_{\sigma_i}$-divisible for $i\in \{0,1\}$, in particular which has unbounded $(Y_{\sigma_0},Y_{\sigma_1})$-torsion.
\end{thm1}

We hope that the image of $D_{A _\infty}^\otimes(\rhob)^{\ss}$ via (\ref{eq:intro4}) is isomorphic to $\pi\vert_P$ (up to twist), though we presently do not know how to attack this question. Proposition~\ref{intro:smooth} and Theorem~\ref{mainintro} in fact hold for \emph{any} continuous representation $\rhob:\Gal(\overline{K}/K)\rightarrow \GL_2(\F)$ {with} $f=2$ without need of any global context, representation $\pi$ or genericity condition (other than $p>2$). Let us sketch the two main steps in the proof of Theorem~\ref{mainintro}.

The first main step is that the map $H^0(U_0, (m_* \cF_{\rhobar}^\ss)^G) \oplus H^0(U_1, (m_* \cF_{\rhobar}^\ss)^G)
\rightarrow D_{A_\infty}^{\otimes}(\rhobar)^\ss$ in (\ref{eq:intro5}) cannot be an isomorphism. This is in fact true with $\cF_{\rhobar}^\ss$ replaced by any {nonzero} $G$-equivariant vector bundle on $\zlt$, see Theorem \ref{AIM23}. The proof goes as follows. Assuming the above map is an isomorphism, one first proves that each $H^0(U_i, (m_* \cF_{\rhobar}^\ss)^G)$ is a nonzero closed $A_\infty$-submodule of $D_{A_\infty}^{\otimes}(\rhobar)^\ss$. For $\alpha\in \Q$, $\alpha<1$ sufficiently close to $1$ we define $V_\alpha := \{|Y_{\sigma_0}| = |Y_{\sigma_1}|^\alpha \ne 0 \}\subset U_0$ and prove that $m^{-1}(V_\alpha)\rightarrow V_\alpha$ is again a $G$-torsor, so that $H^0(V_\alpha, (m_* \cF_{\rhobar}^\ss)^G)$ is a nonzero locally free $H^0(V_\alpha, \cO_{V_\alpha})$-module. Then, taking moreover $\alpha^{-1} \in \Z[1/p]$, we prove that the restriction map $H^0(U_0,(m_* \cF_{\rhobar}^\ss)^G)\rightarrow H^0(V_\alpha,(m_* \cF_{\rhobar}^\ss)^G)$ is both injective and zero, yielding a contradiction since $H^0(U_0,(m_* \cF_{\rhobar}^\ss)^G)\ne 0$ under our assumption. The injectivity essentially comes from the inclusion $V_\alpha\subset U_0$ (though we also use the injection $H^0(U_0, (m_* \cF_{\rhobar}^\ss)^G)\hookrightarrow H^0(U_0\cap U_1, (m_* \cF_{\rhobar}^\ss)^G)\cong D_{A_\infty}^{\otimes}(\rhobar)^\ss$). The fact the map is $0$ is more subtle and comes from the $A_\infty$-structure on $H^0(U_0, (m_* \cF_{\rhobar}^\ss)^G)$ and the definition of $V_\alpha$ which force any element of $H^0(U_0,(m_* \cF_{\rhobar}^\ss)^G)$ to have image $0$ in $H^0(V_\alpha,(m_* \cF_{\rhobar}^\ss)^G)$.

The second main step crucially uses the fact that the $H^0(Z_{\cO_K},\cO_{Z_{\cO_K}})$-modules $H^0(\zok, (m_* \cF_{\rhobar}^\ss)^G)$ and $H^1_{\mathrm{cts}}(G/(p,p^{-1})^{\Z}, H^0(\zlt,\cF_{\rhobar}^\ss)^{(p,p^{-1})^{\Z}})$ are derived complete with respect to the ideal $I = (Y_{\sigma_0},Y_{\sigma_1})$ of $\F\bbra{Y_{\sigma_0}^{p^{-\infty}},Y_{\sigma_1}^{p^{-\infty}}}\cong H^0(Z_{\cO_K},\cO_{Z_{\cO_K}})$, a property which is well-behaved with respect to exact sequences. Now, assume that the composition (\ref{eq:intro4}) is $0$, then from (\ref{eq:intro6}) we get an $H^0(Z_{\cO_K},\cO_{Z_{\cO_K}})$-linear map $D_{A _\infty}^\otimes(\rhob)^{\ss}\rightarrow H^1_{\mathrm{cts}}(G/(p,p^{-1})^{\Z}, H^0(\zlt,\cF_{\rhobar}^\ss)^{(p,p^{-1})^{\Z}})$. As $Y_{\sigma_0}$ (or $Y_{\sigma_1}$) is invertible on the source and the target is derived complete with respect to $I$, this map is also $0$. Hence by (\ref{eq:intro5}) we have a short exact sequence of $H^0(Z_{\cO_K},\cO_{Z_{\cO_K}})$-modules
\[0\rightarrow H^0(\zok, (m_* \cF_{\rhobar}^\ss)^G) \rightarrow H^0(U_0, (m_* \cF_{\rhobar}^\ss)^G) \oplus H^0(U_1, (m_* \cF_{\rhobar}^\ss)^G)\rightarrow D_{A_\infty}^{\otimes}(\rhobar)^\ss\rightarrow 0,\]
which implies $H^0(\zok, (m_* \cF_{\rhobar}^\ss)^G)\ne 0$ by the first main step. Since $Y_{\sigma_0}$ is invertible on $D_{A_\infty}^{\otimes}(\rhobar)^\ss$ and since $H^0(\zok, (m_* \cF_{\rhobar}^\ss)^G)$ is derived complete with respect to $I$, this exact sequence splits, i.e.,~there are $H^0(Z_{\cO_K},\cO_{Z_{\cO_K}})$-linear surjections $H^0(U_i, (m_* \cF_{\rhobar}^\ss)^G)\twoheadrightarrow H^0(\zok, (m_* \cF_{\rhobar}^\ss)^G)$ for $i=0,1$. Since $H^0(\zok, (m_* \cF_{\rhobar}^\ss)^G)$ is derived complete these surjections must again be $0$, contradiction. Hence the composition (\ref{eq:intro4}) is nonzero. The fact $D_{A_\infty}^{\otimes}(\rhobar)^\ss$ is $Y_{\sigma_i}$-divisible for $i\in \{0,1\}$ then implies Theorem \ref{mainintro}.

As a byproduct of the proof of Theorem \ref{mainintro} we could at least prove (see Corollary \ref{non0bis} and recall we do not know if $H^0(\zok, (m_* \cF_{\rhobar}^\ss)^G)$ is $0$ or not):

\begin{cor1}
Assuming $f=2$, at least one of the spaces $H^0(\zok, (m_* \cF_{\rhobar}^\ss)^G)$, $H^1_{\mathrm{cts}}(G/(p,p^{-1})^{\Z}, H^0(\zlt,\cF_{\rhobar}^\ss)^{(p,p^{-1})^{\Z}})$ is nonzero.
\end{cor1}

Finally, in \S\S~\ref{sec:aside-trivial-galois}, \ref{sec:princ-seri-case2} we check what happens to (\ref{eq:intro5}), (\ref{eq:intro6}) and the image of (\ref{eq:intro4}) when $\cF_{\rhobar}^\ss$ is replaced by $\cF _\chi$ for a $1$-dimensional character $\chi$ of $\Gal(\overline{K}/K)$ (recall that $(m_*\cF _\chi)^G\cong \cO_{Z_{\cO_K}}$). In that case everything is (so!)~much easier: we have an isomorphism $H^1_{\mathrm{cts}}(G,\!H^0(\zlt,\cF _\chi))\cong H^1(\zok,\cO_{\zok})$ (Proposition \ref{avoir} in the appendix), from which we deduce that the $P$-representation $H^1_{\mathrm{cts}}(G,\!H^0(\zlt,\cF _\chi))$ is irreducible smooth, the map $\delta$ in (\ref{eq:intro5}) is surjective and $H^1_{\mathrm{cts}}(G/(p,p^{-1})^{\Z}, H^0(\zlt,\cF_{\chi})^{(p,p^{-1})^{\Z}})=0$ (Corollary \ref{irrtriv}, Proposition \ref{nonzero}). In \S~\ref{sec:princ-seri-case2} we then check that the image of (\ref{eq:intro4}) is a codimension one $P$-equivariant subspace of a principal series, just as expected.

\textbf{Acknowledgements}: We started to seriously think about the locality issue around the early months of 2022 while writing \cite{BHHMS3}. The geometric interpretation of $D_{A_\infty}^\otimes(\rhobar)$ in \emph{loc.~cit.}~made us optimistic at the time, but month after month we realized that, once again, this issue was far more complicated than expected. Since 2022 we have discussed our problems with many people, starting with Laurent Fargues. We would like to thank him, and also all the people who patiently listened to (some of) us: Johannes Ansch\"utz, Juan Esteban Rodr\'\i guez Camargo, Andrea Dotto, Youshua Kesting, Arthur-C\'esar Le Bras, Jo\~ao Louren\c co, Nataniel Marquis, Vytautas Pa{\v{s}}k{\=u}nas, Matteo Tamiozzo, \dots\ In particular we would like to thank J.~E. Rodr\'\i guez Camargo for suggesting that derived completeness {might} be useful. We finally thank Mohammed Boubakeur and Xavier Caruso for their help in our computer-assisted attempts.

F.\;H.\ is partially supported by an NSERC grant, a grant from Simons Foundation International [SFI-MPS-SFM-00006210, FH], and a Visiting Professorship at Korea Institute for Advanced Study. 
Y. \;H.\ is \ partially \ supported by CAS Project for Young Scientists in Basic Research (Grant No.\ YSBR-033), National Natural Science Foundation of China Grants 12288201 and 12425103, National Center for Mathematics and Interdisciplinary Sciences and Hua Loo-Keng Key Laboratory of Mathematics, Chinese Academy of Sciences.
K.\;K.\ is partially supported by NSF grant DMS-2310225 and a PSC-CUNY Trad B award. 
C.\;B.\,, S.\;M.\ and B.\;S.\ are partially supported by the ANR-25-CE40-4664,
S.\;M.\ and B.\;S.\ are partially supported by the Institut Universitaire de France, and S.\;M.\ is partially supported by the ANR--MSREI OPE-2025-0369. 
S.\;W.\;S.\ is partially supported by NSF grant DMS-2401353, NSF RTG grant DMS-2342225, and a Simons Travel Grant.
F.\;H.\ thanks the Korea Institute for Advanced Study and Universit\'e Paris--Saclay for the excellent working conditions provided.
Finally we thank the American Institute of Mathematics for hosting and supporting our SQuaRE meeting in 2023.

\section{Trying to directly prove the locality of \texorpdfstring{$\pi$}{pi}}\label{perf0}

Let $\pi$ be a finite length admissible supersingular representation of $\GL_2(K)$ over $\F$ of global origin with only supersingular constituents. We explain how difficult it is to determine the image of $\pi^\vee$ in the $(\psi,\oks)$-module $D_A(\pi)$ defined in (\ref{eq:intro3}) as soon as $f\geq 2$, even when $f=2$, though $D_A(\pi)$ itself is explicit.

\subsection{A strengthening of results of \texorpdfstring{Pa{\v{s}}k{\=u}nas}{Paskunas}}\label{strengthening}

We assume $f\geq 1$ and strengthen some of the results of \cite{paskunas-restriction}.

We consider the following subgroup of $\GL_2(K)$:
\[Q:=\matr{p^{\Z}}{K}0{p^{\Z}}.\]
We shall study the restriction to $Q$ of an irreducible admissible supersingular representation of $\GL_2(K)$ over $\F$, generalizing the results of \cite{paskunas-restriction} which treat the restriction to the Borel subgroup $B$ (containing $Q$). 

We let $I:=\smatr{\cO_K^\times}{\cO_K}{p\cO_K}{\cO_K^\times}$ the Iwahori subgroup, $I_1:=\smatr{1+p\cO_K}{\cO_K}{p\cO_K}{1+p\cO_K}\!\subset I$ the pro-$p$ Iwahori subgroup and $H:=\smatr{[\F_q^{\times}]}00{[\F_q^{\times}]}$ which is isomorphic to $I/I_1$. We have the following decomposition 
\begin{equation}\label{eq:Q-decomp}
\GL_2(K)=QI\sqcup Q\smatr{0}11{0}I.
\end{equation}
Using $B=QT_0$, where $T_0:=\smatr{\cO_K^\times}00{\cO_K^\times}$, this follows from the more familiar formula $\GL_2(K)=BI\sqcup B\smatr0110 I$.
 
Let $Z_1$ be the center of $I_1$ and let $\mathfrak{m}_{I_1}$ (resp.~$\mathfrak{m}_{I_1/Z_1}$, $\mathfrak{m}_{N_0}$) be the maximal ideal of the local Iwasawa algebra $\F\bbra{I_1}$ (resp.~$\F\bbra{I_1/Z_1}$, $\F\bbra{N_0}\cong \F\bbra{\cO_K}$). If $\pi$ is a smooth representation of $\GL_2(K)$ over $\F$ with a central character, define an increasing filtration
$F_n\pi:=\pi[\mathfrak{m}_{I_1/Z_1}^{n+1}]$ for all $n\geq -1$, with the convention that $F_{-1}\pi=\{0\}$. If $v\in\pi$, define
\[\deg(v):=\mathrm{min}\{n\geq -1:~v\in F_n\pi\}\in\Z_{\geq -1}.\]
Hence $\deg(v)=-1$ if and only if $v=0$ and $\deg(v)=0$ if and only if $v\in \pi^{I_1}\backslash\{0\}$. We can also restrict $\pi$ to $N_0:=\smatr{1}{\cO_K}01$ and define similarly
\[\deg'(v):=\min\{n\geq -1: v\in \pi[\mathfrak{m}_{N_0}^{n+1}]\}.\]
It is clear that $\deg(v)\geq \deg'(v)$ and $\deg(v)=0\Rightarrow \deg'(v)=0$. 
 
We fix an embedding $\sigma_0:\F_q\hookrightarrow \F$ and use $\F\bbra{N_0}\cong \F\bbra{Y_0,Y_1,\dots,Y_{f-1}}$, where $Y_j:=Y_{\sigma_0\circ \varphi^j}$ with $Y_\sigma:=\sum_{\lambda\in \F_q^\times}\sigma(\lambda)^{-1}\smatr1{[\lambda]}01$ for $\sigma:\F_q\hookrightarrow \F$. For $\underline{k}\in\Z_{\geq 0}^f$ we write
\begin{equation}\label{Yk}
\|\underline{k}\|:=\sum_{i=0}^{f-1}k_i,\ \ \ \underline{Y}^{\underline{k}}:=\prod_{i=0}^{f-1}Y_i^{k_i}.
\end{equation}
For $0\leq j\leq f-1$ we let $Z_j :=\sum_{\lambda\in \F_q^\times}\sigma_0(\lambda)^{-p^j}\smatr10{p[\lambda]}1 \in \mathfrak{m}_{I_1/Z_1}$ (do not confuse $Z_1$ here with the center of $I_1$!), and let $y_j$, $z_j \in \gr(\F\bbra{I_1/Z_1})$ denote the reductions of $Y_j$, $Z_j$ in degree $-1$ in the associated graded ring $\gr(\F\bbra{I_1/Z_1})$ with respect to the filtration $(\mathfrak{m}_{I_1/Z_1}^{n})_{n\geq 0}$. Incidentally, we recall that $\gr(\F\bbra{I_1/Z_1})$ is the universal enveloping algebra of an $\F$-Lie algebra with basis $y_j, z_j, h_j := [y_j,z_j]$ ($0 \le j \le f-1$), where $h_j$ is central, cf.\ \cite[\S~5.3]{BHHMS1}. Let $J_0$ denote the (2-sided) ideal of $\gr(\F\bbra{I_1/Z_1})$ generated by all $h_j$, so that $\gr(\F\bbra{I_1/Z_1})/J_0$ is a commutative polynomial algebra in the $y_j$, $z_j$ ($0 \le j \le f-1$). For later reference we also recall the ideal $J := (h_j,y_jz_j : 0 \le j \le f-1)$ of $\gr(\F\bbra{I_1/Z_1})$ that contains $J_0$ and such that $\gr(\F\bbra{I_1/Z_1})/J$ is isomorphic to the commutative ring $\F[y_j,z_j]/(y_jz_j; 0\leq j\leq f-1)$. Note also that $\gr(\F\bbra{N_0})\cong \F[y_j; 0\leq j\leq f-1]$ with $Y_j\in \mathfrak{m}_{N_0}$ lifting $y_j$.

We let $\gr(\pi)$ denote the graded {$\gr(\F\bbra{I_1/Z_1})$-}module with respect to the filtration $F_n\pi$.
By one of the main results of \cite{BHHMS1}, in a global setting and under some hypotheses, $\gr(\pi)$ factors as a module over $\gr(\F\bbra{I_1/Z_1})/J\cong \F[y_j,z_j]/(y_jz_j; 0\leq j\leq f-1)$. The following result is the key lemma for Theorem \ref{thm:Q-irred}. 

\begin{lem}\label{lem:key}
Let $\pi$ be a smooth representation of $\GL_2(K)$ over $\F$ with a central character such that $\gr(\pi)$ is annihilated by the 2-sided ideal $J_0$. Let $v\in \pi\backslash\{0\}$ be an element fixed by $\smatr10{p\cO_K}1$. Then there exists $\underline{k}\in\Z_{\geq 0}^f$ with $\|\underline{k}\|=\deg(v)$ such that $0\neq \underline{Y}^{\underline{k}}v\in \pi^{I_1}$.
Moreover, we have 
\[\deg(v)=\deg'(v).\]
\end{lem}
\begin{proof}
For $w\in \pi$ let $\gr(w)\in \gr(\pi)$ be the image of $w$ in $\pi[\m_{I_1/Z_1}^{\deg(w)+1}]/\pi[\m_{I_1/Z_1}^{\deg(w)}]$ if $w\ne 0$ and $\gr(w):=0$ if $w=0$. The proof is as in \cite[Prop.~3.5.1]{BHHMS3}. In \emph{loc.~cit.}~the following fact is proved: if $w\in\pi$ is not fixed by $I_1$ and if $\gr(w)\in \gr(\pi)$ is annihilated by all variables $z_j$, then there exists $i\in \{0,\dots, f-1\}$ such that $y_i\gr(w)\neq0$ in $\gr(\pi)$. (This does not use that $\gr(\pi)$ is annihilated by the ideal $J$, only that $\m_{I_1/Z_1}$ is generated by all $Y_j$ and $Z_j$.) As a consequence, $Y_iw\neq 0$ in $\pi$, $\gr(Y_iw)=y_i\gr(w)$ in $\gr(\pi)$ and
\begin{equation*}
\deg(Y_iw)=\deg(w)-1.
\end{equation*}
Now for $v$ as in the statement we first show that there exists $\underline{k}\in\Z_{\geq 0}^f$ with $\|\underline{k}\|=\deg(v)$ such that $\underline{Y}^{\underline{k}}v\in \pi^{I_1}\setminus \{0\}$. If $v\in \pi^{I_1}$ we are done. Otherwise the assumption on $v$ implies $z_j\gr(v)=0$ for all $j$, and as above we get $Y_iv\neq 0$ and $\deg(Y_iv) = \deg(v)-1$ for some $i\in \{0,\dots, f-1\}$. If $Y_i v\in \pi^{I_1}$ we are done. Otherwise, as $\gr(v)$ is killed by all $z_j$ and $\gr(\F\bbra{I_1/Z_1})/J_0$ is commutative, $y_i\gr(v)$ is also killed by all $z_j$, hence we may repeat the same argument replacing $v$ by $Y_i v$. Inductively we obtain $\underline{k}\in\Z_{\geq0}^{f}$ such that $\underline{Y}^{\underline{k}} v\in \pi^{I_1}\backslash\{0\}$ and $\deg(\underline{Y}^{\underline{k}} v) = \deg(v)-\|\underline{k}\|$. Hence $\deg(v) = \|\underline{k}\|$. Moreover we have $\deg(v) \ge \deg'(v) \ge \|\underline{k}\| = \deg(v)$, where the first inequality is general and the second holds since $\underline{Y}^{\underline{k}} v\ne 0$.
\end{proof}

\begin{rem}
We originally proved Lemma \ref{lem:key} with $\gr(\pi)$ killed by $J$ instead of just $J_0$. The fact it holds with $J_0$ was suggested to us by Yitong Wang.
\end{rem}

Let $\Pi:=\smatr{0}1p0$ and for $0\leq i\leq q-1$ define the following operator
\[S_i:=\sum_{\lambda\in\F_q}\lambda^i\matr{p}{[\lambda]}01 = \sum_{\lambda\in\F_q}\lambda^i\matr{[\lambda]}110\Pi\in \F[Q].\]
Recall that $S_i$ is nicely compatible with the $H$-action:
\begin{equation}\label{eq:S_i-H-action}
\smatr {[a]}00{[b]} \circ S_i = (ba^{-1})^{i} \left(S_i \circ \smatr {[a]}00{[b]}\right) \quad\forall\ a,b \in \F_q^\times.
\end{equation}
For $n\geq 1$ we write $(S_i)^n$ for $S_i\circ \cdots \circ S_i $ ($n$ times). 

If $\pi$ is a smooth representation of $\GL_2(K)$ over $\F$ and $v\in \pi^{I_1}$, the set $\{\Pi v, S_iv:0\leq i\leq q-1\}$ spans the finite-dimensional $\GL_2(\cO_K)$-representation $\langle \GL_2(\cO_K)\cdot \Pi v\rangle$.
 Assuming that $v$ is an $I$-eigenvector with eigencharacter $\chi$, note that the latter is just the image of the induced morphism $\Ind_I^{\GL_2(\cO_K)}\chi^s\ra \pi|_{\GL_2(\cO_K)}$, where $\F\Pi v\cong \chi^s\hookrightarrow \pi|_I$ (the statement follows from \cite[Lemma 2.5(ii)]{BP}).
 More precisely, the elements $f_j$, $\phi$ of \cite[\S~2]{BP} are sent to $S_j v$, $\Pi v$, respectively. We observe that if $\chi\neq \chi^s$ then any \emph{proper} quotient $Q$ of $\Ind_I^{\GL_2(\cO_K)}\chi^s$ is multiplicity free as an $H$-representation, so $Q^{I_1}$ is spanned by $\{S_jv,~j\in T\}$ for some subset $T\subset\{0,\dots,q-1\}$.

\begin{lem}\label{lem:BP-irred}
Keep the above notation. Then the $\GL_2(\cO_K)$-representation $\tau:=\langle \GL_2(\cO_K)\cdot \Pi v\rangle$ is irreducible if and only if $\tau^{I_1}$ has dimension $1$.
\end{lem}
\begin{proof}
The condition is clearly necessary. Conversely, assume $\dim_{\F}\tau^{I_1}=1$. As $(\Ind_I^{\GL_2(\cO_K)}\chi^s)^{I_1}$ has dimension $2$, $\tau$ is a proper quotient of $\Ind_I^{\GL_2(\cO_K)}\chi^s$. If $\chi=\chi^s$, the statement follows from \cite[Lemma 2.6]{BP}. If $\chi\neq \chi^s$, then by \cite[Thm.~2.4]{BP} the cosocle of $\Ind_I^{\GL_2(\cO_K)}\chi^s$, hence of $\tau$, is irreducible; we denote it by $\sigma$. Using \cite[Lemma 2.7]{BP}, we see that the natural morphism $\tau^{I_1}\ra \sigma^{I_1}$ is always surjective, so that the following sequence 
\[0\ra \rad(\tau)^{I_1}\ra \tau^{I_1}\ra \sigma^{I_1}\ra0\]
is exact. Since $\dim_{\F}\tau^{I_1}=1$ by assumption, we must have $\rad(\tau)^{I_1}=0$, hence $\rad(\tau)=0$ and $\tau\cong \sigma$.
\end{proof}

We recall \cite[Lemme 2.9]{yongquan-jussieu}:

\begin{lem}\label{lem:Sn-vanish} 
Let $\pi$ be an irreducible supersingular representation of $\GL_2(K)$ over $\F$ with a central character.
Then for any $v\in \pi^{I_1}$ there exists an integer $n\gg0$ such that $(S_0)^nv=0$.
\end{lem}
 
For $0\leq r\leq q-1$, we write $r=\sum_{j=0}^{f-1}p^ir_i$ for its $p$-adic expansion. If $0\leq r,r'\leq q-1$, we write $r\preceq r'$ if $r_j\leq r_j$ for all $j$. We need to introduce the following (subtle) condition:
\[(\mathbf{C})\ \pi^{I_1}\text{ does not contain any pair of characters }\{\chi_1,\chi_2\}\text{ such that }\chi_2=\chi_1 {\det}^{(q-1)/2}.\]
Condition (C) will be satisfied in most applications.

\begin{lem}\label{lem:diff-character}
Let $\pi$ be an irreducible admissible supersingular representation of $\GL_2(K)$ over $\F$ satisfying (C). Let $v,v'\in \pi^{I_1}\backslash\{0\}$ on which $I$ acts via characters $\chi,\chi'$ respectively. Assume $\chi\neq \chi'$. Then there exist $k\geq 1$ and a sequence of integers $i_1,\dots,i_k\in\{0, 1,\dots,q-1\}$ such that precisely one of the vectors $(S_{i_k}\circ\cdots \circ S_{i_1})(v)$, $(S_{i_k}\circ\cdots \circ S_{i_1})(v')$ is nonzero.
\end{lem}
\begin{proof}
First note that if there exists $i\in \{0,\dots,q-1\}$ such that $S_iv=0$ but $S_iv'\neq0$ (resp.~$S_iv\neq0$ but $S_iv'=0$), then we are done. Thus we can assume 
\begin{equation}\label{eq:cond-equiv}
\forall\ 0\leq i\leq q-1,\ \ S_iv=0\Longleftrightarrow S_iv'=0.
\end{equation}
Note also that $S_iv\in\pi^{I_1}$ if and only if $S_{i'}v=0$ for any $i'\preceq i$, $i'\ne i$; {this follows from the explicit formula for $S_i$ (where $\mu\in\F_q^{\times}$):}
\[\smatr 1{-[\mu]}01 S_iv = \sum_{i' \preceq i} \binom i{i'} \mu^{i-i'} S_{i'} v.\]

\textbf{Step 1.} By Lemma \ref{lem:Sn-vanish}, there exists $n\geq0$ such that $(S_0)^{n+1}v=0$ and $(S_0)^{n+1}v'=0$. \ Choosing \ $n$ \ to \ be \ the \ smallest \ non-negative \ integer \ such that $(S_0)^{n + 1}v = 0$ or $(S_0)^{n + 1}v' = 0$, and replacing $v$ by $(S_0)^nv$ and $v'$ by $(S_0)^nv'$, we may assume $S_0v=0$ or $S_0v' = 0$. By \eqref{eq:cond-equiv} we have $S_0v=S_0v'=0$.

\textbf{Step 2.}
We choose ${0}\leq i_1\leq q-1$ such that $S_{i_1}v\in \pi^{I_1}$ and $\langle \GL_2(\cO_K)\cdot S_{i_1}v\rangle$ is an irreducible representation of $\GL_2(\cO_K)$: this is possible by \cite[Lemma 4.1]{paskunas-restriction}. By \eqref{eq:cond-equiv} and the comment that follows we also have $S_{i_1}v'\in \pi^{I_1}\backslash\{0\}$. Replacing $v$ (resp.~$v'$) by $S_{i_1} v$ (resp.~$S_{i_1}v'$), we may assume $\langle \GL_2(\cO_K)\cdot v\rangle$ irreducible. If $S_0v\ne 0$ then it again generates an irreducible representation of $\GL_2(\cO_K)$ by \cite[Lemma 2.6]{BP}, \cite[Lemma 2.7]{BP}. Replacing $v$, $v'$ by $(S_0)^nv$, $(S_0)^nv'$ respectively (as in Step $1$), we can furthermore assume $S_0v=S_0v'=0$. 

\textbf{Step 3.} We prove that $\chi\neq \chi^s$ and $\chi'\neq \chi'^s$. Assume by contradiction $\chi=\chi^s$. Since $S_0v=0$, \cite[Lemma 2.6]{BP} implies $S_iv=0$ except for $i=q-1$. By \eqref{eq:cond-equiv}, we also have $S_{i}v'=0$ except for $i=q-1$, which implies $\chi'=\chi'^{s}$ (using \cite[Lemma 2.7]{BP}). Using condition (C) and the fact $\pi$ has a central character, we get $\chi=\chi'$, contradiction.
 
\textbf{Step 4.} Let $r\in\{1,\dots,q-2\}$ be the unique integer such that 
$\chi(\smatr{[a]}{0}0{[a]^{-1}})=a^{r}$ for $a\in\F_q^{\times}$. Let $\tau:=\langle \GL_2(\cO_K)\cdot \Pi v\rangle\neq 0$.
Since $S_0v=0$, \cite[Lemma 2.7(i)]{BP} implies $0\ne S_{r}v\in \F \Pi v \subset \tau^{I_1}\subset \pi^{I_1}$. Hence $\Pi(S_rv)\in \F v\backslash\{0\}$ as $\pi$ has a central character. Since we have $\langle \GL_2(\cO_K)\cdot v\rangle $ irreducible by Step 2, $\langle \GL_2(\cO_K)\cdot \Pi(S_rv)\rangle$ is irreducible. By \eqref{eq:cond-equiv} and the comment that follows we also have $S_rv'\in\pi^{I_1}\backslash \{0\}$.
 
\textbf{Step 5.} Set $v_1:=S_rv$, $v_1':=S_rv'$ in $\pi^{I_1}\setminus\{0\}$, and let $\chi_1,\chi_1'$ be the corresponding characters of $H$.
Then $\chi_1 \ne \chi_1'$ by \eqref{eq:S_i-H-action}, as $\chi \ne \chi'$. 
By Step 4, $\tau_1:=\langle \GL_2(\cO_K)\cdot \Pi v_1\rangle$ is an irreducible representation of $\GL_2(\cO_K)$ (and hence is the cosocle of $\Ind_I^{\GL_2(\cO_K)}\chi_1^s$), in particular $\tau_1^{I_1}$ is $1$-dimensional. {Thus the set $\{j: S_jv_1\neq0\}$ has a unique minimal element with respect to $\preceq $. }
Assuming \eqref{eq:cond-equiv} for $v_1$ and $v_1'$ (otherwise we are done), we see from \cite[Lemma 2.7]{BP} together with \eqref{eq:cond-equiv} and the comment that follows that 
$\dim_{\F} \tau_1'^{I_1}=1$, where $\tau'_1:=\langle \GL_2(\cO_K)\cdot \Pi v'_1\rangle$.
Hence $\tau_1'$ is irreducible by Lemma \ref{lem:BP-irred}. {Using \cite[Lemma 2.7]{BP} and \eqref{eq:cond-equiv} again we see that $\tau_1$ and $\tau_1'$ are isomorphic up to twist; indeed the set $\{j: S_jv_1\neq 0\}$ determines $\tau_1$ up to twist.}
Using condition (C) (and the fact $\pi$ has a central character) it is easy to check that all this implies $\chi_1=\chi_1'$, contradiction. This finishes the proof.
\end{proof}

Since $|H|$ is prime to $p$, for any $v\in \pi$ we can write
\begin{equation}\label{eq:w=sum}
v=\sum_{\chi:H\ra \F^{\times}}v_{\chi},
\end{equation} 
where the sum runs over all characters $\chi:H\ra\F^{\times}$ and $H$ acts on $v_{\chi}$ via $\chi$. Such an expression is unique, and $v_{\chi}\in \pi^{I_1}$ if $v\in\pi^{I_1}$. Define the length of $v\in \pi$ as
\[\ell(v):=|\{\chi: v_{\chi}\neq0\}|\in \Z_{\geq 0}.\]
Note that $\ell(v)=0$ if and only if $v=0$ and $\ell(v)=1$ if and only if $v$ is a nonzero eigenvector of $H$. 

The following result is a strengthening of \cite[Thm.~4.3]{paskunas-restriction} for certain supersingular representations of $\GL_2(K)$.

\begin{thm}\label{thm:Q-irred}
Let $\pi$ be an irreducible admissible supersingular representation of $\GL_2(K)$ over $\F$ satisfying (C) and such that $\gr(\pi)$ is annihilated by the ideal $J_0$, then $\pi|_{Q}$ is irreducible.
\end{thm} 
\begin{proof}
We modify the proof of \cite[Thm.~4.3]{paskunas-restriction}. Let $w\in \pi\backslash\{0\}$. Since $\pi$ is smooth, there exists an integer $n\geq 0$ such that $w$ is fixed by $\smatr10{p^{n+1}\cO_K}1$. Write $t:=\smatr{p}001$. 
Using the equality
\begin{equation}\label{tp}
\matr{1}0{p^{n}\cO_K}1t=t\matr{1}{0}{p^{n+1}\cO_K}1,\
\end{equation}
we see that $t^{n} w$ is fixed by $\smatr{1}{0}{p\cO_K}1$. Thus by Lemma \ref{lem:key} and the assumption on $\gr(\pi)$, we obtain a nonzero element $v\in \langle Q\cdot w\rangle \cap \pi^{I_1}$. Since the maps $w\mapsto tw$ and $w\mapsto Y_iw$ send $H$-eigenvectors to $H$-eigenvectors (maybe zero), we have $\ell(v)\leq \ell(w)$. 

Next we show that there exists a nonzero element in $ \langle Q\cdot w\rangle \cap \pi^{I_1}$ which is an $H$-eigenvector. Writing $v=\sum_{\chi}v_{\chi}$ as in \eqref{eq:w=sum}, it follows that $v_{\chi}\in \pi^{I_1}$. If $v_{\chi}=0$ except for one $\chi$, we are done. Otherwise, there exist $\chi_1\neq \chi_2$ such that $v_{\chi_1}\neq0$ and $v_{\chi_2}\neq0$. By Lemma \ref{lem:diff-character} there exists a sequence of integers $i_1,\dots,i_k\in\{0,1,\dots,q-1\}$ such that either $(S_{i_k}\circ\cdots \circ S_{i_1})v_{\chi_1}=0$ or $(S_{i_k}\circ\cdots \circ S_{i_1})v_{\chi_2}=0$ (but not both). In particular $(S_{i_k}\circ\cdots \circ S_{i_1})v$ is nonzero and has length strictly less than $\ell(v)$ (recall~\eqref{eq:S_i-H-action}). Applying the argument in the first paragraph of this proof to $(S_{i_k}\circ\cdots \circ S_{i_1})v\in \langle Q\cdot w\rangle$ instead of $w$ (note that $(S_{i_k}\circ\cdots \circ S_{i_1})v$ need not be fixed by $I_1$) and noting that the length can only decrease, we finally obtain after finitely many iterations
\[0\ne v'\in \langle Q\cdot v\rangle\cap \pi^{I_1}\subset \langle Q\cdot w\rangle \cap \pi^{I_1}\]
with $\ell(v')=1$, equivalently $v'$ is an $H$-eigenvector. By Lemma \ref{lem:Sn-vanish} (and again~\eqref{eq:S_i-H-action}), replacing $v'$ by $(S_0)^nv'$ for a suitable $n\geq 0$ we can moreover assume that $v'\in \langle Q\cdot v\rangle\cap \pi^{I_1}\subset \langle Q\cdot w\rangle\cap \pi^{I_1}$ is a nonzero eigenvector of $H$ such that $S_0v'=0$.

By \cite[Lemma 3.4]{paskunas-restriction} we have
\begin{eqnarray}
\matr0110 v'&=&-\sum\limits_{\lambda\in\F_q^{\times}} \matr{-p[\lambda^{-1}]}10{p^{-1}[\lambda]}v' \notag \\
&=&-\sum\limits_{\lambda\in\F_q^{\times}}\matr{p}{[\lambda^{-1}]}0{p^{-1}}\matr{-[\lambda^{-1}]}00{[\lambda]}v'. \label{eq:sv'}
\end{eqnarray}
Since $v'$ is an $H$-eigenvector, we deduce $\smatr{0}110v'\in \langle Q\cdot v'\rangle$. 
Since $\pi$ is an irreducible $\GL_2(K)$-representation and $I$ acts on $v'$ via a character, we obtain using \eqref{eq:Q-decomp}:
\[\pi=\langle \GL_2(K)\cdot v'\rangle=\langle Q\cdot v'\rangle\subset \langle Q\cdot w\rangle.\] 
Hence $\pi=\langle Q\cdot w\rangle$ for any nonzero $w\in \pi$ and so $\pi|_Q$ is irreducible. 
\end{proof}

\begin{rem}\label{thm:Q-irred+}\ 
\begin{enumerate}
\item Since any irreducible admissible representation of $\GL_2(K)$ over $\F$ has a central character, we can even replace $Q$ by $\smatr{p^{\Z}}{K}0{1}$ in Theorem \ref{thm:Q-irred}.
\item It is likely that the statement of Theorem \ref{thm:Q-irred} does not hold without assuming (C). But if one replaces $Q$ by the larger subgroup $QH$ (or by $\smatr{p^{\Z}[\F_q^\times]}{K}0{1}$ in view of (i)) then one can prove the statement holds \emph{without} assuming that $\pi$ satisfies (C). The proof is similar but much simpler as we do not need Lemma \ref{lem:diff-character} anymore and can directly use the $H$-action to get an $H$-eigenvector in $\langle QH\cdot w\rangle \cap \pi^{I_1}$.
\end{enumerate}
\end{rem}

We also have the following partial strengthening of \cite[Thm.~4.4]{paskunas-restriction}.

\begin{thm}\label{localforQ}
Let $\pi$ and $\pi'$ be smooth representations of $\GL_2(K)$ over $\F$ admitting a central character. Assume
\begin{enumerate}
\item\label{as1} $\pi$ is irreducible admissible supersingular;
\item\label{as2} both $\gr(\pi)$ and $\gr(\pi')$ are annihilated by the ideal $J_0$. 
\end{enumerate}
Then the restriction map induces an isomorphism
\[\Hom_{\GL_2(K)}(\pi,\pi')\congto \Hom_{QH}(\pi,\pi').\]
\end{thm} 
\begin{proof}
We follow the proof of \cite[Thm.~4.4]{paskunas-restriction}. Clearly the restriction map is injective. We need to show that any $f\in\Hom_{QH}(\pi,\pi')$ is automatically $\GL_2(K)$-equivariant. We may assume $f$ is nonzero, so $f$ is injective since $\pi|_{QH}$ is irreducible by Remark \ref{thm:Q-irred+}(ii) with assumption \ref{as1}. As a consequence, we deduce the following fact: if $v\in \pi\backslash\{0\}$, then $f(v)\in \pi'\backslash\{0\}$ and $f$ induces an $N_0$-equivariant isomorphism $\langle N_0\cdot v\rangle\simto \langle N_0\cdot f(v)\rangle$. This implies in particular
\begin{equation}\label{eq:deg'}
\deg'(v)=\deg'(f(v)).
\end{equation}

Choose $v\in \pi^{I_1}\backslash\{0\}$. A priori, $f(v)$ need not be fixed by $\smatr{1}0{p\cO_K}1$. But since $\pi'$ is smooth, there exists an integer $n\geq 1$ such that $f(v)$ is fixed by $\smatr10{p^{n+1}\cO_K}1$. As in the proof of Theorem \ref{thm:Q-irred}, we know that $t^n f(v)$ is fixed by $\smatr{1}{0}{p\cO_K}1$ (and $t^n v$ is also fixed by $\smatr{1}{0}{p\cO_K}1$ by (\ref{tp}) since $v\in \pi^{I_1}$). Thus by Lemma \ref{lem:key} (which uses assumption \ref{as2}), we have 
\begin{equation}\label{eq:equality-degs}
\deg (t^nv)=\deg'(t^nv) =\deg'(t^nf(v))=\deg(t^nf(v)),
\end{equation}
where the middle equality holds by \eqref{eq:deg'}. Moreover, still by Lemma \ref{lem:key} there exists $\underline{k}\in \Z_{\geq0}^f$ such that $\|\underline{k}\|=\deg(t^nv)$ and $\underline{Y}^{\underline{k}} t^nv\in \pi^{I_1}\backslash\{0\}$. Then we have 
\[\deg(\underline{Y}^{\underline{k}} t^nf(v))\leq \deg(t^nf(v))-\|\underline{k}\|\overset{\eqref{eq:equality-degs}}{=}0,\]
where the first inequality is a (trivial) general fact and the second follows from \eqref{eq:equality-degs}. Since $\underline{Y}^{\underline{k}} t^nf(v)\ne 0$ as $\underline{Y}^{\underline{k}} t^nv\neq0$ and $f$ is injective, we deduce $\deg(\underline{Y}^{\underline{k}} t^nf(v))=0$ hence $\underline{Y}^{\underline{k}}t^nf(v)\in (\pi')^{I_1}$. Now we put
\[v':=\underline{Y}^{\underline{k}} t^nv,\]
so that $v'\in \pi^{I_1}\backslash\{0\}$ and $f(v')\in(\pi')^{I_1}\backslash\{0\}$. Using Lemma \ref{lem:Sn-vanish} (applied to $v'$) together with the injectivity and $Q$-equivariance of $f$, we can replace $v'$ by $(S_0)^nv'$ for a suitable integer $n\geq 0$ and assume moreover $S_0v'= 0$ and $S_0f(v')=0$. Using $\GL_2(K)=QHI_1\cup QH\smatr{0}11{0}I_1$ with formula \eqref{eq:sv'} and the $QH$-equivariance of $f$, we can then conclude as in the end of the proof of \cite[Thm.~4.4]{paskunas-restriction}.
\end{proof}

\begin{rem}
Theorem \ref{localforQ} is \emph{wrong} with $Q$ instead of $QH$, even assuming that $\pi$ and $\pi'$ satisfy (C). Indeed consider $\theta:K^\times \rightarrow \F^\times$ such that $\theta(p^{\Z}(1+p\cO_K))=1$ and $\theta(x)=\o x^r$ for $x\in [\F_q^\times]$ and $r\in \{1,\dots,q-2\}$. Then for any $\pi$ we have $\pi\vert_Q\cong (\pi\otimes (\theta\circ{\det}))\vert_Q$ but $\pi$ and $\pi\otimes (\theta\circ{\det})$ are not isomorphic as $QH$-representations (assuming $\pi$ satisfies (C) if $r = (q-1)/2$).
\end{rem}

\subsection{The mysterious compact module \texorpdfstring{$D_A(\pi)^\natural$}{D\_A(pi)\^{ }natural}}\label{natural}

For a smooth representation $\pi$ of $\GL_2(K)$ that is contained in category $\mathcal C$ of \cite[\S~3.1.2]{BHHMS2} we study the relations between $\pi$ and the image of $\pi^\vee$ in $D_A(\pi)$.

Let $\pi$ be an admissible smooth representation of $\GL_2(K)$ over $\F$ on which $Z_1$ acts trivially. %
We define as in (\ref{eq:intro3})
\[D_A(\pi):= A \wh\otimes_{\F\bbra{N_0}} (\pi^\vee),\]
{where $\pi^\vee$ carries its $\m_{I_1/Z_1}$-adic filtration.
  We will assume from now on that $\pi$ lies in category $\mathcal C$ defined in \cite[\S~3.1.2]{BHHMS2}\footnote{We remark that the definition in \cite[\S~3.1.2]{BHHMS2} is slightly incorrect, since there it was assumed that objects admit a central character, and this condition is not stable under direct sums. We thank Changjiang Du for pointing this out.}, which we recall means that $\gr(D_A(\pi))$ is a finitely generated $\gr(A)$-module.
  Note that if some power of $J$ kills $\gr(D_A(\pi))$, then $\pi \in \cC$, and if $\pi \in \cC$, then $D_A(\pi)$ is a finite free $A$-module.}

{As $\pi \in \cC$, we have that} $D_A(\pi)$ is a (possibly $0$) finite free $({\vp},\oks)$-module over $A$ endowed with a continuous surjection
\[\psi:D_A(\pi)\twoheadrightarrow D_A(\pi)\]
commuting with $\oks$ and such that $\psi(\varphi(a)d)=a\psi(d)$ for all $(a,d)\in A\times D_A(\pi)$ (see \cite[Rk.~3.26]{BHHMS2}).

Recall that the $(\psi,\oks)$-module $D_A(\pi)$ has a maximal \'etale quotient $D_A(\pi)^{\et}$, which is also a finite free $A$-module. It has the property that there is a canonical continuous $\vp$-semilinear $\vp : D_A(\pi)^{\et} \to D_A(\pi)^{\et}$ that commutes with the action of $\oks$ and such that $\psi(a\vp(d)) = \psi(a)d$ for all $a \in A$, $d \in D_A(\pi)$ \cite[\S~3.1.3]{BHHMS2}.

\begin{rem}\label{rk:action-FN0}
We choose to define the action of $\F\bbra{N_0}$ on $\pi^\vee$ and hence $D_A(\pi)$ by $\lambda f = f \circ \lambda$ for $\lambda \in \F\bbra{N_0}$, cf.\ \cite[Rk.~3.22]{BHHMS2}.
\end{rem}

We define the subgroup (called mirabolic)
\[P := \smatr{K\s}{K}{0}{1} \subset \GL_2(K).\]
For a $(\psi,\oks)$-module $D$ over $\F\bbra{N_0}$, as in \cite[Prop. III.1.1]{Colmez2} we endow
\[\plim_\psi D:= \left\{(d_n)_{n\geq 0},\ d_n\in D,\ \psi(d_n)=d_{n-1}\ {\rm for\ }n\geq 1\right\}\]
with a left $P$-action as follows, where $a \in \oks$ and $b \in K$:
\begin{align*}
 \smatr{p}{0}{0}{1} &: (d_0,d_1,\dots) \mapsto (d_1,d_2,\dots), \\
 \smatr{p^{-1}}{0}{0}{1} &: (d_0,d_1,\dots) \mapsto (\psi(d_0),d_0,d_1,\dots), \\
 \smatr{a}{0}{0}{1} &: (d_0,d_1,\dots) \mapsto (a (d_0),a(d_1),\dots), \\
 \smatr{1}{b}{0}{1} &: (d_0,d_1,\dots) \mapsto (\smatr{1}{b}{0}{1} d_0,\smatr{1}{pb}{0}{1} d_1,\smatr{1}{p^2b}{0}{1} d_2,\dots).
\end{align*}
We are abusing notation in the last row: we have $p^n b \in \ok$ and so $\smatr{1}{p^nb}{0}{1} \in \F\bbra{N_0} \subset A$ for large $n$ only, but by cofinality we can ignore the first $n$ coordinates. It is easy to check that this is well-defined.

From the definition of $D_A(\pi)$ we have a canonical continuous morphism of $\F\bbra{N_0}$-modules $\pi^\vee\rightarrow D_A(\pi)$ which commutes with $\psi$ and $\oks$ {and whose image generates $D_A(\pi)$ as $A$-module} (the $N_0$-action on $\pi^\vee$ differs from the usual one by the automorphism $\smatr{1}{y}{0}{1} \mapsto \smatr{1}{-y}{0}{1}$, cf.\ Remark~\ref{rk:action-FN0} and the reference there). Here $a \in \oks$ acts on $\pi^\vee$ as $f \mapsto f \circ \smatr{a^{-1}}{0}{0}{1}$ and $\psi$ on $\pi^\vee$ is $f \mapsto f \circ \smatr{p}{0}{0}{1}$. We define
\begin{equation}\label{pinatural}
D_A(\pi)^\natural:= \im(\pi^\vee\rightarrow D_A(\pi)),
\end{equation}
which is a compact (even profinite) linear topological $\F\bbra{N_0}$-submodule of $D_A(\pi)$ preserved by $\psi$ and $\oks$. Moreover $\psi:D_A(\pi)^\natural\twoheadrightarrow D_A(\pi)^\natural$ is again surjective as follows from the bijectivity of $\psi$ on $\pi^\vee$.

\begin{prop}\label{limpsi}
Let $\pi$ be an admissible smooth representation of $\GL_2(K)$ over $\F$ in the category $\mathcal C$. Then the surjection $\pi^\vee\twoheadrightarrow D_A(\pi)^\natural$ induces a canonical continuous $P$-equivariant surjection 
\[\pi^\vee \twoheadrightarrow \vplim_{\psi}D_A(\pi)^{\natural}\]
for the projective limit topology on the right hand side (where the $P$-action on $\pi^\vee$ differs from the usual one by the automorphism $\smatr{x}{y}{0}{1} \mapsto \smatr{x}{-y}{0}{1}$, cf.\ Remark~\ref{rk:action-FN0}).
\end{prop}
\begin{proof}
The case $f=1$ is due to Colmez \cite{Colmez}. The proof for $f>1$ is essentially the same, but we give the details. As $\pi^\vee$ and $D_A(\pi)^{\natural}$ are both linearly compact $\F\bbra{N_0}$-modules, the surjection $\pi^\vee\twoheadrightarrow D_A(\pi)^{\natural}$ {induces a surjection $\vplim_{\psi}\pi^\vee\twoheadrightarrow \vplim_{\psi} D_A(\pi)^{\natural}$} by Mittag-Leffler for linearly compact modules (\cite[Thm.\ 7.1]{jensen}, 
applied over the ring $\F$ here). As $\pi^\vee \congto \vplim_{\psi}\pi^\vee$ (since $\smatr{p}{0}{0}{1}$ is bijective on $\pi$) we deduce a surjection $\pi^\vee\twoheadrightarrow \vplim_{\psi}D_A(\pi)^{\natural}$. Its $P$-equivariance is formal (alternatively, it follows from the definitions).
\end{proof}

\begin{rem}\label{remglobal}
We could alternatively consider $D_A(\pi)^{\et,\natural} := \im(\pi^\vee \to D_A(\pi)^{\et})$. However, in the cases of global interest we have $D_A(\pi) = D_A(\pi)^{\et}$, and moreover the natural map $\vplim_{\psi}D_A(\pi)^{\natural} \to \vplim_{\psi}D_A(\pi)^{\et,\natural}$ is always an isomorphism, as $\psi$ is nilpotent on $\ker(D_A(\pi) \to D_A(\pi)^{\et})$ (see \cite[\S~3.1.2]{BHHMS2} and \cite[Rk.~3.26]{BHHMS2}).
\end{rem}

The map in Proposition \ref{limpsi} is not injective for principal series, see \S~\ref{sec:remarks-map-pivee} below. However it is so in the supersingular case.

\begin{cor}\label{limpsiiso}
Let $\pi$ be an admissible smooth representation of $\GL_2(K)$ over $\F$ in the category $\mathcal C$. Assume that $\pi$ has finite length and that all its irreducible constituents $\pi'$ are supersingular with $D_A(\pi')\ne 0$. Then the surjection $\pi^\vee\twoheadrightarrow D_A(\pi)^\natural$ induces a canonical continuous $P$-equivariant isomorphism
\[\pi^\vee \congto \vplim_{\psi}D_A(\pi)^{\natural}.\]
\end{cor}
\begin{proof}
By \ Proposition \ \ref{limpsi} \ it \ is \ enough \ to \ prove \ that \ the \ composition $\pi^\vee\twoheadrightarrow \vplim_{\psi}D_A(\pi)^{\natural}\hookrightarrow \vplim_{\psi}D_A(\pi)$ is injective. By d\'evissage using the exactness of $\pi'\mapsto \protect\varprojlim_\psi D_A(\pi')$ for $\pi'$ in $\mathcal C$ (which follows from the exactness of $\pi'\mapsto D_A(\pi')$ in \cite[Prop.~3.12]{BHHMS2} and the surjectivity of $\psi$ on $D_A(\pi')$), we can assume that $\pi$ is (irreducible) supersingular. The case $f=1$ can be deduced from \cite[Thm.~4.2]{BVi}. By \cite[Thm.~1.1(i)]{paskunas-restriction} (and the fact $\pi$ has a central character), $\pi\vert_{P}$ is an irreducible smooth representation of $P$ over $\F$. As $D_A(\pi)\ne 0$, the map $\pi^\vee\rightarrow D_A(\pi)$ is nonzero (since it is nonzero on graded modules or since the $A$-submodule generated by the image equals $D_A(\pi)$; note that $A$ is a Zariski ring as in \cite[\S~II.2]{LiOy} and it suffices to check this on graded modules). Hence also the (continuous) map $\pi^\vee\rightarrow \vplim_{\psi}D_A(\pi)$ is nonzero. Since $\vplim_{\psi}D_A(\pi)$ is Hausdorff (for the projective limit topology), its kernel is a proper closed $P$-invariant subspace of $\pi^\vee$, hence must be $0$.
\end{proof}

\begin{cor}\label{naturalirr}
Let $\pi$ be an irreducible admissible supersingular representation of $\GL_2(K)$ over $\F$. Assume that $\gr(\pi)$ is killed by the ideal $J$ {(\S~\ref{strengthening})} ({hence} $\pi$ is in $\mathcal C$)
and that $\pi$ satisfies condition (C) of \S~\ref{strengthening}. Then $D_A(\pi)^{\natural}$ does not contain any proper nonzero closed $\F\bbra{N_0}$-submodule preserved by $\psi$ on which $\psi$ is surjective.
\end{cor}
\begin{proof}
This follows from Theorem \ref{thm:Q-irred} with Remark \ref{thm:Q-irred+}(i) since such a submodule would give a proper nonzero subrepresentation of $\pi\vert_{\smatr{p^{\Z}}{K}0{1}}$ by Corollary \ref{limpsiiso}.
\end{proof}

\begin{rem}\label{rem:nonunique}\ 
  \begin{enumerate}
  \item We remark that in general, unlike $f = 1$ \cite[Cor.~II.5.12]{Colmez2}, $D_A(\pi)^\natural$ is \emph{not} unique as minimal nonzero closed $\F\bbra{N_0}$-submodule of $D_A(\pi)$ which is $\oks$-stable and on which $\psi$ is surjective.
  To see this we can take $f = 2$ and take $\o\rho_1 \not\cong\o\rho_2$ both irreducible and sufficiently generic such that $D_A^\otimes(\o\rho_1(1)) \cong D_A^\otimes(\o\rho_2(1))$ as $(\vp,\oks)$-modules over $A$, see \cite[Rk.~2.8.5(ii)]{BHHMS3}.
  In a suitable global context for $i=1,2$ let $\o r_i$ be an automorphic globalization of $\o \rho_i$ satisfying the assumptions of \cite[Thm.~1.2]{BHHMS3}, so that $D_A(\pi(\o r_i)) \cong D_A^\otimes(\o\rho_i(1))$ and $\pi(\o r_i)$ is irreducible and supersingular \cite[Thm.~3.105(i)]{BHHMS2}.
  If we had $D_A(\pi(\o r_1))^\natural \cong D_A(\pi(\o r_2))^\natural$ as $(\psi,\oks)$-modules over $\F\bbra{N_0}$, then we would deduce that $\pi(\o r_1) \cong \pi(\o r_2)$ as $\GL_2(K)$-representations by Corollary~\ref{limpsiiso} and \cite[Thm.~4.4]{paskunas-restriction}.
  This implies that $\o\rho_1 \cong\o\rho_2$ (the determinant of $\o\rho_i$ is determined by the central character of $\pi(\o r_i)$ and the restriction to inertia by the $\GL_2(\cO_K)$-socle), contradiction.
\item {Even if $\rhob$ is irreducible the $(\vp,\oks)$-module $D_A(\pi)$ may be reducible, where $\pi = \pi(\o r)$ and $\o r$ a suitable globalization of $\rhob$.
  For example, when $f = 3$ using \cite[Thm.\ 1.2]{BHHMS3} one can find examples such that $D_A(\pi) = D_A^{(1)} \oplus D_A^{(2)}$ is a direct sum of two nonzero $(\vp,\oks)$-modules over $A$.
  We remark that $D_A(\pi)^\natural$ cannot be of the form $D_A^{(1),\natural} \oplus D_A^{(2),\natural}$ (where $\psi$ is surjective on each $D_A^{(i),\natural} \subset D_A^{(i)}$), as this would contradict Corollary~\ref{limpsiiso}. (Note that $D_A^{(i),\natural} \ne 0$ for all $i$, as $D_A(\pi)^\natural$ generates $D_A(\pi)$ as $A$-module.)}
\end{enumerate}
\end{rem}

Corollary \ref{naturalirr} is quite a strong ``irreducibility'' statement for $D_A(\pi)^{\natural}$ as we do not even need the action of $\cO_K^\times$. Note that we have an analogous statement without assuming (C) if we require that the $\F\bbra{N_0}$-submodule preserved by $\psi$ is also preserved by $[\F_q^\times]$ (see Remark \ref{thm:Q-irred+}(ii)).

Now let $\pi$ be a subquotient of $\pi(\o r)$ with $\o r$ a global Galois representation such that its restriction $\o r_v$ to a decomposition group at a fixed place $v \vert p$ is sufficiently generic (see \S~\ref{sec:intro} for $\o r$ and $\pi(\o r)$; recall $\pi(\o r)$ has finite length by \cite[Thm.~1.1.1]{BHHMS4}). From \cite{BHHMS1}, \cite{YW0}, \cite[Thm.~1.2]{YW}, and \cite[Cor.~1.1.3]{BHHMS5} we know that $\gr(\pi)$ is killed by the ideal $J$ {(\S~\ref{strengthening})} and that all its constituents $\pi'$ satisfy $D_A(\pi')\ne 0$. From the proof of \cite[Lemma 3.2.6]{BHHMS4} (when $\o r_v$ is split reducible, the case $\o r_v$ irreducible was known earlier) and \cite[Thm.~1.1.2(i)]{BHHMS5} (when $\o r_v$ is nonsplit reducible) we also know that {in most cases} $\pi$ satisfies condition (C) in \S~\ref{strengthening} {(for instance when $f=2$ and $\o r_v$ is reducible we need to avoid a few cases like when the restriction to inertia of the semisimplification of $\o r_v$ is $\omega_2^{(q-1)/2} \oplus 1$ up to twist)}. If $\pi$ only has supersingular constituents (which only means that one has to ``avoid'' the two principal series in $\pi(\o r)$ when $\o r_v$ is reducible, see \cite{BHHMS4}), we can apply Corollary \ref{limpsiiso} to $\pi$, and if $\pi$ is moreover irreducible, $D_A(\pi)^{\natural}$ satisfies the conclusion of Corollary~\ref{naturalirr}.

By \cite{BHHMS3} and \cite{YW2} we have $D_A(\pi(\o r))\cong D_A^\otimes(\o r_v(1))$. Let $\pi$ be the maximal subquotient of $\pi(\o r)$ with only supersingular constituents. If $\o r_v$ is irreducible we have $\pi\cong \pi(\o r)$ and hence $D_A(\pi)\cong D_A^\otimes(\o r_v(1))$. If $\o r_v$ is reducible, writing $\o r_v(1)\cong \smatr{\chi_1}{*}0{\chi_2}$ we have
\begin{equation}\label{form}
 D_A^\otimes(\o r_v(1))\ \cong \ \Big(D_A^\otimes(\chi_1)\!\begin{xy} (0,0)*+{}="a"; (9,0)*+{}="b"; {\ar@{-}"a";"b"}\end{xy}\! D_A^\otimes(\o r_v(1))^{\ss} \!\begin{xy} (15,0)*+{}="a"; (24,0)*+{}="b"; {\ar@{-}"a";"b"}\end{xy} \!D_A^\otimes(\chi_2)\Big),
\end{equation}
where the superscript ``ss'' is for ``supersingular''. {If $\o r_v$ is irreducible we also write $D_A^\otimes(\o r_v(1))^{\ss}:=D_A^\otimes(\o r_v(1))$.} From the exactness of $\pi' \mapsto D_A(\pi')$ (see \cite[Prop.~3.12]{BHHMS2}) and the calculations in \cite{BHHMS3} and \cite{YW2}, we deduce
\begin{equation}\label{piglobal}
D_A(\pi)\cong D_A^\otimes(\o r_v(1))^{\ss}
\end{equation}
which {has rank $2^f$ over $A$ in the irreducible case and} rank $2^f-2$ in the reducible case (see \S~\ref{hand} below for an explicit description of $D_A^\otimes(\o r_v(1))^{\ss}$ {in the reducible case when $f=2$} and note we have in particular $D_A(\pi) = D_A(\pi)^{\et}$). Thus $D_A(\pi)$ is not a mystery but we need to understand $D_A(\pi)^{\natural}$. This question is crucial:

\begin{cor}\label{localfail1}
Let $\pi$ be the maximal subquotient of $\pi(\o r)$ with only supersingular constituents. If the $\GL_2(K)$-representation $\pi$ is local, then $D_A(\pi)^{\natural}$ is local as a $(\psi,[\F_q^\times])$-module over $\F\bbra{N_0}$. The converse holds if $\pi$ is irreducible.
\end{cor}

We recall that ``local'' means ``depending only on $\o r_v$ and not on the global context''.
Also, by a \emph{$(\psi,[\F_q^\times])$-module over $\F\bbra{N_0}$} we mean any $\F\bbra{N_0}$-module $D$ endowed with an additive map $\psi:D\rightarrow D$ such that $\psi(\varphi(a)d)=a\psi(d)$~$\forall\ (a,d)\in \F\bbra{N_0}\times D$ and a semilinear action of $[\F_q^\times]$ that commutes with $\psi$. We note that $\pi$ is indeed irreducible in the important special cases where either $\o r_v$ is irreducible or when $f = 2$.

\begin{proof}
The direction $\Rightarrow$ is obvious. The direction $\Leftarrow$ follows from Corollary \ref{limpsiiso} with Theorem
\ref{localforQ} (since we assume that $\pi$ is irreducible).
\end{proof}

\begin{rem}
When $\o r_v$ is semisimple, the locality of $\pi$ implies the one of $\pi(\o r)$. When $\o r_v$ is non-semisimple, there is also the problem of the nonsplit extensions between $\pi$ and the two principal series in $\pi(\o r)$. Since for principal series the map in Proposition \ref{limpsi} is not injective (see Lemma \ref{kernel}), and since we cannot prove the locality of $\pi$ anyway, we do not address this case in this text.
\end{rem}

Unfortunately, even in the simplest case $f=2$ and $\o r_v$ reducible (where $\pi$ is irreducible, see \cite[Thm.~3.105]{BHHMS2}), it seems impossible to describe explicitly the $\F\bbra{N_0}$-submodule $D_A(\pi)^{\natural}$ inside $D_A^\otimes(\o r_v(1))^{\ss}$. We explain this next.

\subsection{Trying to describe the \texorpdfstring{$\psi$}{psi}-module \texorpdfstring{$D_A(\pi)^\natural$}{D\_A(pi)\^{ }natural} ``by hand''}\label{hand}

We explain how difficult it is to explicitly describe the $\psi$-module $D_A(\pi)^\natural$.

When $f=1$, we have $\F\bbra{N_0}\cong \F\bbra{X}$, $A=\F\ppar{X}$, and the compact $\F\bbra{X}$-submodules of a finite-dimensional $\F\ppar{X}$-vector space are just the finitely generated $\F\bbra{X}$-sub\-modules (an easy exercise). However, as soon as $f>1$, this completely breaks down (one trivial example for $f=2$: $\F\bbra{Y_0,\frac{Y_1^2}{Y_0}}$ is a compact $\F\bbra{Y_0,Y_1}$-submodule of $A$ which is not finitely generated).

Unfortunately this also happens for the $\F\bbra{N_0}$-submodules $D_A(\pi)^{\natural}$ of Corollary \ref{limpsiiso}. We let $P^+ := \smatr{\ok\backslash\{0\}}{\ok}{0}{1}$, which is a submonoid of $P=\smatr{K\s}{K}{0}{1}$.

\begin{prop}\label{notfinite}
Assume $f>1$ and let $\pi$ be an irreducible admissible supersingular representation of $\GL_2(K)$ over $\F$ which is in the category $\mathcal C$ {(\S~\ref{natural})} and such that $D_A(\pi)\ne 0$. Then the compact $\F\bbra{N_0}$-submodule $D_A(\pi)^{\natural}$ of $D_A(\pi)$ is not finitely generated over $\F\bbra{N_0}$.
\end{prop}
\begin{proof}
(See also \cite[Rk.~3.21]{BHHMS2}.) Since $\pi$ has a central character, twisting by an unramified character {(and enlarging $\F$)} if necessary we can assume that $\smatr{p}{0}{0}{p}$ acts trivially on $\pi$ (this enables us to apply results of \cite{paskunas-restriction}). Let us first prove that any nonzero vector subspace $V$ of $\pi$ preserved by $P^+$ contains a Serre weight (i.e.,~an irreducible $\GL_2(\cO_K)$-subrepresentation of $\pi$). Let $v\in V$ be nonzero. As $\GL_2(K)$ acts smoothly, there exists an integer $n\geq0$ such that $v$ is fixed by $\smatr{1+p^{n+1}\cO_K}{p^{n+1}\cO_K}{p^{n+1}\cO_K}{1+p^{n+1}\cO_K}$ so that $t^nv$ is fixed by $H_n\coloneqq\smatr{1+p^{n+1}\cO_K}{p^{2n+1}\cO_K}{p\cO_K}{1+p^{n+1}\cO_K}$. The decomposition $I_1=(I_1 \cap B)H_n$ shows that $\F[I_1]t^nv=\F[I_1\cap B]t^nv\subset V$. Therefore $0\neq(\F[I_1]t^nv)^{I_1}\subset V \cap \pi^{I_1}$. We may therefore assume $v \in (V \cap \pi^{I_1}) \setminus\{0\}$. By Lemma~\ref{lem:Sn-vanish} we can replace $v$ by {$(S_0)^m v$} $\in V$ for some $m \ge 0$ to guarantee that $v \in (V \cap \pi^{I_1})\setminus \{0\}$ and $S_0 v = 0$. Hence we can apply \cite[Lemma~3.4]{paskunas-restriction} which implies $\smatr{0}{1}{1}{0}v\in \langle P^+\cdot v\rangle$, and thus $\langle \GL_2(\cO_K)\cdot v\rangle\subset \langle P^+\cdot v\rangle \subset V$ (using $\GL_2(\cO_K)\subset P^+I_1 \cup P^+\smatr{0}{1}{1}{0}I_1$). This proves the statement.

Now let $V$ be the continuous dual of $D_A(\pi)^{\natural}$ (under the discrete-compact duality), which is a nonzero vector subspace of $\pi$ preserved by $P^+$, and let $\sigma$ be a Serre weight such that $\langle P^+\cdot \sigma\rangle\subseteq V$. Note that $\langle P^+\cdot \sigma\rangle$ is denoted $I^+(\sigma,\pi)$ in \cite[\S~3.1.1]{yongquan-jussieu} (and does not depend on $\sigma$ by \cite[Cor.~3.17]{yongquan-jussieu}). If $D_A(\pi)^{\natural}$ is finitely generated over $\F\bbra{N_0}$, then $V^{N_0}$, and hence $I^+(\sigma,\pi)^{N_0}$, are finite-dimensional over $\F$. But by \cite[Cor.~2.10]{Benj} this contradicts \cite[Thm.~2.24]{Benj} (when $f=2$) and \cite[Thm.~1.1]{Wu} (when $f\geq 2$).
\end{proof}

\begin{rem}\label{rk:DAfinitetype}
If $\pi$ is (a constituent of) a principal series of $\GL_2(K)$ over $\F$, then $D_A(\pi)^{\natural}$ is finitely generated over $\F\bbra{N_0}$. See Lemma~\ref{lm:DAfinitetype} below.
\end{rem}

Let us spend some time on the simplest example:
\begin{equation}\label{example}
f=2\text{ \ \ and \ \ }\o r_v=\matr{\omega_2^h}{*}01\otimes \omega^{-1},
\end{equation}
where $h=h_0+ph_1$ with $0\leq h_i\leq p-1$ and $(h_0,h_1)\notin \{(0,0),(p-1,p-1)\}$, and where $\omega_2:\gK\rightarrow \F$ is Serre's niveau $2$ fundamental character that we take $\F$-valued via the fixed embedding $\sigma_0:\Fq\hookrightarrow \F$ (see \S~\ref{strengthening}).

By \ (\ref{form}) \ $D_A^\otimes(\o r_v(1))$ \ has \ the \ form \ $D_A^\otimes(\omega_2^h)\!\begin{xy} (0,0)*+{}="a"; (8,0)*+{}="b"; {\ar@{-}"a";"b"}\end{xy}\! D_A^\otimes(\o r_v(1))^{\ss} \!\begin{xy} (15,0)*+{}="a"; (23,0)*+{}="b"; {\ar@{-}"a";"b"}\end{xy} \!D_A^\otimes(1)$. \ Using \cite[Lemma 2.2.2]{BHHMS3} and \cite[(51)]{BHHMS3} we can check that
\begin{equation}\label{ssummand}
D_A^\otimes(\o r_v(1))^{\ss}=Ae_0 \oplus Ae_1\ \ \text{ with }\ \ \left\{\begin{array}{rcl} \varphi(e_0) &=& e_1\\ \varphi(e_1)&=&\big(\frac{Y_0}{\varphi(Y_0)}\big)^he_0\ =\ \big(\frac{Y_0}{Y_1^p}\big)^he_0\end{array}\right.
\end{equation}
and $[\F_q^\times]$ acting trivially on $e_0,e_1$ (we won't need the full $\cO_K^\times$-action and recall that the $\F$-algebra homomorphism $\vp : A \to A$ satisfies $\vp(Y_i) = Y_{i-1}^p$.) From \cite[Thm.~1.1.1]{BHHMS4} $\pi(\o r)$ has the form
\begin{equation}\label{f=2red}
\text{PS}_1 \!\begin{xy} (0,0)*+{}="a"; (8,0)*+{}="b"; {\ar@{-}"a";"b"}\end{xy}\! \pi \!\begin{xy} (15,0)*+{}="a"; (23,0)*+{}="b"; {\ar@{-}"a";"b"}\end{xy} \! \text{PS}_2,
\end{equation}
where $\text{PS}_i$ are two principal series and $\pi$ is irreducible supersingular with $\gr(\pi)$ killed by the ideal $J$ {(\S~\ref{strengthening})}.
By \cite[Thm.~3.1.3]{BHHMS3} and \cite[Thm.~1.1]{YW2}, at least under genericity assumptions on the integers $h_0$, $h_1$ (which must be ``far'' from $0$ and $p-1$, see \emph{loc.~cit.}~for details), we have that $\pi$ satisfies condition (C) in \S~\ref{strengthening} and that $D_A(\pi(\o r))\cong D_A^\otimes(\o r_v(1))$ with $D_A(\text{PS}_2)\cong D_A^\otimes(\omega_2^h)$, $D_A(\text{PS}_1)\cong D_A^\otimes(1)$ (see Lemma \ref{PS} below for $D_A(\text{principal series})$) and $D_A(\pi)\cong D_A^\otimes(\o r_v(1))^{\ss}$, cf.~\eqref{piglobal}. We are interested in $D_A(\pi)^\natural$.

It is convenient to make the following base change in $D_A(\pi)$:
\begin{equation}\label{basechange}
\left\{\begin{array}{rcl} f_0 &:=& \frac{Y_{1}^{h_1}}{Y_{0}}e_0,\\ 
f_1&:=&\frac{Y_{0}^h}{Y_{1}^{1+h_1}}e_1.\end{array}\right.
\end{equation}
Then $[\F_q^\times]$ acts on each $f_i$ by a (non-trivial) character and a computation yields
\begin{equation}\label{psi0}
\left\{\begin{array}{rcl} \varphi(f_0) &=& \frac{1}{Y_{0}^{h_0}Y_{1}^{p-1-h_1}}f_1,\\ \varphi(f_1)&=&\frac{1}{Y_{0}^{p-1-h_0}Y_{1}^{h_1}}f_0.\end{array}\right.
\end{equation}
In particular, using that $\psi(a\vp(x)) = \psi(a)x$ for $a \in A$, $x \in D_A(\pi)$, we have 
\begin{equation}\label{psi}
\left\{\begin{array}{rcl} \psi(f_0) &=& \psi(Y_{0}^{p-1-h_0}Y_{1}^{h_1})f_1,\\ 
\psi(f_1)&=&\psi(Y_{0}^{h_0}Y_{1}^{p-1-h_1})f_0.\end{array}\right.
\end{equation}
We will use the following lemma (available for any $f\geq 1$).

\begin{lem}\label{psi1}
Let $a\in A$ and $n\geq 1$. Then inside $A$ we have
\begin{equation}\label{eq:psi-formula}
\psi^n(\F\bbra{N_0}a)=\sum_{\underline{k}\in\{0,\dots,p^n-1\}^f}\F\bbra{N_0}a_{\underline k},
\end{equation}
where $a_{\underline k}\in A$ are the unique elements such that
\[a=\sum_{\underline{k}\in\{0,\dots,p^n-1\}^f}\!\!\!\varphi^n(a_{\underline k})\underline Y^{\underline k}.\]
\end{lem}
\begin{proof}
We take $n = 1$, the general case following by induction. For $f \in \F\bbra{N_0}$ we have
\begin{equation}\label{eq:psi-fa}
\psi(fa) = \sum_{\un k} \psi(f \un Y^{\un k}) a_{\un k}.
\end{equation}
From the definitions we have $\psi(\F\bbra{N_0}) \subset \F\bbra{N_0}$, hence from~\eqref{eq:psi-fa} we deduce that ``$\subset$'' holds in~\eqref{eq:psi-formula}. For the other inclusion ``$\supset$'', note from~\eqref{eq:psi-fa} that
\begin{equation*}
\psi(\un Y^{\un{p-1}-\un \ell}a) = \sum_{\un k} a_{\un k} \lambda_{\un k, \un \ell},
\end{equation*}
where $\lambda_{\un k, \un \ell} := \psi(\un Y^{\un k + \un{p-1}-\un\ell}) \in \F\bbra{N_0}$ ($\un 0 \le \un k, \un \ell \le \un{p-1}$). It suffices to prove that for all $\un k \in \Z^f_{\ge 0}$ we have
\begin{equation}\label{eq:psi-Y}
\psi(\un Y^{\un k}) \in \F\bbra{N_0}^\times \iff \un k \in \{\un 0, \un{p-1}\},
\end{equation}
because this implies that the $p^f \times p^f$ matrix $(\lambda_{\un k, \un \ell})_{\un k,\un \ell}$ is invertible (even upper-triangular mod $\mathfrak{m}_{N_0}$).

To prove~\eqref{eq:psi-Y}, note that the cases where $\un k = \un 0$ or $k_j \ge p$ for some $j$ are trivial. We may thus suppose that $\un 0 \le \un k \le \un{p-1}$ and $\un k \ne \un 0$. We only give the proof for $f=2$, the interested reader can work out the general case. Write $\un k = (i,j)$. If $(i,j) \notin \{(0,0), (p-1,p-1)\}$, we may and will assume $i > 0$. Then we have 
\begin{equation*}
Y_0^i Y_1^j = \sum_{\lambda_s,\mu_s \in \F_q\s} \lambda_1^{-1}\cdots \lambda_i^{-1}\mu_1^{-p}\cdots\mu_j^{-p} \smatr 1{[\lambda_1]+\cdots+[\mu_j]}01,
\end{equation*}
so by definition of $\psi$, 
\begin{align*}
\psi(Y_0^i Y_1^j) &= \sum_{\substack{\lambda_s,\mu_s \in \F_q\s \\ \sum \lambda_s+\sum \mu_s = 0}} \lambda_1^{-1}\cdots \lambda_i^{-1}\mu_1^{-p}\cdots\mu_j^{-p} \smatr 1{\frac 1p([\lambda_1]+\cdots+[\mu_j])}01 \\
& \equiv \sum_{\substack{\lambda_s,\mu_s \in \F_q\s \\ \sum \lambda_s+\sum \mu_s = 0}} \lambda_1^{-1}\cdots \lambda_i^{-1}\mu_1^{-p}\cdots\mu_j^{-p} \pmod{\mathfrak{m}_{N_0}}.
\end{align*}
Write $\lambda_s = \lambda_1 \alpha_s$ and $\mu_s = \lambda_1 \beta_s$, with $\alpha_s,\beta_s \in \F_q\s$ (so $\alpha_1 = 1$). The above sum becomes
\begin{equation*}
\equiv \sum_{\substack{\lambda_1,\alpha_s,\beta_s \in \F_q\s, \alpha_1 = 1 \\ \sum \alpha_s+\sum \beta_s = 0}} \lambda_1^{-i-pj} \alpha_2^{-1} \cdots \alpha_i^{-1}\beta_1^{-p}\cdots\beta_j^{-p} \pmod{\mathfrak{m}_{N_0}},
\end{equation*}
which is $0 \pmod {\mathfrak{m}_{N_0}}$ by summing first over $\lambda_1$, as $0 < i+pj < q-1$ by assumption.

It remains to prove that $\psi(Y_0^{p-1}Y_1^{p-1}) \notin \mathfrak{m}_{N_0}$. Recall the elements $T_0, T_1$ and $Z_0, Z_1$ of $\F\bbra{N_0}$ from \cite[\S~3.3]{BHHMS3} (each of those pairs generate $\mathfrak{m}_{N_0}$). By definition, $Y_i \equiv Z_i \pmod{\mathfrak{m}_{N_0}}$, so $Y_0^{p-1}Y_1^{p-1} \equiv Z_0^{p-1}Z_1^{p-1} \pmod{\mathfrak{m}_{N_0}^{2p-1}}$. As $\psi(Y_i^p) \in \mathfrak{m}_{N_0}$ for each $i$, we deduce that $\psi(\mathfrak{m}_{N_0}^{2p-1}) \subset \mathfrak{m}_{N_0}$, so
\begin{equation}\label{eq:psi-Y-Z}
\psi(Y_0^{p-1}Y_1^{p-1}) \equiv \psi(Z_0^{p-1}Z_1^{p-1}) \pmod{\mathfrak{m}_{N_0}}.
\end{equation}
From the definitions, $Z_0 = \alpha T_0 + \beta T_1$, $Z_1 = \gamma T_0 + \delta T_1$, with $\smatr \alpha\beta\gamma\delta \in \GL_2(\F)$, and $\psi(T_0^i T_1^j) = (-1)^{i+j}$ for $0 \le i,j \le p-1$. Therefore $\psi(Z_0^{p-1}Z_1^{p-1}) \equiv C \psi(T_0^{p-1}T_1^{p-1}) \equiv C \pmod{\mathfrak{m}_{N_0}}$, where $C \in \F$ is the coefficient of $T_0^{p-1}T_1^{p-1}$ in $(\alpha T_0 + \beta T_1)^{p-1}(\gamma T_0 + \delta T_1)^{p-1}$, and by~\eqref{eq:psi-Y-Z} it suffices to show that $C \ne 0$. A short computation shows that $C = (\alpha\delta-\beta\gamma)^{p-1} \ne 0$.
\end{proof}

In particular if $f=2$ and $0\leq k_0,k_1\leq p-1$, we have by Lemma \ref{psi1} that
\begin{equation}\label{psi2}
\psi(\F\bbra{N_0}Y_{0}^{k_0}Y_{1}^{k_1})=\F\bbra{N_0}.
\end{equation}
By \eqref{psi0} and (\ref{psi}) we see that $\psi(\un Y^{\un k}f_i) = \psi(\un Y^{\un k+\un \ell^{(i)}})f_{1-i}$ for $i = 0,1$, where $\un\ell^{(0)} = (p-1-h_0,h_1)$ and $\un\ell^{(1)} = (h_0,p-1-h_1)$. Using moreover~\eqref{eq:psi-Y} and (\ref{psi2}), we see that $\F\bbra{N_0}f_0 \oplus \F\bbra{N_0}f_1\subset D_A(\pi)$ is a compact $\F\bbra{N_0}$-submodule preserved by $(\psi,[\F_q^\times])$ such that $\psi$ is surjective. However, it is finitely generated over $\F\bbra{N_0}$, and hence cannot be $D_A(\pi)^{\natural}$ by Proposition \ref{notfinite}. One can also argue that it is not stable by $1+p\cO_K$.

One could then consider the closure in $D_A(\pi)$ of the $\F\bbra{N_0}$-submodule $\langle \cO_K^\times\cdot (\F\bbra{N_0}f_0 \oplus \F\bbra{N_0}f_1)\rangle$, which we denote by $\o M$.

\begin{lem}\label{lm:closure-submod}
The closure $\o M$ of the $\F\bbra{N_0}$-submodule $\langle \cO_K^\times\cdot (\F\bbra{N_0}f_0 \oplus \F\bbra{N_0}f_1)\rangle$ of $D_A(\pi)$ is compact.
\end{lem}
\begin{proof}
Define the subring $B := \F\bbra{N_0}\bbra{Y_0^pY_1^{-1},Y_1^pY_0^{-1}} = \F\bbra{Y_0,Y_1,Y_0^pY_1^{-1},Y_1^pY_0^{-1}}$ of $A$, which is compact, local, and $\varphi$-stable. It suffices to show that $\oks \cdot f_i \subset Bf_i$ for $i = 0,1$.

We claim that $f_{a,i} \in 1+\m_B$ for $a \in \oks$ and $i=0,1$, where we recall that $f_{a,i} = \frac{\o a^{p^i}Y_i}{a(Y_i)} \in 1+F_{1-p}A$ \cite[(21)]{BHHMS3}. From the definitions, $a(Y_i)-\o a^{p^i} Y_i$ is contained in the kernel of $\F\bbra{N_0} \onto \F[N_0/N_0^p]$, which is the ideal $(Y_0^p,Y_1^p)$ \cite[Lemma 3.37(i)]{BHHMS2}. Hence $\frac{a(Y_i)}{\o a^{p^i}Y_i} \in 1+(Y_i^{p-1},\frac{Y_{1-i}^p}{Y_i})\F\bbra{N_0} \subset 1+\m_B$, proving the claim.

From Lemma 2.2.2 and \cite[(51)]{BHHMS3} we know $a(e_i) = \vp^i(f_{a,i})^{h(1-\vp)/(1-q)} e_i$ for $a \in \oks$ and $i=0,1$. It follows that $a(f_i) \in \F f_{a,0}^{\Z} f_{a,1}^{\Z} \vp^i(f_{a,i})^{h(1-\vp)/(1-q)} f_i$, which is contained in $Bf_i$ by the claim, as desired.
\end{proof}

As moreover $\psi$ commutes with $\cO_K^\times$ we see that $\o M$ is a compact $\F\bbra{N_0}$-submodule of $D_A(\pi)$ that is preserved by $(\psi,\cO_K^\times)$ and on which $\psi$ is still surjective. Moreover, it is likely that $\o M$ is \emph{not} finitely generated over $\F\bbra{N_0}$. For some time, we thought that $\o M$ was a good candidate for $D_A(\pi)^{\natural}$.
However, it properly contains the $\psi$-submodule $\F\bbra{N_0}f_0 \oplus \F\bbra{N_0}f_1$, and hence cannot be $D_A(\pi)^{\natural}$ by Corollary \ref{naturalirr} {(assuming $\pi$ satisfies condition (C))}. In fact we have the following result which strengthens \emph{loc.~cit.}~in~this~case:

\begin{lem}\label{stronger}
The $\psi$-module $D_A(\pi)^{\natural}$ does not contain any proper nonzero closed $\F\bbra{N_0}$-submodule preserved by $\psi$.
\end{lem}
The point is that we can drop the surjectivity assumption for $\psi$ in Corollary \ref{naturalirr}.

\begin{proof}
Let $M$ be an $\F\bbra{N_0}$-submodule as in the statement of the lemma, and define $M_\infty:=\bigcap_{n\geq 0}\psi^n(M)$. This is also a compact proper $\F\bbra{N_0}$-submodule of $D_A(\pi)^{\natural}$ which is preserved by $\psi$.

Let us first prove $M_\infty\ne 0$. We assume $M_\infty=0$ and seek a contradiction. Then we claim that for any $C\geq 0$ we have $\psi^n(M)\subset F_{-C}D_A(\pi)$ for $n\gg 0$. Indeed, if there is a sequence $x_m\in \psi^{n_m}(M)\backslash F_{-C}D_A(\pi)$ with $n_{m+1}>n_m$, we can extract a converging subsequence as $M$ is compact in the ultrametric (hence metric) space $D_A(\pi)$. Then its limit $x_\infty$ satisfies $x_\infty \in D_A(\pi)\backslash F_{-C}D_A(\pi)$ (as the latter in closed in $D_A(\pi)$) hence $x_\infty\ne 0$. But we also have $x_\infty\in M_\infty$ which contradicts $M_\infty=0$. Now let $0\ne v\in M\subset A f_0 \oplus A f_1$, then for any $C\geq 0$ we have $\psi^{2n}(v)\in (F_{-C}A) f_0\oplus (F_{-C}A) f_1$ for $n \gg 0$, using that the filtration on $D_A(\pi)$ is good (\cite[Def.~I.5.1]{LiOy}). Writing $v=a_0f_0+a_1f_1$ with $a_i \in A$, by an iteration using (\ref{psi0}) we can express $f_i$ in terms of $\varphi^{2n}(f_i)$ for $i=0,1$. This enables us to compute $\psi^{2n}(\F\bbra{N_0}v)$, and an easy calculation shows that $\psi^{2n}(\F\bbra{N_0}a_i \un Y^{\un {p^{2n}}}) \subset F_{-C}A$, i.e.,\ $\psi^{2n}(\F\bbra{N_0}a_i) \subset F_{-C+f}A$ for all $i \in \{0,1\}$ and $n \gg 0$. By Lemma \ref{psi3} below this implies $a_i=0$ hence $v=0$, contradiction. Thus $M_\infty\ne 0$.

Now let us prove that $\psi$ is surjective on $M_\infty$. Let $0\ne v\in M_\infty$ and for any $n\geq 0$ let $ M_n:=\{x\in \psi^n(M) : \psi(x)=v\}$. Then the $M_n$ are nonempty closed subsets of $M$ such that $M_{n+1}\subset M_n$. If $\bigcap_{n\geq 0}M_n=\emptyset$, then as $M$ is compact we have a \emph{finite} number of $M_n$ such that $\bigcap_n M_n= \emptyset$, which is impossible. So $M_\infty$ is a nonzero compact proper $\F\bbra{N_0}$-submodule of $D_A(\pi)$ preserved by $\psi$ on which $\psi$ is surjective. This contradicts Corollary \ref{naturalirr} and finishes the proof. 
\end{proof}

\begin{lem}\label{psi3}
Let $a\in A$ and assume there is $C\geq 1$ such that for all $n\gg 0$ we have $\psi^{nf}(\F\bbra{N_0}a)\subseteq F_{-C}A$. Then $a=0$.
\end{lem}
\begin{proof}
Let $n\gg 0$ such that $\psi^{nf}(\F\bbra{N_0}a)\subseteq F_{-C}A$ and write as in Lemma \ref{psi1} $a=\sum_{\underline{k}\in\{0,\dots,q^n-1\}^f}\varphi^{nf}(a_{\underline k})\underline Y^{\underline k}$. Then \emph{loc.~cit.}~and the assumption imply $a_{\underline k}\in F_{-C}A$ for all $\underline k$, hence $\varphi^{nf}(a_{\underline k})\in F_{-q^nC}A$, hence $a\in F_{-q^nC}A$. Since this holds for all $n\gg 0$ and the filtration on $A$ is separated, we deduce $a=0$.
\end{proof}

\begin{rem}
It is likely that Lemma \ref{stronger} extends to more general cases.
\end{rem}

Lemma \ref{stronger} has the following consequence, which implies that it is not straightforward to find explicit elements in $D_A(\pi)^{\natural}$:

\begin{lem}\label{flippant}
Let $S$ be the multiplicative subset of $\F\bbra{N_0}$ generated by $Y_0$, $Y_1$, then
\[D_A(\pi)^{\natural} \cap (\F\bbra{N_0}_Se_0 \oplus \F\bbra{N_0}_Se_1)=0.\]
\end{lem}
\begin{proof}
Assume that the intersection in the statement is nonzero and set
\[M:=D_A(\pi)^{\natural} \cap (\F\bbra{N_0} f_0 \oplus \F\bbra{N_0} f_1).\]
Then $M$ is also nonzero: take any nonzero element in $D_A(\pi)^{\natural} \cap (\F\bbra{N_0}_Se_0 \oplus \F\bbra{N_0}_Se_1)$, multiply it by suitable powers of $Y_0$ and $Y_1$ and use (\ref{basechange}). Thus $M$ is a proper nonzero closed $\F\bbra{N_0}$-submodule of $D_A(\pi)^{\natural}$ preserved by $\psi$ (we cannot have $D_A(\pi)^{\natural} \subset \F\bbra{N_0} f_0 \oplus \F\bbra{N_0} f_1$ by Proposition \ref{notfinite}). By Lemma \ref{stronger} it cannot exist.
\end{proof}

To this day, we have no explicit description of the $\F\bbra{N_0}$-submodule $D_A(\pi)^{\natural}$ in $D_A(\pi)$, though $D_A(\pi)$ itself is very simple (\ref{ssummand}). In particular we do not know if $D_A(\pi)^{\natural}$ is local as a $(\psi,[\F_q^\times])$-module over $\F\bbra{N_0}$ (see Corollary \ref{localfail1}).

\subsection{The map \texorpdfstring{$\pi^\vee \to \protect\varprojlim_\psi D_A(\pi)^\natural$}{pi\^{ } -> lim D\_A(pi)\^{ }natural} is not always injective}\label{sec:remarks-map-pivee}

We take a small digression in our quest for the ``Locality'' by proving that the surjection $\pi^\vee \twoheadrightarrow \protect\varprojlim_\psi D_A(\pi)^\natural$ in Proposition \ref{limpsi} is not injective for principal series. We again suppose that $f$ is arbitrary.

We denote by $T$ the torus of diagonal matrices in $\GL_2(K)$.
Recall from \cite[Prop.\ 3.78]{BHHMS2}:

\begin{lm}\label{PS}
Suppose that $\pi$ is a (smooth) principal series $\Ind_B^{\GL_2(K)}(\chi)$ for some smooth character $\chi = \chi_1 \otimes \chi_2:T\rightarrow \F^\times$. Then $\pi$ lies in the category $\cC$ {(\S~\ref{natural})} and $D_A(\pi)$ is \'etale and free of rank $1$ over $A$. More precisely, let $\kappa \in \pi^\vee$ be the element sending $f \in \Ind_B^{\GL_2(K)}(\chi)$ to $f(\smatr{0}{1}{1}{0}) \in \F$. Then the image $\o\kappa$ of $\kappa$ in $D_A(\pi)$ by $\pi^\vee\ra D_A(\pi)$ is a basis of $D_A(\pi)$, and we have
\begin{align}
\vp(\o\kappa) &= \chi_2(p)^{-1} \o\kappa, \label{eq:6} \\
a(\o\kappa) &= \chi_2(a)^{-1} \o\kappa \text{\ \ $\forall ~a \in \oks$.}\label{eq:7}
\end{align}
\end{lm}

We still suppose $\pi = \Ind_B^{\GL_2(K)}(\chi)$ as in Lemma \ref{PS}. Then recall we have an exact sequence of $B$-representations (see for instance \cite[Thm.~1]{GAFAVigneras})
\begin{equation}\label{eq:5}
 0 \to \sigma \to \pi|_B \to \chi \to 0,
\end{equation}
where the surjection is evaluation at $1$ and where $\sigma$ is irreducible. The sequence is nonsplit if and only if $\chi_1 \ne \chi_2$. We call $\Theta$ the $P$-equivariant composition
\[\Theta:\pi^\vee\twoheadrightarrow \vplim_{\psi}D_A(\pi)^{\natural}\hookrightarrow \vplim_{\psi}D_A(\pi).\]

\begin{lm}\label{kernel}
With the above notation, the kernel of $\Theta$ is precisely $\chi^\vee=\chi^{-1}$.
\end{lm}
\begin{proof}
The map $\Theta$ is $P$-equivariant, continuous and nonzero (as $\pi^\vee \to D_A(\pi)$ is nonzero). By (\ref{eq:5}) (and using that $\sigma\vert_P$ is still irreducible as $\pi$ has a central character), it suffices to show that $\Theta$ kills $\chi^{-1}$. Let $\kappa' \in \pi^\vee$ be defined by $\kappa'(f) := f(1)$ for $f\in \Ind_B^{\GL_2(K)}(\chi)$, so that $\kappa'$ spans the 1-dimensional $P$-subrepresentation $\chi^{-1}$. Since $\psi(\kappa') = \chi_1(p)\kappa'$, it suffices to show that $\kappa'$ is killed by the canonical map $\pi^\vee \to D_A(\pi)$. But $\kappa'$ belongs to the dual of the first direct summand $\pi_1$ of \cite[(3.68)]{BHHMS2} (which is the $I$-equivariant subspace of $f\in \Ind_B^{\GL_2(K)}(\chi)$ with support in $BI$). Since $\pi_1$ is an admissible $I$-representation, one can consider $D_A(\pi_1)$ and an examination of the proof of \cite[Prop.\ 3.76(ii)]{BHHMS2} shows that $\gr D_A(\pi_1)=0$. Since the filtration on $D_A(\pi_1)$ is exhaustive and separated, we have $D_A(\pi_1)=0$ (\cite[Prop.~I.4.2.2(1)]{LiOy}). By functoriality of $D_A$ it follows that $\kappa'$ is sent to $0$ by $\pi^\vee \to D_A(\pi)$.
\end{proof}

Suppose now $\pi = \pi(\o r)$ with $\o r_v$ generic and reducible. Then the surjection $\pi^\vee \twoheadrightarrow \protect\varprojlim_\psi D_A(\pi)^\natural$, or equivalently the map $\Theta$, is not an injection. But their kernel is at least finite-dimensional over $\F$: $\pi$ is of finite length (\cite{BHHMS4}) and this follows by a d\'evissage as in the proof of Corollary \ref{limpsiiso} (using the injectivity of $\Theta$ for supersingular constituents of $\pi$ by \emph{loc.~cit.}).

We can now provide the argument for Remark \ref{rk:DAfinitetype}:

\begin{lem}\label{lm:DAfinitetype}
Suppose that $\pi$ is a subquotient of a principal series $\Ind_B^{\GL_2(K)}(\chi)$ for some smooth character $\chi = \chi_1 \otimes \chi_2:T\rightarrow \F^\times$. Then $D_A(\pi)^\natural$ is a free $\F\bbra{N_0}$-module of rank at most 1. It is of rank 0 if and only if $\pi$ is 1-dimensional.
\end{lem}
\begin{proof}
Suppose first that $\pi = \Ind_B^{\GL_2(K)}(\chi)$. We write $\pi = \pi_1 \oplus \pi_2$ as in \cite[(3.68)]{BHHMS2}, i.e.,\ $\pi_1$ (resp.\ $\pi_2$) consists of functions supported on $BI$ (resp.\ $B\smatr 0110 I$). From the proof of Lemma~\ref{kernel} we know $D_A(\pi_1) = 0$, so that $D_A(\pi)^\natural$ is isomorphic to the image of $\theta : \pi_2^\vee \to D_A(\pi_2)$. Noting $\pi_2|_{N_0} \cong \Ind_1^{N_0} 1$, we see that $\pi_2^\vee$ is a free $\F\bbra{N_0}$-module of rank 1, and it remains to show that $\theta$ is injective. The induced map $\gr(\theta) : \gr(\pi_2^\vee) \to \gr(D_A(\pi_2)) = \gr(\pi_2^\vee)_S$ on associated graded modules is an injection by \cite[Prop.\ 3.76(ii)]{BHHMS2}. Then $\theta$ is injective, as $\pi_2^\vee$ is complete and $D_A(\pi_2)$ is separated (using \cite[Thm.\ I.4.2.4(5)]{LiOy}).

Let us consider the general case. Let $\pi_\chi := \Ind_B^{\GL_2(K)}(\chi)$ and let $\pi$ be a subquotient. If $\pi \ne \pi_\chi$, then $\chi_1 = \chi_2$, so $\chi$ extends uniquely to a smooth character of $\GL_2(K)$, still denoted by $\chi$, and we have an exact sequence (with obvious notation)
\begin{equation*}
0 \to \chi \to \pi_\chi \to \St \otimes \chi \to 0.
\end{equation*}
If $\pi = \chi$ is 1-dimensional, then $D_A(\pi) = 0$ from the definitions, cf.\ \cite[Prop.\ 3.76(i)]{BHHMS2}. If $\pi = \St \otimes \chi$, by exactness of $D_A$ the injection $D_A(\pi) \hookrightarrow D_A(\pi_\chi)$ is an isomorphism. It is not difficult to check from the explicit description in the previous paragraph that the kernel of $\pi_\chi^\vee \to D_A(\pi_\chi)$ surjects onto $\chi^\vee$, and hence the map $D_A(\pi)^\natural \to D_A(\pi_\chi)^\natural$ is an isomorphism.
\end{proof}

\section{Towards the locality of \texorpdfstring{$\pi$}{pi} via perfectoids I}\label{perfI}

For $\pi$ of global origin having only supersingular constituents as in (\ref{piglobal}), we define two $P$-equivariant subspaces of $\vplim_{\psi}D_A(\pi)$: $D_{A_\infty}(\pi) \cap \protect\varprojlim_\psi D_A(\pi)^\natural$ in \S~\ref{sec:embedd-Ainfty} and $H^0(\zlt,\cF)^G$ in \S~\ref{geometricinter}, the latter being purely local. We hope, but cannot prove, that they are equal and dense in $\vplim_{\psi}D_A(\pi)^{\natural}$, which would imply the locality of $\vplim_{\psi}D_A(\pi)^{\natural}$. %

\subsection{The embedding \texorpdfstring{$D_{A_\infty}(\pi) \into\protect\varprojlim_\psi D_A(\pi)$}{D\_Ainfty(pi) -> lim D\_A(pi)\^{ }natural}}\label{sec:embedd-Ainfty}

Since we could not prove ``by hand'' the locality of $D_A(\pi)^{\natural}$ as $\psi$-module over $\F\bbra{N_0}$ in the simplest case (see \S~\ref{hand}), we called on the perfectoid theory for help. As some preliminary, we construct here a canonical $P$-equivariant embedding $A_\infty \otimes_A D_A(\pi) \into\protect\varprojlim_\psi D_A(\pi)$ that motivates the geometric perfectoid approach.

Recall that $A_\infty$ is the perfectoid version of $A$ (\cite[Lemma 2.4.2]{BHHMS3} {and also \S~\ref{sec:intro}}). %
For $n\geq 0$ let $A^{1/p^n}:=\varphi^{-n}(A)\subset A_\infty$ and
\[A_{(\infty)}:=\varinjlim_{\vp} A \cong \bigcup_{n \ge 0} A^{1/p^n}\subset A_\infty,\]
where the isomorphism in the middle comes from the following obvious commutative diagram in the category of $\F$-algebras, where the maps on the bottom are inclusions:
\begin{equation}\label{eq:diag-A}
  \begin{gathered}
    \xymatrix{
      A \ar@{^{(}->}^{\vp}[r]\ar@{=}[d] & A \ar@{^{(}->}^{\vp}[r] & A \ar@{^{(}->}^{\vp}[r] & \cdots \\
      A \ar@{^{(}->}[r] & A^{1/p} \ar@{^{(}->}[r]\ar^{\cong}_{\vp}[u] & A^{1/p^2} \ar@{^{(}->}[r]\ar^{\cong}_{\vp^2}[u] &
      \cdots }
  \end{gathered}
\end{equation}
The ring $A_{(\infty)}$ is the perfection of $A$, and the ring $A_\infty$ the completion of $A_{(\infty)}$ for the ``$(Y_0,Y_1,\dots,Y_{f-1})$-adic'' topology, see \cite[Rk.~3.3(iii)]{BHHMS2} and \cite[Lemma 2.4.2]{BHHMS3} with \cite[(50)]{BHHMS3} (we use quotation marks as all the $Y_i$ are units, see above (\ref{Yk})). 

Let $D$ be a finite free \'etale $(\vp,\oks)$-module over $A$, i.e.,~$1 \otimes \vp$ induces an isomorphism $A \otimes_{\vp,A} D \congto D$. We have a commutative diagram analogous to (\ref{eq:diag-A}):
\begin{equation}\label{eq:diag-D}
  \begin{gathered}
    \xymatrix{
      D \ar@{^{(}->}^{\vp}[r]\ar@{=}[d] & D \ar@{^{(}->}^{\vp}[r] & D \ar@{^{(}->}^{\vp}[r] & \cdots \\
      D \ar@{^{(}->}[r] & A^{1/p} \otimes_A D \ar@{^{(}->}[r]\ar^{\cong}_{\vp \otimes \vp}[u] & A^{1/p^2} \otimes_A D
      \ar@{^{(}->}[r]\ar^{\cong}_{\vp^2 \otimes \vp^2}[u] & \cdots }
  \end{gathered}
\end{equation}
which gives an isomorphism
\[D_{(\infty)}:=A_{(\infty)} \otimes_A D\cong \bigcup_{n \ge 0} (A^{1/p^n} \otimes_A D) \congto \varinjlim_{\vp} D\]
and endows $\varinjlim_{\vp} D$ with a canonical structure of $A_{(\infty)}$-module. Note that $\vp:=\vp\otimes \vp$ is now a bijection on $D_{(\infty)}$. It will be convenient to note that $\vp^{-n}(D) = \vp^{-n}(A) \otimes_A D = A^{1/p^n} \otimes_A D$ (inside $D_{(\infty)}$) for all $n \ge 0$ and that we have a commutative diagram
\begin{equation}\label{eq:diag-D2}
  \begin{gathered}
    \xymatrix{
      A^{1/p^n} \otimes_A D \ar^(0.55)\sim_(0.55){\vp^n \otimes \vp^n}[rr]&& D. \\
      \vp^{-n}(D) \ar@{=}[u]\ar^\sim_{\vp^n}[urr] }
  \end{gathered}
\end{equation}

Recall that $\psi : D \to D$ is a canonical left inverse of $\vp$ defined in \cite[\S~3.2]{BHHMS3}. We define a canonical injection
\[\theta_0 : D_{(\infty)}\cong \varinjlim_{\vp}  D \hookrightarrow \plim_\psi D\]
(where the indexing set is $\Z_{\ge 0}$ in both limits) as follows: the composition with the $n$-th projection ($n \ge 0$) on the right is defined by the commutative diagram
\begin{equation}
\begin{gathered}\label{rajout}
\xymatrix{
&&&& D  \\
D \ar_{\vp}[r]\ar^{\vp^n}[urrrr] & D \ar_{\!\!\vp}[r]\ar_{\vp^{n-1}}[urrr] & \cdots & D \ar_\vp[r]\ar_{\vp}[ur] & D \ar_\vp[r]\ar@{=}[u] & D \ar_\vp[r]\ar_{\psi}[ul] & \cdots \\
    }
\end{gathered}
\end{equation}
In general, the $m$-th copy of $D$ in ${\ilim}_{\vp} D$ is mapped to the $n$-th copy of $D$ in $\plim_\psi D$ by $\vp^{n-m}$ if $n \ge m$ and by $\psi^{m-n}$ if $n \le m$. Everything is well defined as $\psi \circ \vp = 1$~on~$D$. 

Note that $\plim_\psi D$ is also an $A_{(\infty)}$-module by making $a\in A^{1/p^m}$ act on $(x_n)_{n\geq 0}\in \plim_\psi D$ by $(\psi^{m-n}(\varphi^m(a)x_m)\text{ if }n<m,\ \varphi^n(a)x_n\text{ if }n\geq m)$. The following formal lemma is left to the reader.

\begin{lem}\label{linear}
The injection $\theta_0$ is $A_{(\infty)}$-linear. 
\end{lem}

One also checks that the image of $\theta_0$ consists of all $(x_n)_{n\geq 0}\in \plim_\psi D$ such that the sequence $\vp^{-n}(x_n)$ is eventually constant in $D_{(\infty)}$ (take $m\gg 0$ such that $\varphi^{-n}(x_n)=\varphi^{-n-1}(x_{n+1})$ for $n\geq m$, then its preimage is $x_m$ seen in the the $m$-th copy of $D$ in ${\ilim}_{\vp} D$, or equivalently $\varphi^{-m}(x_m)\in D_{(\infty)}$).

Let $D_\infty := A_\infty \otimes_A D$, an \'etale $(\vp,\oks)$-module over $A_\infty$ for which $\vp:=\vp\otimes \vp$ is again bijective. It is a straightforward exercise to check that the following formula defines a semilinear action of the monoid $P^+=\smatr{\ok\backslash\{0\}}{\ok}{0}{1}$ on $D_{\infty}$ which preserves $D$ and $D_{(\infty)}$ and which uniquely extends to an action of $P=\smatr{K\s}{K}{0}{1}$ which preserves $D_{(\infty)}$:
\begin{equation}\label{action}
\matr{p^n a}{b}{0}{1} x = b \vp^n(ax)\text{ \ for \ all \ }n \geq 0, \ a \in \oks, \ b \in \ok, \ x \in D_\infty,
\end{equation}
where we see $b$ in $\F\bbra{N_0}\cong \F\bbra{\ok}$.

Recall that $\plim_\psi D$ also has a natural $P$-action, see the beginning of \S~\ref{natural}.

\begin{lem}\label{P}
The injection $\theta_0$ is $P$-equivariant.
\end{lem}
\begin{proof}
Since $P$ acts on both sides and the monoid $P^+$ generates the group $P$, it is enough to prove the $P^+$-equivariance, i.e.,~that $\theta_0$ commutes with $\smatr{p}{0}{0}{1}$, $\smatr{a}{0}{0}{1}$ for $a\in \oks$ and $\smatr{1}{b}{0}{1}$ for $b\in \ok$. This is a straightforward exercise keeping track of all the previous actions.
\end{proof}

\begin{rem}
Note that $\psi$ is not defined on $D_\infty$ and the canonical injection $D\hookrightarrow D_\infty$ does not intertwine $\psi$ on $D$ with $\varphi^{-1}$ on $D_\infty$ as $\psi$ is not injective.
\end{rem}

Recall that $D$, $D_{(\infty)}$, $D_\infty$ have canonical topologies as finite free modules over $A$, $A_{(\infty)}$, $A_\infty$ (the topology on all these rings being the ``$(Y_0,Y_1,\dots,Y_{f-1})$-adic'' topology). We endow $\plim_\psi D$ with the projective limit topology. The above $P$-actions are easily checked to be continuous for these topologies.

\begin{prop}\label{embedding}\
  \begin{enumerate}
  \item By continuity, $\theta_0$ extends to a $P$-equivariant $A_{(\infty)}$-linear morphism $\theta : D_\infty \longrightarrow \plim_\psi D$.
  \item If $x\in D_\infty$ and $\theta(x) = (x_n)_{n\geq 0}$, then $\lim_{n \to \infty} \vp^{-n}(x_n) = x$ in $D_\infty$.
  \item The map $\theta$ is injective with image
    \[\Big\{ (x_n)_{n\geq 0} \in \plim_\psi D : \text{the sequence }\vp^{-n}(x_n)\text{ converges in }D_\infty\text{ as }n\rightarrow \infty\Big\},\]
    which is dense in $\plim_\psi D$.
  \end{enumerate}
\end{prop}
\begin{proof}
(i) By Lemma \ref{linear} and Lemma \ref{P}, and as $D_\infty$ is the completion of $D_{(\infty)}$ and $D$ is complete, it suffices to check that the $m$-th projection $D_{(\infty)} \to D$ is continuous for any fixed $m \ge 0$. Via the identification in \eqref{eq:diag-D}, this map becomes
\[A_{(\infty)} \otimes_A D \longrightarrow D,\ \vp^{-n}(a) \otimes x \longmapsto \psi^{n-m}(a)\vp^m(x)\]
for any $a \in A$, $x\in D$ and $n \ge m$ (using that $\psi(a\vp(x)) = \psi(a)x$). Picking any basis $(e_j)_j$ of $D$ over $A$ and using that $(\vp^m(e_j))_j$ is another basis of $D$ over $A$, we reduce to showing continuity in the case $D = A$, i.e.,\ to showing that the map for $n \ge m$
\[\pi_m : A_{(\infty)} \longrightarrow A,\ \vp^{-n}(a) \longmapsto \psi^{n-m}(a)\]
is continuous. For this we use the total degree in the variables $Y_j$, $\deg : A_\infty \to \Z[1/p] \cup \{\infty\}$. By Lemma~\ref{lm:phi-psi-cts} below we deduce that $\deg(\pi_m(a)) \ge p^m \deg(a)-f$ for all $a \in A_{(\infty)}$, and we are done.

(ii) We first write $x \in D_\infty$ as $x = \sum_{n = 0}^\infty y_n$ with $y_n \in \vp^{-n}(D) = A^{1/p^n} \otimes_A D$ and $y_n \to 0$ as $n \to \infty$. By continuity, $\theta(x) = \sum_{n = 0}^\infty \theta(y_n)$. By \eqref{eq:diag-D} and \eqref{eq:diag-D2}, one easily checks that the $m$-th projection of $\theta(x)$ in $D$ equals
\[x_m = \vp^m\Big(\sum_{n=0}^m y_n\Big) + \sum_{n > m} \psi^{n-m}(\vp^n(y_n)),\]
or equivalently $\vp^{-m}(x_m) = \sum_{n=0}^m y_n + \sum_{n > m} \vp^{-m}(\psi^{n-m}(\vp^n(y_n)))$ in $\vp^{-m}(D)$. It suffices to show that the tail $\sum_{n > m} \vp^{-m}(\psi^{n-m}(\vp^n(y_n)))$ tends to 0 in $\vp^{-m}(D)$ as $m \to \infty$. By choosing a basis $(e_j)_j$ of $D$ and using $\psi^{n-m}(a\vp^n(e_j)) = \psi^{n-m}(a)\vp^m(e_j)$ for $a\in A$, as in (i) we reduce to the case $D = A$. Then Lemma~\ref{lm:phi-psi-cts} below implies $\deg(\vp^{-m}(\psi^{n-m}(\vp^n(y_n)))) \ge \deg(y_n)-p^{-m}f$ (now $y_n \in \vp^{-n}(A)$), which finishes the proof of (ii) since $\deg(y_n)\rightarrow \infty$.

(iii) The injectivity follows from (ii). The image of $\theta$ is dense because, by construction, all projections of $\theta_0$ on $D$ are surjective. Finally by (ii) it suffices to show that if $(x_n)_{n\geq 0} \in \plim_\psi D$ is such that $\vp^{-n}(x_n)$ converges to $x\in D_\infty$ as $n\rightarrow \infty$, then we have $\theta(x) = (x_n)_{n\geq 0}$. From the definition of $\theta_0$ and the discussion below (\ref{rajout}), the coordinate of $\theta(\vp^{-n}(x_n))=\theta_0(\vp^{-n}(x_n))\in \plim_\psi D$ in the $m$-th copy of $D$ in $\plim_\psi D$ equals $\psi^{n-m}(x_n) = x_m$ for any $n \ge m\geq 0$. As $\theta$ is continuous, we deduce that $\theta(x) = (x_n)_{n\geq 0}$.
\end{proof}

\begin{lm}\label{lm:phi-psi-cts}\
  \begin{enumerate}
  \item For $a \in A_\infty$ and $n \in \Z$ we have $\deg(\vp^n(a)) = p^n\deg(a)$.
  \item For $a \in A$ and $n \ge 0$ we have $\deg(\psi^n(a)) \ge p^{-n}\deg(a)-f$.
  \end{enumerate}
\end{lm}
\begin{proof}
Part (i) is obvious. For part (ii), by induction it is enough to show that $\deg(\psi(a)) \ge p^{-1}\deg(a)-f\frac{p-1}p$. To prove this, it is convenient to use the (non-canonical) variables $Z_0, Z_1,\dots,Z_{f-1}$ of $\F\bbra{N_0}$ (instead of the variables $Y_j$) as defined in the beginning of \cite[\S~3.3]{BHHMS3} (note that $A$ is defined similarly using these variables). Defining $\un Z^{\un k}$ as in (\ref{Yk}), the result follows from the fact that $\psi(\un Z^{\un k})$ for $\un k\in \{0,\dots,p-1\}^f$ is of the form $\un Z^{\un j}$ for some $\un j \in \Z_{\geq 0}^f$ with $\Vert \un j\Vert \le p^{-1}\Vert \un k\Vert$ (see the uniqueness part of the proof of \cite[Prop.\ 3.3.1]{BHHMS3}).
\end{proof}

\begin{rem}\label{noaction}\ 
\begin{enumerate}
\item
One can check that there is no natural action of $A_\infty$ on $\plim_{\psi} D$ extending the action of $A_{(\infty)}$ (see below (\ref{rajout})) and the action of $A_\infty$ on the image of $\theta$.

\noindent
To see this, it suffices to show that there exist sequences $(a_m)_{m \ge 1},(b_m)_{m \ge 1} \in A$ such that $\psi(b_m) = b_{m-1}$ for all $m$, $\val(a_m)/p^m \to 0$ as $m \to \infty$, and $\psi^m(a_mb_m) \not\to 0$ as $m \to \infty$. We define $T_0,\dots,T_{f-1}$ as at the beginning of the proof of \cite[Prop.\ 3.3.1]{BHHMS3}, so that in particular $\F\bbra{N_0} = \F\bbra{T_0,\dots,T_{f-1}}$ and $\psi((1+T_0)^i) = 0$ for $1 \le i \le p-1$. Pick any sequence $(x_m)_{m \ge 1}$ in $A$ such that $x_m \not\to 0$ as $m \to \infty$, and let $k_m := (1+T_0)^{p^m-1}\vp^m(x_m)$. Setting $a_m := T_0^{p^m(1+m)}(1+T_0)$ and $b_m := \vp^{m-1}(k_1) + \cdots + k_m$, it is easily checked that all conditions are verified.
\item
We can also view $\plim_\psi D$ as $\plim_{n \ge 0} \vp^{-n}(D)$ via the commutative diagram
\begin{equation*}
\xymatrix{
D & D \ar^\psi[l] & D \ar^\psi[l] & \ar^\psi[l] \cdots \\
D\ar@{=}[u] & \vp^{-1}(D) \ar^\vp_\cong[u]\ar[l] & \vp^{-2}(D) \ar^{\vp^2}_\cong[u]\ar[l] & \cdots \ar[l]
  }
\end{equation*}
where the transition map $\vp^{-n}(D) \to \vp^{-(n-1)}(D)$ is $\vp^{-(n-1)} \psi \vp^n$ (it is $A^{1/p^{n-1}}$-linear for $n \ge 1$ via the inclusion $A^{1/p^{n-1}} \into A^{1/p^n}$ and restricts to the identity on $\vp^{-(n-1)}(D)$). From $\psi \circ \vp = 1$ we deduce inductively that $\vp^{-n}(D) = D \oplus \bigoplus_{m = 1}^n \vp^{-m}(\ker \psi)$ ($n \ge 0$), and hence that
\[\plim_\psi D\cong \varprojlim_{n \ge 0} \vp^{-n}(D) \cong D \oplus \prod_{m=1}^\infty \vp^{-m}(\ker\psi).\]
\end{enumerate}
\end{rem}

Let now $\pi$ be an admissible smooth representation of $\GL_2(K)$ over $\F$ in the category $\mathcal C$ {(\S~\ref{natural})} such that $D_A(\pi)$ is \'etale. We thus have a canonical $P$-equivariant $A_{(\infty)}$-linear embedding with dense image
\begin{equation}\label{embeddingpi}
\theta:A_\infty \otimes_A D_A(\pi) \into\protect\varprojlim_\psi D_A(\pi).
\end{equation}
Assume moreover that $\pi$ has finite length and that all its constituents $\pi'$ are supersingular with $D_A(\pi')\ne 0$. Then from Corollary \ref{limpsiiso} we have a $P$-equivariant cartesian diagram
\begin{equation}\label{square}
\begin{gathered}
\xymatrix{D_{A_\infty}(\pi):=A_\infty \otimes_A D_A(\pi)  \ar@{^{(}->}^{\ \ \ \ \ \ \ \ \ \ \theta}[r] \ar@{}[dr] | {\square} & \smash{\varprojlim_\psi}\: D_A(\pi) & \\
  \mathop{\vphantom{{}^2}}X \ \ar@{^{(}->}[r] \ar@{^{(}->}[u] & \varprojlim_\psi D_A(\pi)^\natural \ar@{^{(}->}[u] & \ar^{\ \ \ \ \ \ \sim}[l]\pi^\vee}
\end{gathered}
\end{equation}
where by definition $X$ is the intersection of $D_{A_\infty}(\pi)$ and $\varprojlim_\psi D_A(\pi)^\natural$. The main reason to consider the diagram (\ref{square}) is that, when $\pi$ is the maximal subquotient of $\pi(\o r)$ with only supersingular constituents as in Corollary \ref{localfail1}, the $(\varphi,\oks)$-module $D_{A_\infty}(\pi)$ has a natural geometric perfectoid interpretation, see \cite{BHHMS3}. And there is also a possible geometric candidate for $X$ that we will define (and try to study)~below.

We denote by $\F\bbra{N_0^{p^{-\infty}}}$ the completed perfection of $\F\bbra{N_0}$. We have $\F\bbra{N_0^{p^{-\infty}}} \cong \F\bbra{Y_0^{p^{-\infty}},\dots,Y_{f-1}^{p^{-\infty}}}\hookrightarrow A_\infty$. When working with $\varprojlim_\psi D_A(\pi)^\natural$ we at least have an action of the subring $\F\bbra{N_0^{p^{-\infty}}}$ of $A_\infty$ (compare with Remark \ref{noaction}(i)):

\begin{lem}\label{lem:action-extends}
The action of $\bigcup_{n\geq 0}\varphi^{-n}(\F\bbra{N_0})\subset A_{(\infty)}$ on $\varprojlim_\psi D_A(\pi)^\natural$ extends by continuity to an action of $\F\bbra{N_0^{p^{-\infty}}} \cong \F\bbra{Y_0^{p^{-\infty}},\dots,Y_{f-1}^{p^{-\infty}}}$.
\end{lem}
\begin{proof}
Write $a \in \F\bbra{Y_0^{p^{-\infty}},\dots,Y_{f-1}^{p^{-\infty}}}$ as $a = \sum_{n = 0}^\infty a_n$ with $a_n \in \vp^{-n}(\F\bbra{N_0})$ and $a_n \to 0$ as $n \to \infty$. Let $x=(x_n)_{n\geq 0}\in \varprojlim_\psi D_A(\pi)^\natural$, from the action of $A_{(\infty)}$ on $x$ (see before Lemma \ref{linear}) it suffices to check that for any $m\geq 0$,
\begin{equation}\label{eq:zero-seq}
\psi^{n-m}(\varphi^n(a_n)x_n) \to 0\text{ in }D_A(\pi)^\natural\text{ as }n\to \infty.
\end{equation}
Write $a_n = \sum_{\un i} \un Y^{\un i/p^n} a_{n,\un i}$ with $a_{n,\un i} \in \F\bbra{N_0}$, and where the sum runs through $\un i \in \{0,1,\dots,p^n-1\}^f$. As $a_n \to 0$ as $n \to \infty$ we still have $a_{n,\un i} \to 0$ as $n \to \infty$ (uniformly in $\un i$). We can write $\psi^{n-m}(\varphi^n(a_n)x_n) = \sum_{\un i} \vp^m(a_{n,\un i})\psi^{n-m}(\un Y^{\un i} x_n)$. As $\vp^m(a_{n,\un i}) \to 0$ uniformly in $\un i$ and as $\psi^{n-m}(\un Y^{\un i} x_n)$ are contained in the (linearly) compact $\F\bbra{N_0}$-module $D_A(\pi)^\natural$, we deduce that~\eqref{eq:zero-seq} holds.
\end{proof}

Since, evidently, the action of $\bigcup_{n\geq 0}\varphi^{-n}(\F\bbra{N_0})$ extends by continuity to an action of $\F\bbra{N_0^{p^{-\infty}}}$ also on $D_{A_\infty}(\pi)$, we deduce that $X$ in \eqref{square} is an $\F\bbra{N_0^{p^{-\infty}}}$-module.

Ideally, we would like the bottom injection in (\ref{square}) to have dense image, at least for $\pi$ the maximal subquotient of $\pi(\o r)$ with supersingular constituents. This is the case when $f=1$ as follows from \cite[Lemma IV.2.2]{Colmez2} (with \cite{Colmez}). This density combined with Proposition \ref{notfinite} would have the following consequence:

\begin{lem}\label{notfinitebis}
Assume $f > 1$. Let $\pi$ as in (\ref{square}) and assume that $X$ is dense in $\varprojlim_\psi D_A(\pi)^\natural$. Then $X$ is not finitely generated as an $\F\bbra{N_0^{p^{-\infty}}}$-module.
\end{lem}
\begin{proof}
The density assumption implies that the composition $X\hookrightarrow \varprojlim_\psi D_A(\pi)^\natural$ with the projection to the first factor $D_A(\pi)^\natural$ has dense image. Assume $X$ is finitely generated over $\F\bbra{N_0^{p^{-\infty}}}$, then by Proposition \ref{notfinite} it is enough to prove that the image of this composition in $D_A(\pi)^\natural$ is finitely generated over $\F\bbra{N_0}$ (as it will be closed and hence equal to the whole $D_A(\pi)^\natural$). Let $\m$ be the maximal ideal of the local ring $\F\bbra{N_0^{p^{-\infty}}}$ and $d:=\dim_{\F}X/\m X$. Since $\varphi$ is bijective on $X$ (as $X$ is a $P$-representation), it is bijective on $X/\fm X$. Let $(e_i)_{1\leq i\leq d}\in X^d$ lifting a basis of $X/\fm X$, by Nakayama's lemma we have $\sum_{i=1}^d\F\bbra{N_0^{p^{-\infty}}}e_i=X$ and the action of $\varphi$ on $X$ is given on the vectors $e_i$ by a (possibly non-unique) matrix in $\GL_d(\F\bbra{N_0^{p^{-\infty}}})$. Choosing such a matrix gives a $\varphi$-equivariant surjection of \'etale $\varphi$-modules $Y:=\bigoplus_{i=1}^d\F\bbra{N_0^{p^{-\infty}}} e_i\twoheadrightarrow X$. By an $\F$-linear variant of \cite[Prop.~3.2.7]{KL1} (see also \cite[Prop.~2.6.3]{BHHMS3}), there is a $d$-dimensional \'etale $\varphi$-vector subspace $V$ of $Y$ such that $\F\bbra{N_0^{p^{-\infty}}}\otimes_{\F}V\congto Y$. Moreover the composition $Y\twoheadrightarrow X\subset  D_{A_{\infty}}(\pi)$ induces a $\varphi$-equivariant morphism $A_{\infty} \otimes_\F V \rightarrow D_{A_{\infty}}(\pi)$, which by (the $\varphi$-version of) \cite[Thm.~2.6.4]{BHHMS3} maps $V$ to $D_A(\pi)\subset D_{A_{\infty}}(\pi)$. It follows that $X$ is generated over $\F\bbra{N_0^{p^{-\infty}}}$ by finitely many vectors of $D_A(\pi)$. Hence it is enough to prove that, if $x\in D_A(\pi)$, then the image of the composition of $\theta$ with the projection on the first factor:
\[\F\bbra{N_0^{p^{-\infty}}}x\buildrel\theta\over\hookrightarrow \varprojlim_\psi D_A(\pi)\twoheadrightarrow D_A(\pi)\]
is a finitely generated $\F\bbra{N_0}$-submodule of $D_A(\pi)$. But since $x\in D_A(\pi)$, arguing as in the proof of Proposition \ref{embedding}(i), we see that this image consists of all $\sum_{n \ge 0} \psi^n(\vp^n(a_n))x$, where $a_n \in \F\bbra{N_0^{p^{-n}}}$ and $a_n \to 0$ as $n \to \infty$. Hence the image equals $\F\bbra{N_0}x$ (noting that $\psi(\F\bbra{N_0}) = \F\bbra{N_0}$ by Lemma~\ref{psi1}).
\end{proof}

When $\pi$ is irreducible (supersingular), to have $X$ dense it would be enough to prove $X\ne 0$ by Corollary \ref{limpsiiso} as $\pi^\vee\vert_P$ is then (topologically) irreducible. However, even in the simplest case $f=2$ and $\pi$ as in (\ref{f=2red}), we do not know if $X$ is $0$ or not for lack of a better understanding of $D_A(\pi)^\natural$. Likewise we do not know if the geometric candidate for $X$ in \S~\ref{geometricinter} below is $0$ or not (let alone if it is isomorphic to $X$).

\begin{rem}\label{rem:X-fg}
If $f = 1$ and we suppose for simplicity that $\pi$ is irreducible, then $X$ is still not finitely generated as $\F\bbra{N_0^{p^{-\infty}}}$-module, as we now briefly explain. In this case $\pi$ corresponds to an irreducible Galois representation $\o \rho : G_{\Qp} \to \GL_2(\F)$, and we have $D_A(\pi) \cong D(\o\rho)$ (at least up to twist that we can forget here) with $D(\o\rho) = \F\ppar T e_1 \oplus \F\ppar T e_2$ and $\varphi(e_1) = e_2$, $\varphi(e_2) = \lambda T^{-h(p-1)} e_1$ for some $\lambda \in \F\s$ and some $h \in \Z$, $p+1 \nmid h$. By \cite[Cor.\ II.5.12]{Colmez2} our $D_A(\pi)^\natural$ contains Colmez's $(D_A(\pi))^\natural$ as a closed $\F\bbra T$-submodule that is preserved by $\psi$, $\Zp\s$, and on which $\psi$ is surjective. Hence by Corollary~\ref{limpsiiso} (and \cite[Prop.\ 3.76(iv)]{BHHMS2}) we deduce that $D_A(\pi)^\natural = (D_A(\pi))^\natural$ (noting that $\pi|_P$ is irreducible). {(Alternatively one can deduce this from \cite[Prop.~IV.4.16]{Colmez}.)} Thus by \cite[Lemme IV.2.2(ii),\;(iii)]{Colmez2} we deduce that
\begin{equation}\label{eq:X1}
X = \big\{ z \in D_{A_\infty}(\pi) : \text{the sequence $\{\varphi^n(z)\}_{n \ge 0}$ is bounded}\big\}
\end{equation}
and that $X$ is dense in $\varprojlim_\psi D_A(\pi)^\natural$. An explicit calculation using~\eqref{eq:X1} shows that
\begin{equation*}
X = \Big\{ x e_1 + y e_2 : x,y \in F\ppar{T^{p^{-\infty}}}, \val(x) \ge \frac h{p+1}, \val(y) \ge \frac {hp}{p+1} \Big\},
\end{equation*}
where val is the $T$-adic valuation. As $p+1 \nmid h$, it follows easily that $X$ is not finitely generated as $\F\bbra{T^{p^{-\infty}}}$-module. On the other hand, it is almost finitely generated as an $\F\bbra{T^{p^{-\infty}}}$-module with respect to its maximal ideal. Namely, if we let $X_n$ (for $n \ge 0$) denote the $\F\bbra{T^{p^{-\infty}}}$-submodule of $X$ that is generated by $\varphi^{-n}(X \cap D_A(\pi))$, then each $X_n$ is finitely generated, and for any $c \in \Z[1/p]_{>0}$ there exists an $n$ such that $T^c X \subset X_n$. For $f > 1$ it is natural to ask if $X$ (assumed to be dense in $\varprojlim_\psi D_A(\pi)^\natural$) is an almost finitely generated $\F\bbra{N_0^{p^{-\infty}}}$-module, though this might be quite optimistic.
\end{rem}

\subsection{The perfectoid space \texorpdfstring{$\zlt$}{Z\_LT}}\label{sec:space-zlt}

We \ explicitly \ describe \ the \ global \ sections \ $H^0(\zlt,\cO_{\zlt})$ \ and \ its \ subspace $H^0(\zlt,\cO_{\zlt})^{G\cap (p^{\Z})^f}$. These results will be crucially used in the sequel.

We consider the perfectoid space $\zlt$ of \cite[\S~2.3]{BHHMS3} defined as the open subset of $\Spa \F\bbra{T_0^{q^{-\infty}},T_1^{q^{-\infty}},\dots,T_{f-1}^{q^{-\infty}}}$ obtained by removing the vanishing locus of $T_0T_1\cdots T_{f-1}$ (so $\zlt$ can be seen as a perfectoid space over the perfectoid field $\F\ppar{T_i^{q^{-\infty}}}$ for each $i$). Here each $T_i$ for $i\in \{0,\dots,f-1\}$ is the formal variable of the Lubin--Tate formal group over $\ok$ associated to the uniformizer $p$ such that the logarithm is $\sum_{n\geq 0}\frac{T_i^{q^n}}{p^n}$, where we use the fixed embedding $\sigma_0:\F_q\hookrightarrow \F$ of \S~\ref{strengthening} to work with $\F$-coefficients. Equivalently we have
\[Z_{\LT}=\Spa\big(\F\ppar{T_0^{q^{-\infty}}},\F\bbra{T_0^{q^{-\infty}}}\big)\times_{\Spa(\F)}\cdots \times_{\Spa(\F)}\Spa\big(\F\ppar{T_{f-1}^{q^{-\infty}}},\F\bbra{T_{f-1}^{q^{-\infty}}}\big),\]
where the fiber product {is taken} in the category of adic spaces. We refer to Corollary \ref{section} below for a concrete description of its global sections
\begin{equation}\label{global}
H^0(\zlt,\cO_{\zlt})=\F\ppar{T_0^{q^{-\infty}}}\widehat\otimes_{\F}\cdots\widehat\otimes_{\F}\F\ppar{T_{f-1}^{q^{-\infty}}}
\end{equation}
in the case $f=2$, {(where $\widehat\otimes_{\F}$ is a certain Fr\'echet completion,}
see the corollary for a precise statement), the general case being similar, {see Remark~\ref{casf>2}}.

It follows that $\zlt$ is naturally endowed with an action of the semidirect product $(K^\times)^f\rtimes S_f$, where $S_f$ acts on the left on $(K^\times)^f$ by permuting the factors. Here the $i$-th factor $K^\times$ acts trivially on $T_j$ for $j\ne i$ and acts $\F$-linearly on the variable $T_i$ by the Lubin--Tate action for $\oks$ and by $\varphi_q(T_i)=T_i^q$ for $p$, and $S_f$ permutes the $T_i$, see the beginning of \cite[\S~2.4]{BHHMS3}. Moreover $\zlt$ is a quasi-Stein adic space (see Lemma \ref{stein} below for the case $f=2$, the general case is similar). Hence a locally {finite} free $\cO_{\zlt}$-module $\mathcal F$ on $\zlt$ is determined by its global sections $H^0(\zlt,{\mathcal F})$ (which follows for instance from \cite[Thm.~2.5.5]{KL2} and \cite[Cor.~2.6.4]{KL2}).

We now determine explicitly the global sections $H^0(\zlt)=H^0(\zlt,\cO_{\zlt})$. For simplicity we assume in the remainder of this subsection that $f = 2$. For $N \ge 1$ we define
\[Z_N := \{ |T_0^N| \le |T_1| \ne 0, |T_1^N| \le |T_0| \ne 0 \},\]
a rational open subset of $Z:=\Spa \F\bbra{T_0^{q^{-\infty}},T_1^{q^{-\infty}}}$ that is contained in $\zlt$.

\begin{lm}\label{lm:projN}
We have $H^0(\zlt) = \varprojlim_{N \ge 1} H^0(Z_N)$ (as complete topological rings).
\end{lm}
\begin{proof}
By \cite[Def.~3.1.6]{SWBerkeley} we have $H^0(\zlt) = \cO_Z(\zlt) = \varprojlim \cO_Z(U)$, where the limit runs over rational subsets $U \subset Z$ contained in $\zlt$. If $U$ is such a subset, then $U$ is quasi-compact and $|T_0T_1| \ne 0$ on $U$, so by \cite[Lemma 7.31]{Wedhorn:adic} we deduce that there exists $N \ge 0$ such that $|a| < |T_0T_1|$ on $U$ for all $a \in (T_0,T_1)^{N+1} \F\bbra{T_0^{q^{-\infty}},T_1^{q^{-\infty}}}$. Thus $U \subset Z_{N}$, so we established cofinality and are done.
\end{proof}

The proof shows that we have
\begin{equation}\label{rk:union-rational}
\bigcup_{N \ge 1} Z_N = \zlt
\end{equation}
because rational subsets form a basis of the topology of $Z$.

\begin{lm}\label{lm:global-sec-XN}
We have
\begin{multline*}
H^0(Z_N) = \bigg\{\sum_{n=0}^\infty \lambda_n T_0^{d_{0,n}} T_1^{d_{1,n}} : \lambda_n \in \F, d_{0,n},d_{1,n} \in \Z[1/p];\\
Nd_{0,n}+d_{1,n},d_{0,n}+Nd_{1,n} \to \infty \bigg\}.
\end{multline*}
This is a Banach ring, with topology given by the norm 
\[ \bigg|\sum_{n=0}^\infty \lambda_n T_0^{d_{0,n}} T_1^{d_{1,n}}\bigg|_N := 2^{{}-\min\{Nd_{0,n}+d_{1,n},d_{0,n}+Nd_{1,n} \;:\; n \ge 0\}}.\]
(Here we implicitly assume that all $\lambda_n$ are nonzero and that the $(d_{0,n},d_{1,n})$ are pairwise distinct.)
\end{lm}
\begin{proof}
As in \cite[\S~3.1]{SWBerkeley} we obtain $H^0(Z_N) = \cO_Z(Z_N)$ by taking \[A_0 := \F\bbra{T_0^{q^{-\infty}},T_1^{q^{-\infty}}}\Big[\frac{T_0^N}{T_1},\frac{T_1^N}{T_0}\Big],\] completing for the $(T_0,T_1)$-adic topology and inverting $T_0T_1$. (Incidentally, note that the $(T_0,T_1)$-adic topology equals the $T_i$-adic topology for any $i$.) The ring $A_0$ is given explicitly by the series $\sum_{n=0}^\infty \lambda_n T_0^{d_{0,n}} T_1^{d_{1,n}}$ such that
\begin{itemize}
\item $d_{0,n},d_{1,n} \in \Z[1/p]$,
\item the subset $\{d_{0,n},d_{1,n} : n \ge 0 \} \subset \R$ is bounded below,
\item $d_{0,n}+d_{1,n} \to \infty$ as $n \to \infty$,
\item for all $i \in \Z_{\ge 0}$ we have $d_{0,n} < -i \Rightarrow d_{1,n} \ge (i+1)N$ and $d_{1,n} < -i \Rightarrow d_{0,n} \ge (i+1)N$.
\end{itemize}
Note (by the second item) that the third item can be replaced by the condition
\begin{itemize}
\item $Nd_{0,n}+d_{1,n} \to \infty$ and $d_{0,n}+Nd_{1,n} \to \infty$ as $n \to \infty$.
\end{itemize}
Moreover, the fourth item implies that $Nd_{0,n}+d_{1,n} \ge 0$, $d_{0,n}+Nd_{1,n} \ge 0$ for all $n$. In particular, the $(T_0,T_1)$-adic topology is separated. It is then straightforward to see that the $(T_0,T_1)$-adic completion $\wh A_0$ of $A_0$ is given by the series $\sum_{n=0}^\infty \lambda_n T_0^{d_{0,n}} T_1^{d_{1,n}}$ such that
\begin{itemize} \item $d_{0,n},d_{1,n} \in \Z[1/p]$,
\item $Nd_{0,n}+d_{1,n} \to \infty$ and $d_{0,n}+Nd_{1,n} \to \infty$ as $n \to \infty$,
\item for all $i \in \Z_{\ge 0}$ we have $d_{0,n} < -i \Rightarrow d_{1,n} \ge (i+1)N$ and $d_{1,n} < -i \Rightarrow d_{0,n} \ge (i+1)N$.
\end{itemize}
Another short calculation shows that $\wh A_0[(T_0T_1)^{-1}]$ is precisely the ring in the statement of this lemma. (To see that we obtain all series in the statement, by symmetry it is enough to consider series with $d_{0,n} < 0$ for all $n$ and then write $T_0^{d_{0,n}} T_1^{d_{1,n}} = (T_0^{-1}T_1)^{-\lfloor d_{0,n}\rfloor} T_0^{d_{0,n}-\lfloor d_{0,n}\rfloor} T_1^{d_{1,n}+\lfloor d_{0,n}\rfloor}$.)

For the claim on the topology, it is not hard to verify that for $f \in \wh A_0$ and $i \ge N$,
\begin{equation*}
  f \in (T_0,T_1)^i \wh A_0 \ \Longrightarrow \ |f|_N \le 2^{-i} \ \Longrightarrow \ f \in (T_0,T_1)^{i-N} \wh A_0.
\end{equation*}
(Also, $|{\cdot}|_N$ is not multiplicative, only sub-multiplicative.)
\end{proof}

\begin{rem}
Alternatively we get the Banach topology on $H^0(Z_N)$ from the fact that $H^0(Z_N)$ is Tate with pseudo-uniformizer $T_0$ (or $T_1$).
\end{rem}

From Lemma \ref{lm:projN} and Lemma \ref{lm:global-sec-XN} we obtain the following description of the Fr\'echet ring $H^0(\zlt)$.

\begin{cor}\label{section}
We have 
\begin{multline*}
H^0(\zlt) = \bigg\{\sum_{n=0}^\infty \lambda_n T_0^{d_{0,n}} T_1^{d_{1,n}} : \lambda_n \in \F, d_{0,n},d_{1,n} \in \Z[1/p];\\
xd_{0,n}+yd_{1,n} \to \infty \ \forall \ x,y > 0\bigg\}.
\end{multline*}
\end{cor}

\begin{ex}
For example, $\sum_{i \ge 0} T_0^{-i} T_1^{i^2} \in H^0(\zlt)$.
\end{ex}

{\begin{rem}\label{rem:H0ZLT-topology}
    By Lemma~\ref{lm:projN} it follows that the topology of $H^0(\zlt)$ is given by the family of norms $\lvert\cdot\rvert_N$ for $N \ge 1$ (restricted to $H^0(\zlt)$).
    It is then easy to see that the induced topology on $\F\bbra{T_0^{q^{-\infty}},T_1^{q^{-\infty}}}$ is the $(T_0,T_1)$-adic topology.
    In particular, $\F\bbra{T_0^{q^{-\infty}},T_1^{q^{-\infty}}}$ is a closed subring of $H^0(\zlt)$ (as it is complete).
\end{rem}

\begin{lm}\label{stein}
The adic space $\zlt$ is quasi-Stein \cite[Def.\ 2.6.2]{KL2}.
\end{lm}
\begin{proof}
This follows from the more general statement that $X,Y$ quasi-Stein in $\Perf_{\fq}$ implies $X \times_{\fq} Y$ quasi-Stein (using the proof of \cite[Lemma 8.3.5]{SWBerkeley}), but we provide an explicit argument here. By (\ref{rk:union-rational}) we have an increasing union $\zlt = \bigcup_{N \ge 1} Z_N$, and $T_0$ is a topologically nilpotent unit in $H^0(\zlt)$. It thus suffices to show the (injective) map $H^0(Z_{N+1}) \to H^0(Z_N)$ has dense image. Take $F := \sum_{n=0}^\infty \lambda_n T_0^{d_{0,n}} T_1^{d_{1,n}} \in H^0(Z_N)$, with $\lambda_n \in \F, d_{0,n},d_{1,n} \in \Z[1/p]; Nd_{0,n}+d_{1,n},d_{0,n}+Nd_{1,n} \to +\infty$. Then for each $k \ge 0$ the partial series $F_k := \sum_{n=0}^k \lambda_n T_0^{d_{0,n}} T_1^{d_{1,n}}$ lies in (the image of) $H^0(Z_{N+1})$ and $F_k \to F$ as $k \to +\infty$.
\end{proof}

The following lemma about invariants will be very useful.

\begin{lm}\label{invariant}
  We have $H^0(\zlt)^{(p,p^{-1})^\Z}=\F\bbra{T_0^{q^{-\infty}},T_1^{q^{-\infty}}}^{(p,p^{-1})^\Z}$.
\end{lm}
\begin{proof}
Let $F\in H^0(\zlt)$ such that $(p,p^{-1})(F)= F$. This implies that if a monomial $T_0^{a} T_1^{b}$ appears in the expansion of $F$ (see Corollary \ref{section}), then all monomials $T_0^{q^na}T_1^{q^{-n}b}$ for $n\in \Z$ must also appear. If $a<0$, then $q^na+q^{-n}b\rightarrow -\infty$ as $n\rightarrow \infty$, which contradicts the condition in Corollary \ref{section}. The same holds by symmetry if $b<0$. Hence all monomials $T_0^{a}T_1^{b}$ appearing in $F$ must be such that $a\geq 0$ and $b\geq 0$.
\end{proof}

\begin{rem}
\label{casf>2}
For $f\geq 1$ it is not difficult to show that
\begin{multline*}
H^0(\zlt) = \bigg\{\sum_{n=0}^\infty \lambda_n T_0^{d_{0,n}}T_1^{d_{1,n}}\cdots T_{f-1}^{d_{f-1,n}} : \lambda_n \in \F, d_{i,n}\in \Z[1/p]; \\
\sum_{i=0}^{f-1}x_id_{i,n} \to \infty \ \forall \ x_0,\dots,x_{f-1} > 0\bigg\}
\end{multline*}
by using the covering (\ref{rk:union-rational}) with
\[Z_N := \{ |T_0^N| \le |T_i| \ne 0\ \forall\ i,\ |T_i^N| \le |T_0| \ne 0\ \forall \ i\}.\]  
Then \ the \ same \ proof \ as \ for \ Lemma \ref{invariant} \ (using \ elements \ of \ the \ form $(1,\dots ,1,p,1,\dots,1,p^{-1},1,\dots,1)\in G\cap (p^{\Z})^f$ for various positions of $p,p^{-1}$) shows that $H^0(\zlt)^{G\cap (p^{\Z})^f}= \F\bbra{T_0^{q^{-\infty}},\dots,T_{f-1}^{q^{-\infty}}}^{G\cap (p^{\Z})^f}$ (cf.\ Lemma \ref{power}).
\end{rem}

\subsection{A geometric candidate for \texorpdfstring{$D_{A_\infty}(\pi) \cap \protect\varprojlim_\psi D_A(\pi)^\natural$}{D\_Ainfty(pi) cap lim D\_A(pi)\^{ }natural}}\label{geometricinter}

We recall the geometric interpretation of $D_{A_\infty}(\pi(\o r))$ and give a possible (local) candidate for $D_{A_\infty}(\pi) \cap \protect\varprojlim_\psi D_A(\pi)^\natural$, where $\pi$ is the maximal subquotient of $\pi(\o r)$ with only supersingular constituents. We assume $f>1$.

{Recall from \S~\ref{sec:space-zlt} that the perfectoid space $\zlt$ is endowed with an action of $(K^\times)^f\rtimes S_f$.}
We define $G:=\ker((K^\times)^f\rtimes S_f\twoheadrightarrow K^\times,\ (k_0,\dots,k_{f-1},w)\mapsto \prod_{i=0}^{f-1}k_i)$. For any $(K^\times)^f\rtimes S_f$-equivariant locally {finite} free $\cO_{\zlt}$-module $\cF$ the $\F$-vector space
\begin{equation}\label{H0}
H^0(\zlt,\cF)^G
\end{equation}
is endowed with the residual action of $K^\times$. Unfortunately, we do not understand the spaces (\ref{H0}) except in the case where $\cF=\cO_{\zlt}$ (cf.\ \eqref{G} and also \S~\ref{sec:aside-trivial-galois}). {Recall the following result from Lemma \ref{invariant} and Remark \ref{casf>2}.}

\begin{lm}\label{power}
We have $\F\bbra{T_0^{q^{-\infty}},T_1^{q^{-\infty}},\dots,T_{f-1}^{q^{-\infty}}}^{G\cap (p^{\Z})^f}\congto H^0(\zlt,\cO_{\zlt})^{G\cap (p^{\Z})^f}$.
\end{lm}

From Lemma \ref{power} we have in particular
\begin{equation}\label{G}
\F\bbra{T_0^{q^{-\infty}},T_1^{q^{-\infty}},\dots,T_{f-1}^{q^{-\infty}}}^G\congto H^0(\zlt,\cO_{\zlt})^{G}.
\end{equation}
{This is a topological isomorphism by Remark~\ref{rem:H0ZLT-topology}.}

We consider the perfectoid space $\zok$ of \cite[\S~2.3]{BHHMS3} defined as the open subset of $\Spa \F\bbra{N_0^{p^{-\infty}}} = \Spa \F\bbra{Y_0^{p^{-\infty}},Y_1^{p^{-\infty}},\dots,Y_{f-1}^{p^{-\infty}}}$ obtained by removing the unique non-analytic point (equivalently the simultaneous vanishing locus of all $Y_i$). The perfectoid space $\zok$ is not quasi-Stein anymore, but it is quasi-compact {because of the following finite covering of $\zok$ by rational subsets of $\Spa \F\bbra{N_0^{p^{-\infty}}}$:
\begin{equation}\label{eq:U_i}
U_i:= \mathrm{Spa}\Big(\F\ppar{Y_i^{p^{-\infty}}}\left\langle \big(\frac{Y_j}{Y_i}\big)^{p^{-\infty}}, j\ne i\right\rangle, \F\bbra{Y_i^{p^{-\infty}}}\left\langle \big(\frac{Y_j}{Y_i}\big)^{p^{-\infty}}, j\ne i\right\rangle\Big),
\end{equation}
where $i\in \{0,\dots,f-1\}$. (This covering is analogous to the covering in the proof of \cite[Lemma 2.4.2(iv)]{BHHMS3}.)}
{The perfectoid space $\zok$ is also quasi-separated, since it is open in the affinoid $\Spa \F\bbra{N_0^{p^{-\infty}}}$.}

\begin{lem}\label{hartogs}
We have {a topological isomorphism} $$\F\bbra{N_0^{p^{-\infty}}} = \F\bbra{Y_0^{p^{-\infty}},\dots,Y_{f-1}^{p^{-\infty}}} \congto H^0(\zok,{\mathcal O}_{\zok}).$$
\end{lem}
\begin{proof}
There is a description of $U_i\cap U_j$ for $i\ne j$ {analogous to~\eqref{eq:U_i}}. Moreover all the perfectoid Tate algebras $H^0(U_i,{\mathcal O}_{U_i})$ and $H^0(U_i \cap U_j,{\mathcal O}_{U_i\cap U_j})$ embed into
\[H^0\Big(\bigcap_i U_i,{\mathcal O}_{\bigcap_i U_i}\Big)=\F\ppar{Y_0^{p^{-\infty}}}\left\langle \big(\frac{Y_j}{Y_0}\big)^{p^{-\infty}}\!, \big(\frac{Y_0}{Y_j}\big)^{p^{-\infty}}\!, j\ne 0\right\rangle\]
and $\F\ppar{Y_i^{p^{-\infty}}}\langle (\frac{Y_j}{Y_i})^{p^{-\infty}}, j\ne i\rangle$ coincides with the subring of elements which only have denominators in the variable $Y_i$. It is then clear that any element in the kernel of the map $\bigoplus_{i} \!H^0(U_i,{\mathcal O}_{U_i}) \rightarrow \bigoplus_{i\ne j} \!H^0(U_i\cap U_j,{\mathcal O}_{U_i\cap U_j})$ \ is \ an \ element \ of \ $\F\ppar{Y_0^{p^{-\infty}}}\langle (\frac{Y_j}{Y_0})^{p^{-\infty}}\!\!, (\frac{Y_0}{Y_j})^{p^{-\infty}}\!\!, j\ne 0\rangle$ which has no denominator in any variable $Y_i$ for $i\in \{0,\dots,f-1\}$, hence which lies in $\F\bbra{Y_0^{p^{-\infty}},\dots,Y_{f-1}^{p^{-\infty}}}$. {As ${\mathcal O}_{\zok}$ is a sheaf of topological rings, $H^0(\zok,{\mathcal O}_{\zok})$ carries the subspace topology in $\bigoplus_{i} \!H^0(U_i,{\mathcal O}_{U_i})$, and one checks that it is the $(Y_0,\dots,Y_{f-1})$-adic topology.}
\end{proof}

\begin{rem}
  Note that Lemma \ref{hartogs} does not hold when $f = 1$; in that case, we have $H^0(Z_{\mathcal{O}_K}, \mathcal{O}_{Z_{\mathcal{O}_K}}) = \F\ppar{Y_0^{p^{-\infty}}}$.  See also equation \eqref{global}.
\end{rem}

We let ${\rm Perf}_{\F}$ be the category of perfectoid spaces over ${\rm Spa}(\F,\F)$ and ${{\rm Perf}_{\F}}^{\!\!\!\sim}$ the category of sheaves of sets on the big pro-\'etale site of ${\rm Perf}_{\F}$ (\cite[\S~8.2]{SWBerkeley}). Recall that, if $W$ is any adic space over ${\rm Spa}(\F,\F)$, the functor $h_W(-):= {\rm Hom}_{\F}(-,W)$ lies in ${{\rm Perf}_{\F}}^{\!\!\!\sim}$. 

It turns out that there is a morphism of perfectoid spaces
\begin{equation}\label{m}
m:\zlt\longrightarrow \zok
\end{equation}
which induces an isomorphism of sheaves in ${{\rm Perf}_{\F}}^{\!\!\!\sim}$
\begin{equation}\label{isosheaves}
h_{\zlt}/\underline G\congto h_{\zok},
\end{equation}
where $h_{\zlt}/\underline G$ means the sheaf associated to the presheaf $W\mapsto {\rm Hom}_{\F}(W,\zlt)/\underline G(W)$ {with $\underline G(W):=\cC^0(\vert W\vert, G)$ acting through its action $\underline G \times \zlt\to \zlt$ on $\zlt$}, see \cite[Lemma~7.6]{FarguesAJ} or \cite[\S~2.4]{BHHMS3}.

\begin{prop}\label{globalsections}
We have a commutative diagram in which all arrows are isomorphisms:
\begin{equation}\label{globalsections2}
\begin{gathered}
\xymatrix{H^0\big(\zlt,{\mathcal O}_{\zlt}\big)^G\!\!&H^0\big(\zok,{\mathcal O}_{\zok}\big)\ar^{\sim}[l]\\
\F\bbra{T_{0}^{q^{-\infty}},\dots,T_{f-1}^{q^{-\infty}}}^G\ar^{\wr}[u] &\F\bbra{Y_0^{p^{-\infty}},\dots,Y_{f-1}^{p^{-\infty}}}\rlap{${} \cong \F\bbra{N_0^{p^{-\infty}}}$.}\ar^{\ \sim}[l]\ar^{\wr}[u].}
\end{gathered}
\end{equation}
{The vertical maps are homeomorphisms and the horizontal maps continuous.}
\end{prop}
\begin{proof}
Let $W:= {\rm Spa}(\F[T],\F)$.  As $\zlt$, $\zok$ are in ${\rm Perf}_{\F}$ and using Yoneda's lemma, we deduce:
{\small
\begin{multline*}
H^0(\zok,{\mathcal O}_{\zok})={\rm Hom}_{\F}(\zok,W)={\rm Hom}_{{{\rm Perf}_{\F}}^{\!\!\!\sim}}(h_{\zok},h_W)\congto {\rm Hom}_{{{\rm Perf}_{\F}}^{\!\!\!\sim}}(h_{\zlt}/\underline G,h_W)\\
={\rm Hom}_{{{\rm Perf}_{\F}}^{\!\!\!\sim}}(h_{\zlt},h_W)^{\underline G},
\end{multline*}}%
where the isomorphism follows from (\ref{isosheaves}). {In other words, $H^0(\zok,{\mathcal O}_{\zok})$ is the co-equalizer of the maps $\un G \times \zlt \rightrightarrows \zlt \to W$, where the first two arrows are given by projection, respectively the action map.
Equivalently, it is the equalizer of the maps $H^0(\zlt) \rightrightarrows H^0(\un G \times \zlt) \cong \cC^0(G,H^0(\zlt))$, which equals $H^0(\zlt)^G$.}
We deduce (\ref{globalsections2}) by noting that the vertical maps are homeomorphisms by (\ref{G}) and Lemma \ref{hartogs}.
\end{proof}

{Below we show that the horizontal isomorphisms in (\ref{globalsections2}) are also homeomorphisms.}

\begin{cor}\label{FN0}
Let $\cF$ be a $(K^\times)^f\rtimes S_f$-equivariant locally {finite} free $\cO_{\zlt}$-module. Then the $\F$-vector space $H^0(\zlt,\cF)^{G}$ is naturally an $\F\bbra{N_0^{p^{-\infty}}}$-module.
\end{cor}
\begin{proof}
It is obviously an $H^0(\zlt,\cO_{\zlt})^{G}$-module, whence the result by (\ref{globalsections2}).
\end{proof}

Since $H^0(\zlt,\cF)^{G}$ is also endowed with a residual $K^\times$-action, by Corollary \ref{FN0} and the formula (\ref{action}) it becomes a $P$-representation. We also note the following lemma.

\begin{lem}\label{lm:pushforward-sheaf}
  We have $(m_*{\mathcal O}_{\zlt})^{G}  \simeq {\mathcal O}_{\zok}$.
\end{lem}
\begin{proof}
Let $U\subseteq \zok$ be an open subset and $V:= U\times_{\zok}\zlt$. Then $\underline G$ acts on the perfectoid space $V$ and there is a $\underline G$-equivariant isomorphism $h_V\simeq h_U\times_{h_{\zok}}h_{\zlt}$ in ${{\rm Perf}_{\F}}^{\!\!\!\sim}$, where $\underline G$ acts on $h_U\times_{h_{\zok}}h_{\zlt}$ through its action on $h_{\zlt}$. We deduce $h_V/\underline G\simeq h_U\times_{h_{\zok}}(h_{\zlt}/\underline G) \simeq h_U$ in ${{\rm Perf}_{\F}}^{\!\!\!\sim}$, where the last isomorphism follows from (\ref{isosheaves}). The same proof as in Proposition \ref{globalsections} then gives $H^0(U,{\mathcal O}_U)\simeq H^0(V,{\mathcal O}_V)^{G}$.
\end{proof}

We show that the horizontal isomorphisms in Lemma~\ref{globalsections} are homeomorphisms. Consider the diagram
\begin{equation*}
\xymatrix{\bigoplus_{i=0}^1 H^0\big(m^{-1}(U_i),{\mathcal O}_{\zlt}\big)^G&\bigoplus_{i=0}^1 H^0\big(U_i,{\mathcal O}_{\zok}\big)\ar^(0.43){\sim}[l]\\
H^0\big(\zlt,{\mathcal O}_{\zlt}\big)^G\ar@{^{(}->}[u]&H^0\big(\zok,{\mathcal O}_{\zok}\big)\ar^{\sim}[l]\ar@{^{(}->}[u]}
\end{equation*}
for the open cover $U_0$, $U_1$ of $\zok$ in~\eqref{eq:U_i}.
The vertical maps are closed embeddings by the sheaf axiom, and the horizontal maps are isomorphisms by Lemma~\ref{lm:pushforward-sheaf}.
It then suffices to show that the top horizontal map is a homeomorphism.
Moreover, by Lemma~\ref{lem:preimage-m} below we can write $m^{-1}(U_i) = \bigcup_{n=0}^\infty W_n$ with $W_n$ affinoid so that $H^0(m^{-1}(U_i),{\mathcal O}_{\zlt}) = \vplim_n H^0(W_n,\cO_{\zlt})$ is a Fr\'echet space over $\F\ppar{Y_i^{p^{-\infty}}}$ (each $H^0(W_n,\cO_{\zlt})$ being Banach).
Hence its closed subspace $H^0(m^{-1}(U_i),{\mathcal O}_{\zlt})^G$ is also Fr\'echet, and $H^0(U_i,{\mathcal O}_{\zok})$ is even Banach. By the open mapping theorem (applied separately to each direct factor), the top horizontal map is a homeomorphism.

{\begin{rem}
An analogous proof using Lemma \ref{lem:preimage-m} below shows that Lemma \ref{lm:pushforward-sheaf} is an isomorphism of sheaves of topological rings.
\end{rem}}

\begin{lem}\label{lem:preimage-m}
  Suppose that $U$ is any quasi-compact open subset of $\zok$.
  Then $m^{-1}(U)$ can be covered by countably many affinoid open subsets of $\zlt$.
\end{lem}
\begin{proof}
Write $m:\zlt\twoheadrightarrow \zok$ as the composition $\zlt \buildrel{m_1}\over \twoheadrightarrow \zlt/(p,p^{-1})^{\Z} \buildrel{m_2}\over \twoheadrightarrow \zok$. Then $m_1$ is a $(p,p^{-1})^{\Z}$-torsor between perfectoid spaces, {with $(p,p^{-1})^\Z$ acting properly discontinuously on $\zlt$ {(using the map $\kappa_{12}$ in the proof of \cite[Prop.\ 2.4.4]{BHHMS3})} and we claim that} $m_2$ a quasi-compact morphism of perfectoid spaces. {Assuming the claim,} $m_2^{-1}(U)$ is a quasi-compact perfectoid space, and this implies that $m^{-1}(U)=m_1^{-1}(m_2^{-1}(U))$ can be covered by countably many quasi-compact open subsets, or equivalently countably many affinoid open subsets.

To prove the claim, we first show that the rational open subset $W := \{ |T_0|^{q^2} \le |T_1| \le |T_0| \}$ of $\zlt$ surjects onto $\zlt/(p,p^{-1})^{\Z}$, or equivalently that every $(p,p^{-1})^\Z$-orbit in $\zlt$ intersects $W$. Suppose that $x \in \zlt$ with maximal generization $\wt x$. Using the action of $(p,p^{-1})^{\Z}$ we may assume that $\wt x \in W$. {If $x \notin W$, then $|T_1|_{x}>|T_0|_{x}$ or $|T_1|_{x}< |T_0|_{x}^{q^2}$. By \cite[Lemma~4.2.2]{SWBerkeley} with \cite[Prop.~4.2.5]{SWBerkeley} we have (respectively) $|T_1|_{\wt x}\geq |T_0|_{\wt x}$ or $|T_1|_{\wt x}\leq |T_0|_{\wt x}^{q^2}$, which forces (respectively) $|T_1|_{\wt x}= |T_0|_{\wt x}$ or $|T_1|_{\wt x}= |T_0|_{\wt x}^{q^2}$ since $\wt x\in W$. This implies $|T_1|_{\wt x}< |T_0|^{1/q^2}_{\wt x}$ hence also $|T_1|_{x}< |T_0|^{1/q^2}_{x}$ by \emph{loc.~cit.}, or $|T_0|^{q^4}_{\wt x} < |T_1|_{\wt x}$ hence also $|T_0|^{q^4}_{x} < |T_1|_{x}$. Thus we obtain $|T_0|_{x} < |T_1|_{x} < |T_0|^{1/q^2}_{x}$ which is equivalent to $|T_0|^q_{x} < |T_1|^q_{x} < |T_0|^{1/q}_{x}$ and thus $(p,p^{-1})^{- 1}(x) \in W$, or $|T_0|^{q^4}_{x} < |T_1|_{x} < |T_0|^{q^2}_{x}$ which is equivalent to $|T_0|^{q^3}_{x} < |T_1|^{1/q}_{x} < |T_0|^q_{x}$ and thus $(p,p^{-1})(x) \in W$.}

  Consider the composition $f : W \onto \zlt/(p,p^{-1})^{\Z} \to \zok \subset \Spa \F\bbra{Y_0^{p^{-\infty}},Y_1^{p^{-\infty}}}$ which sends analytic points to analytic points. By \cite[Prop.~5.1.3]{SWBerkeley} the map $f$ is adic and by \cite[Prop.~5.1.5]{SWBerkeley} the preimage of any rational subset is rational, hence quasi-compact, in $W$. By surjectivity of the first map {(and since the continuous image of a quasi-compact subset is quasi-compact)}, the preimage of any rational subset of $\Spa \F\bbra{Y_0^{p^{-\infty}},Y_1^{p^{-\infty}}}$ is quasi-compact in $\zlt/(p,p^{-1})^{\Z}$, and hence $m_2$ is quasi-compact.
\end{proof}

\begin{rem}\label{enplus}\
We see the bottom isomorphism in (\ref{globalsections2}) as some sort of ``miracle''.
\begin{enumerate}
\item
For any $d > 1$ and any compact open subgroup $H$ of $\ker((K^\times)^d\twoheadrightarrow K^\times)$, one has $\F\bbra{T_{0},\dots,T_{d-1}}^H=\F$. As a consequence the embedding $\F\bbra{N_0}\hookrightarrow \F\bbra{T_{0}^{q^{-\infty}},\dots,T_{f-1}^{q^{-\infty}}}$ does not factor through $\F\bbra{T_{0}^{q^{-n}},\dots,T_{f-1}^{q^{-n}}}$ for any $n \ge 0$!

Indeed, we may assume that $H \subseteq (1+p\cO_K)^d$ and $n =0$. Suppose $F \!\in \F\bbra{T_{0},\dots,T_{d-1}}^H$ is non-constant. Write $F = \sum_{i=0}^\infty (\lambda_i+g_i(T_0,\dots,T_{d-2})) T_{d-1}^i$ with $\lambda_i \in \F$ and $g_i \in \F\bbra{T_{0},\dots,T_{d-2}}$ with $g_i(0)=0$. Assume first that $g_i \ne 0$ for some $i$, and let $i_0 \ge 0$ be minimal such that $g_{i_0} \ne 0$. It follows from $a(T) \equiv T \pmod{T^2}$ for $a \in 1+p\cO_K$ that $g_{i_0}$ is invariant under an open subgroup of $(\cO_K^\times)^{d - 1}$.   An easy induction on $d \ge 2$ shows that $\F\bbra{T_{0},\dots,T_{d-2}}^{H'}=\F$ for any open subgroup $H'$ of $(\cO_K^\times)^{d-1}$, using the field of norms when $d = 2$. Hence $g_{i_0} \in \F$, a contradiction. Otherwise, $F \in \F\bbra{T_{d-1}}^{H'}$ for some open subgroup $H'$ of $\cO_K^\times$, so $F \in \F$, contradiction.
\item
For any $d> 1$ one can also define $\F\bbra{T_{\text{cyc},0},\dots,T_{\text{cyc},d-1}}$ with the cyclotomic variables $T_{\text{cyc},i}$ instead of the Lubin--Tate variables $T_i$. But in that case we have $\F\bbra{T_{\text{cyc},0}^{p^{-\infty}},\dots,T_{\text{cyc},d-1}^{p^{-\infty}}}^H=\F$ for any compact open subgroup $H$ of $\ker((\Qp^\times)^d\twoheadrightarrow \Qp^\times)$. Indeed, by an argument similar to the one in the proof of Proposition \ref{embedding}, we have a $(\Qp^\times)^d$-equivariant embedding
\[\F\bbra{T_{\text{cyc},0}^{p^{-\infty}},\dots,T_{\text{cyc},d-1}^{p^{-\infty}}} \hookrightarrow \varprojlim_{\psi_0,\dots,\psi_{d-1}}\F\bbra{T_{\text{cyc},0},\dots,T_{\text{cyc},d-1}}, \]
where the projective limit on the right is induced by the usual operators $\psi_i:\F\bbra{T_{\text{cyc},i}}\twoheadrightarrow \F\bbra{T_{\text{cyc},i}}$ of the cyclotomic theory (which do not exist for the Lubin--Tate theory). It follows that, for any compact open subgroup $H$ of $\ker((\Qp^\times)^d\twoheadrightarrow \Qp^\times)$, \ $\F\bbra{T_{\text{cyc},0}^{p^{-\infty}},\dots,T_{\text{cyc},d-1}^{p^{-\infty}}}^H$ \ embeds \ into $\varprojlim_{\psi_0,\dots,\psi_{d-1}}\F\bbra{T_{\text{cyc},0},\dots,T_{\text{cyc},d-1}}^H$ which is $\F$ by (i).
\end{enumerate}
\end{rem}

Since $S_f$ obviously does not act freely on $\zlt$, the morphism $m$ is (unfortunately) {\it not} a $G$-torsor. It is also not pro-\'etale, though it is quasipro\'etale (\cite[\S~9.2]{SWBerkeley}) by \cite[Lemma~7.6]{FarguesAJ}. But $\zok^{\text{gen}}:= {\rm Spa}(A_\infty,A_\infty^\circ)$ ($A_\infty^\circ$ are the power-bounded elements) is an affinoid open subspace of $\zok$ (the subspace of valuations such that $\vert Y_0\vert = \vert Y_1\vert = \cdots = \vert Y_{f-1}\vert \ne 0$, see \cite[\S~2.4]{BHHMS3}, in fact from the proof of Lemma \ref{hartogs} one has $\zok^{\text{gen}}=U_0\cap \cdots \cap U_{f-1}$), and the map $m$ restricts to $m:\zlt^{\text{gen}}:=m^{-1}(\zok^{\text{gen}}) \rightarrow \zok^{\text{gen}}$ where it is now a $G$-torsor (\cite[Prop.~2.4.4]{BHHMS3}).

Let $n\geq 1$ and recall that to any $\rhobar:\gK\rightarrow \GL_n(\F)$ one can associate a perfectoid \'etale Lubin--Tate $(\varphi_q,\oks)$-module $D_{\text{LT},\infty}(\rhobar)$ over $\F\ppar{T_0^{q^{-\infty}}}$ (see for instance \cite[Rk.~2.1.3]{BHHMS3}). Note that $\varphi_q$ is bijective on $D_{\text{LT},\infty}(\rhobar)$, so that $\varphi_q$ and $\oks$ combine into an action of $K^\times$, where $p$ acts by $\varphi_q$. One can then define the $(K^\times)^f\rtimes S_f$-equivariant locally free $\cO_{\zlt}$-module $\cF_{\rhob}$ on the quasi-Stein $\zlt$ associated to its global sections $D_{\text{LT},\infty}(\rhobar)\widehat\otimes_{\F}\cdots\widehat\otimes_{\F}D_{\text{LT},\infty}(\rhobar)$ (where the latter is loosely defined as in (\ref{global}) remembering that $D_{\text{LT},\infty}(\rhobar)$ is an $n$-dimensional $\F\ppar{T_0^{q^{-\infty}}}$-vector space). And one has (see \cite[\S~2.7]{BHHMS3}):
\begin{equation}\label{dainfinitenseur}
D_{A_\infty}^\otimes(\rhobar):=A_\infty \otimes_A D_A^\otimes(\rhobar)=H^0(\zlt^{\text{gen}},\cF_{\rhob})^G=H^0(\zok^{\text{gen}},(m_*\cF_{\rhob})^G).
\end{equation}
Going back to $\rhobar=\o r_v(1)$, where $\o r$ is a $2$-dimensional global Galois representation such that its restriction $\o r_v$ to a decomposition group at a fixed place $v \vert p$ is sufficiently generic, we thus have $D_{A_\infty}(\pi(\o r))\cong H^0(\zlt^{\text{gen}},\cF_{\o r_v(1)})^G$.

If $\o r_v(1)\cong \smatr{\chi_1}{*}0{\chi_2}$ is reducible, similarly to (\ref{form}) we have extensions of $(K^\times)^f\rtimes S_f$-equivariant locally free $\cO_{\zlt}$-modules
\begin{equation}\label{ss}
\cF_{\o r_v(1)}\ \cong\ \Big(\cF_{\chi_1}\!\begin{xy} (0,0)*+{}="a"; (8,0)*+{}="b"; {\ar@{-}"a";"b"}\end{xy}\! \cF_{\o r_v(1)}^{\ss} \!\begin{xy} (15,0)*+{}="a"; (24,0)*+{}="b"; {\ar@{-}"a";"b"}\end{xy} \!\cF_{\chi_2}\Big).
\end{equation}

\begin{definit}\label{cF}
If $\o r_v$ is irreducible we define $\cF:=\cF_{\o r_v(1)}$ and if $\o r_v$ is reducible we define $\cF:=\cF_{\o r_v(1)}^{\ss}$. 
\end{definit}

Denoting by $\pi$ the maximal subquotient of $\pi(\o r)$ with only supersingular constituents, we see that we have
\begin{equation}\label{globalinter}
D_{A_\infty}(\pi)\cong H^0(\zlt^{\text{gen}},\cF)^G\cong H^0(\zok^{\text{gen}},(m_*\cF)^G)\cong D_{A_\infty}^\otimes(\o r_v(1))^{\ss}.
\end{equation}
The $P$-equivariant injective restriction map
\begin{equation}\label{eq:3}
 H^0(\zlt,\cF)^G \buildrel \text{res}\over \hookrightarrow H^0(\zltg,\cF)^G
\end{equation}
shows that $H^0(\zlt,\cF)^G$ can be seen as an ``integral structure'' in $D_{A_\infty}(\pi)$ preserved by $P$.
(Note that this map is injective since $\cF$ is free of finite rank and the map $H^0(\zlt,\cO_{\zlt}) \to H^0(\zltg,\cO_{\zlt})$ is injective
  {(in fact, even the restriction to the open subset denoted $U_{\un n_0} \subset \zltg$ in \cite[\S~2.4]{BHHMS3} is injective by combining
  Corollary~\ref{section} and \cite[Lemma 2.4.7]{BHHMS3}).)}
But we have seen in (\ref{square}) another ``integral structure'' in $D_{A_\infty}(\pi)$ preserved by $P$: $X=D_{A_\infty}(\pi) \cap \varprojlim_\psi D_A(\pi)^\natural$. It is therefore natural to ask the following question:
 
\begin{ques}\label{qu:intersec}
Do we have $H^0(\zlt,\cF)^G\cong D_{A_\infty}(\pi) \cap \varprojlim_\psi D_A(\pi)^\natural$ inside $D_{A_\infty}(\pi) $?
\end{ques}

A related (more important) question is:

\begin{ques}\label{density?}
If \ $H^0(\zlt,\cF)^G\subset  \varprojlim_\psi \!\!D_A(\pi)^\natural$, \ is \ $H^0(\zlt,\cF)^G$ \ dense \ in $\varprojlim_\psi \!\!D_A(\pi)^\natural$?
\end{ques}

Note first that, if $H^0(\zlt,\cF)^G$ is dense in $\varprojlim_\psi D_A(\pi)^\natural$, the proof of Lemma \ref{notfinitebis} shows that $H^0(\zlt,\cF)^G$ is not finitely generated as $\F\bbra{N_0^{p^{-\infty}}}$-module. More importantly, if $H^0(\zlt,\cF)^G$ is dense in $\varprojlim_\psi D_A(\pi)^\natural$, then $\varprojlim_\psi D_A(\pi)^\natural$ coincides with the closure of $H^0(\zlt,\cF)^G$ in $\varprojlim_\psi D_A(\pi)$ via the composition $H^0(\zlt,\cF)^G \buildrel \eqref{eq:3}\over\hookrightarrow D_{A_\infty}(\pi)\buildrel\eqref{embeddingpi} \over \hookrightarrow \varprojlim_\psi D_A(\pi)$ (note that $\varprojlim_\psi D_A(\pi)^\natural$ is closed in $\varprojlim_\psi D_A(\pi)$). It follows that $\varprojlim_\psi D_A(\pi)^\natural$ is local, hence $\pi^\vee\vert_P$ is local by Corollary \ref{limpsiiso}, hence $\pi\vert_P$ is local, hence $\pi$ is local as $\GL_2(K)$-representation (\cite[Thm.~4.4]{paskunas-restriction}). 

Unfortunately, just like $D_{A_\infty}(\pi) \cap \varprojlim_\psi \!D_A(\pi)^\natural$, we do not even know if $H^0(\zlt,\cF)^G$ is $0$ or not, even when $f=2$.

\begin{rem}\label{fargues}
Fargues conjectures that, for any $(K^\times)^f\rtimes S_f$-equivariant locally finite free $\cO_{\zlt}$-module $\cG$ on $\zlt$, the sheaf $(m_*\cG)^G$ is generated by its global sections $H^0(\zlt,\cG)^G$ (and in particular that we have $H^0(\zlt,\cG)^G \ne 0$ {if $\cG \ne 0$}), see \cite[Conj.~10.1, Thm.~10.2]{Fargues}.
\end{rem}

\subsection{The principal series case I}\label{sec:princ-seri-case}

We briefly discuss the analogs of Questions~\ref{qu:intersec} and~\ref{density?} in the much simpler principal series case.
Namely, we assume that $\o r_v$ is reducible and consider the rank 1 subsheaf $\cF_{\chi_1}$ of $\cF_{\o r_v(1)}$ in (\ref{ss}), resp.\ the principal series quotient of $\pi(\o r)$.

Concretely, we write again $\o r_v(1)\cong \smatr{\chi_1}{*}0{\chi_2}$, so that we have a quotient map
\begin{equation*}
 \pi(\o r)\twoheadrightarrow J := \Ind_B^G (\chi_2^{-1} \omega \otimes \chi_1^{-1}),
\end{equation*}
which gives rise to the following commutative diagram, where the left isomorphism comes from Lemma~\ref{PS} and \cite[Lemma 2.9.6]{BHHMS3}:
\begin{equation}\label{eq:DAJ}
  \begin{gathered}
    \xymatrix{D_A(J) \ar@{^{(}->}[r]\ar@{=}^\wr[d] & D_A(\pi(\o r))\ar@{=}^\wr[d] \\
      D_A^\otimes(\chi_1) \ar@{^{(}->}[r] & D_A^\otimes(\o r_v(1))}
  \end{gathered}
\end{equation}
Recall from~\eqref{square} that we have (by definition) a cartesian diagram
\begin{equation*}
\begin{gathered}
\xymatrix{D_{A_\infty}(J)  \ar@{^{(}->}^(0.45){\theta}[r] \ar@{}[dr] | {\square} & \smash{\varprojlim_\psi}\: D_A(J) & \\
  \mathop{\vphantom{{}^2}}X \ \ar@{^{(}->}[r] \ar@{^{(}->}[u] & \varprojlim_\psi D_A(J)^\natural \ar@{^{(}->}[u] & \ar@{->>}_(0.31){\Theta}[l]J^\vee}
\end{gathered}
\end{equation*}
with the one difference that in our situation the map $\Theta : J^\vee \onto \plim_\psi D_A(J)^\natural$ has a 1-dimensional kernel by Lemma~\ref{kernel}.
We know that $X$ is a $P$-representation and an $\F\bbra{N_0^{p^{-\infty}}}$-module (by Lemma~\ref{lem:action-extends}).
We also know that $\varprojlim_\psi D_A(J)^\natural$ is topologically irreducible as $P$-representation, cf.\ \S~\ref{sec:remarks-map-pivee}.

Recall the element $\kappa \in J^\vee$ from Lemma~\ref{PS}, and let $\o\kappa \ne 0$ denote its image in $D_A(J)^\natural$.
As $\psi(\kappa) = \chi_1(p)^{-1}\kappa$ in $J^\vee$ (from the definitions) we deduce that
\begin{align*}
  \Theta(\kappa) &= (\o\kappa,\chi_1(p)\o\kappa,\chi_1(p)^2\o\kappa,\dots) \\
  &= (\o\kappa,\vp(\o\kappa),\vp^2(\o\kappa),\dots) \\
  &= \theta(\o\kappa),
\end{align*}
by the definition of $\theta$ in \S~\ref{sec:embedd-Ainfty}, so $\o\kappa \in X$.

We claim that $X = \F\bbra{N_0^{p^{-\infty}}} \o\kappa$ (inside $D_{A_\infty}(J)$).
Note that the containment ``$\supset$'' is clear, as $X$ is an $\F\bbra{N_0^{p^{-\infty}}}$-module.
For the containment ``$\subset$'', first note that $D_A(J)^\natural = \F\bbra{N_0} \o\kappa$, cf.\ the proof of Lemma~\ref{lm:DAfinitetype}.
It follows from Proposition~\ref{embedding}(ii) that if $x \in X$, then $x$ is a limit of a sequence in $\F\bbra{N_0^{1/p^n}} \o\kappa$ and hence that $x \in \F\bbra{N_0^{p^{-\infty}}} \o\kappa$.

On the other hand, we compute the image of
$$H^0(\zok,(m_*\cF_{\chi_1})^G) \into H^0(\zokg,(m_*\cF_{\chi_1})^G) = D_{A_\infty}^\otimes(\chi_1).$$
{Let $f_0,f_1$ be a basis of the Lubin--Tate $(\varphi_q,\oks)$-module $D_\LT(\o r_v(1))$ over $\F\ppar{T}$ of $\o r_v(1)$ such that $f_0=T^he_0$ is a basis of $D_\LT(\chi_1)\subset D_\LT(\o r_v(1))$ with $h,e_0$ as in \cite[Lemma 2.1.5]{BHHMS3} (applied with $d=1$ {and $\o \rho = \chi_1$}).
{Then $D_{\text{LT},\infty}(\chi_1)\widehat\otimes_{\F}D_{\text{LT},\infty}(\chi_1)$ is a direct summand of $D_{\text{LT},\infty}(\o r_v(1))\widehat\otimes_{\F}D_{\text{LT},\infty}(\o r_v(1))$ which is free of rank 1, generated by the vector $f_0 \otimes f_0$.
  One checks that the vector $f_0 \otimes f_0$ is $G$-invariant, so that $D_{A_\infty}^\otimes(\chi_1) = A_\infty (f_0 \otimes f_0)$ by (\ref{dainfinitenseur}).}
By Lemmas~\ref{lm:pushforward-sheaf} and~\ref{hartogs} we deduce $H^0(\zok,(m_*\cF_{\chi_1})^G) = \F\bbra{N_0^{p^{-\infty}}}(f_0 \otimes f_0)$.

The left vertical isomorphism in~\eqref{eq:DAJ} is not canonical, but $K\s$ acts via the same character (namely $\chi_1$) on $\o\kappa$ and $f_0 \otimes f_0$, so by \cite[Cor.\ 3.9]{BHHMS2} we deduce that $\o\kappa \in \F(f_0 \otimes f_0)$.
It follows that $H^0(\zok,(m_*\cF_{\chi_1})^G) = X$ via the left vertical isomorphism in~\eqref{eq:DAJ}.
Overall we obtain a positive answer to the analogs of Questions~\ref{qu:intersec}, \ref{density?} in this case (using that $\varprojlim_\psi D_A(J)^\natural$ is topologically irreducible for the second question).

Since we were not successful with {$H^0(\zok,(m_*\cF_{\o r_v(1)}^{\ss})^G)=H^0(\zlt,\cF_{\o r_v(1)}^{\ss})^G = H^0(G,\cF_{\o r_v(1)}^{\ss}(\zlt))$ we {next} pass to $H^1_{\mathrm{cts}}(G,\cF_{\o r_v(1)}^{\ss}(\zlt))$ (i.e.,\ $H^1_{\mathrm{cts}}(G,\cF(\zlt))$ with $\cF$ as in Definition \ref{cF})}.

\section{Towards the locality of \texorpdfstring{$\pi$}{pi} via perfectoids II}\label{perfII}

For $\pi$ of global origin having only supersingular constituents as in (\ref{piglobal}), we construct a $P$-equivariant surjection $D_{A_\infty}(\pi)\cong D_{A_\infty}^\otimes(\o r_v(1))^{\ss} \twoheadrightarrow \pi \otimes \eta$ (for a certain local twist $\eta$). When $f=2$ we construct a local quotient of $D_{A_\infty}^\otimes(\o r_v(1))^{\ss}$ by using the geometric interpretation of $D_{A_\infty}^\otimes(\o r_v(1))^{\ss}$ in (\ref{globalinter}) and prove that it is an infinite-dimensional smooth $P$-representation. We hope that this quotient equals $\pi \otimes \eta$.

\subsection{The morphism \texorpdfstring{$A_\infty \otimes_A \Hom_A(D_A(\pi),A) \to \pi \otimes \chi$}{Ainfty otimes Hom\_A(D\_A(pi),A) -> pi otimes chi}}\label{sec:surj-Ainfty-pi}

For a smooth representation $\pi$ of $\GL_2(K)$ that is contained in category $\mathcal C$ (\S~\ref{natural}) we construct a $P$-equivariant morphism $A_\infty \otimes_A \Hom_A(D_A(\pi),A) \to \pi$ (up to twist).

We make no assumption on $f$ and let $\pi$ be an admissible smooth representation of $\GL_2(K)$ over $\F$ in the category $\mathcal C$ {(\S~\ref{natural})} such that $D_A(\pi)$ is \'etale. By \cite[Lemma 3.2.2]{BHHMS3} $\Hom_A(D_A(\pi),A)$ is also naturally a finite free \'etale $(\vp,\oks)$-module over $A$. Recall that $P^+ = \smatr{\ok\backslash\{0\}}{\ok}{0}{1}$ acts on any \'etale $(\vp,\oks)$-module over $A$, for instance on $\Hom_A(D_A(\pi),A)$, by the formulas in (\ref{action}).

\begin{rem}\label{remglobal2}
One can avoid the assumption that $D_A(\pi)$ is \'etale by replacing $D_A(\pi)$ by its quotient $D_A(\pi)^{\et}$ (see the beginning of \S~\ref{natural}) in all the results of this section. Recall that in the cases of global interest we have $D_A(\pi) = D_A(\pi)^{\et}$, see~Remark~\ref{remglobal}.
\end{rem}

By duality the natural map $\pi^\vee \to D_A(\pi)$ gives rise to a map $\Hom_\F^{\cont}(D_A(\pi),\F) \to \Hom_\F^{\cont}(\pi^\vee,\F) = \pi$, which we can further compose to obtain
\begin{equation*}
  \epsilon : \Hom_A(D_A(\pi),A) \xrightarrow{\mu\circ} \Hom_\F^{\cont}(D_A(\pi),\F) \to \Hom_\F^{\cont}(\pi^\vee,\F) = \pi,
\end{equation*}
where the continuous morphism $\mu : A \to \F$ was defined in \cite[\S~3.3]{BHHMS3}.
However the actions of $\vp$ and $\oks$ become twisted.
More precisely,  we let $\chi : P=\smatr{K^\times}{K}{0}{1} \longrightarrow \F\s$ denote the smooth character \[\chi(\smatr{p^n a}{b}{0}{1}) = (-1)^{(f-1)n} N_{\mathbb{F}_q/\mathbb{F}_p} (\o a)^{-1}\] for $n \in \Z$, $a \in \oks$, $b \in K$, where $N_{\Fq/\Fp}$ is the norm $\Fq\rightarrow \Fp(\hookrightarrow \F)$.

\begin{lm}\label{lm:cts-P0-mor}
  The above morphism
  \[\epsilon: \Hom_A(D_A(\pi),A) \longrightarrow \pi \otimes \chi\]
  is $\F\bbra{N_0}$-linear, $P^+$-equivariant, and continuous, where $\pi$ carries the discrete topology.
\end{lm}
\begin{proof}
By \cite[Lemma 3.3.5(ii),(iii)]{BHHMS3} and the definition of $\chi$ we have $\epsilon(\vp(h)) = \smatr{p}{0}{0}{1}\epsilon(h)$ and $\epsilon(ah) = \smatr{a}{0}{0}{1}\epsilon(h)$ for $a \in \oks$, hence $\epsilon$ is $P^+$-equivariant.
By construction and Remark~\ref{rk:action-FN0}, $\epsilon$ is $\F\bbra{N_0}$-linear (as $\chi$ is trivial on $N_0$). 

To justify continuity: as the source is a finite free $A$-module, it suffices to show that $A \rightarrow \pi\otimes \chi$, $a \mapsto \epsilon(ah)$ is continuous for any $h \in \Hom_A(D_A(\pi),A)$. As the natural filtration on $D_A(\pi)$ is good (\cite[\S~I.5]{LiOy}), we have
\[\Hom_A(D_A(\pi),A)\!=\!\HOM_A(D_A(\pi),A) \!=\! \HOM_{\F\bbra{N_0}_S}((\pi^\vee)_S,A) \!=\! \HOM_{\F\bbra{N_0}}(\pi^\vee,A),\]
where $\HOM$ is as in \cite[\S~I.2.5]{LiOy} and $S$ is the multiplicative subset of $\F\bbra{N_0}$ generated by $Y_0\cdots Y_{f-1}$ (see \cite[\S~3.1.1]{BHHMS2}). See \cite[Prop.~I.6.6]{LiOy} for the first equality, the second equality follows by completion and the last by localization and checking compatibility with filtrations by unraveling definitions. In particular, $h(D_A(\pi)^\natural) \subset F_d A$ for some $d \in \Z$ (recall $A$ is a filtered ring with ascending filtration $(F_dA)_{d\in \Z}$) with $D_A(\pi)^\natural$ as in (\ref{pinatural}). As $\mu$ is continuous, $\mu(F_e A) = 0$ for some $e \in \Z$. We deduce from the definitions that $\epsilon(F_{e-d} A\cdot h) = 0$, i.e., $\epsilon$ is continuous.
\end{proof}

Recall that $\pi^\vee$ has dense image in $D_A(\pi)$. Hence if $D_A(\pi)$ is nonzero then the morphism $\epsilon$ is also nonzero (as $\epsilon(h)\ne 0$ for any $h\in \Hom_A(D_A(\pi),A)$ such that $h(D_A(\pi))=A$). As in the proof of Lemma \ref{P}, we deduce from Lemma \ref{lm:cts-P0-mor} and the formulas in (\ref{action}) a $P$-equivariant morphism
\begin{equation}\label{eq:lim-eps}
\varinjlim_\vp \epsilon : A_{(\infty)}\otimes_A \Hom_A(D_A(\pi),A) \longrightarrow  \pi \otimes \chi
\end{equation}
which is nonzero if $D_A(\pi)\ne 0$.

\begin{lm}
The $P$-equivariant morphism (\ref{eq:lim-eps}) induces a continuous $\F\bbra{N_0^{p^{-\infty}}}$-linear and $P$-equivariant morphism
\begin{equation}\label{eq:1}
A_\infty \otimes_A \Hom_A(D_A(\pi),A) \longrightarrow \pi \otimes \chi.
\end{equation}
\end{lm}
\begin{proof}
From (\ref{eq:diag-D}), the restriction of~(\ref{eq:lim-eps}) to $A^{1/p^n} \otimes_A \Hom_A(D_A(\pi),A)$ is given by the map $\epsilon_n$ in the following diagram:
\begin{equation}\label{eq:level-n}
\begin{gathered}
\xymatrix{
\Hom_A(D_A(\pi),A) \ar^(0.65){\epsilon}[r] & \pi \otimes \chi \\
A^{1/p^n} \otimes_A \Hom_A(D_A(\pi),A) \ar^{\wr}_{\vp^n \otimes \vp^n}[u]\ar^(0.75){\epsilon_n}[r] & \pi \otimes \chi.\ar^{\wr}_{\vp^n = \smatr{p^n}{0}{0}{1}}[u] }
\end{gathered}
\end{equation}
We show that $\varinjlim_\vp \epsilon : A_{(\infty)} \otimes_A \Hom_A(D_A(\pi),A) \to \pi \otimes \chi$ is continuous, where $\pi$ carries the discrete topology, so that it extends to (\ref{eq:1}) by completion. Note that the topology of $A_{(\infty)}$ is still given by the total degree, i.e.,\ by the filtration $F_d A_{(\infty)} = \varinjlim_{n \ge 0} (F_{p^n d} A)^{1/p^n}$ ($d\in \Z$). As in the proof of Lemma~\ref{lm:cts-P0-mor} it suffices to show that for any fixed $h \in \Hom_A(D_A(\pi),A)$ the map $A_{(\infty)} \to \pi \otimes \chi$, $a \mapsto (\varinjlim_\vp \epsilon)(ah)$ is continuous.

By picking a basis $(e_i)_{i=1}^m$ of $D := \Hom_A(D_A(\pi),A)$ and corresponding filtration $F_d D := \sum_{i=1}^m F_d A \cdot e_i$ we see that there exists some $c \in \Z$ such that $\vp(F_d D) \subset F_{pd+c}D$ for all $d \in \Z$. Choose $d(h) \in \Z$ such that $h \in F_{d(h)}D$. Choose $e \in \Z$ such that $\epsilon(F_e D) = 0$ (by Lemma~\ref{lm:cts-P0-mor}) and $\delta \in \Z$ so that $p^n(d(h)+\delta)+\frac{(p^n-1)c}{p-1} \le e$ for all $n \ge 0$. We claim that $(\varinjlim_\vp\epsilon)(F_\delta A_{(\infty)} \otimes h) = 0$, which will complete the proof. It suffices to show that $\epsilon_n((F_{p^n\delta} A)^{1/p^n} \otimes h) = 0$ for all $n \ge 0$, or equivalently that $\epsilon(F_{p^n\delta} A \cdot \vp^n(h)) = 0$, cf.~(\ref{eq:level-n}). By induction we have $\vp^n(h) \in F_{p^n d(h) + \frac{(p^n-1)c}{p-1}} D$, thus by our choice of $\delta$ the map (\ref{eq:1}) is continuous.

By continuity, the map (\ref{eq:1}) is $P$-equivariant. Also, from \eqref{eq:level-n} the map $\epsilon_n$ is $\F\bbra{N_0^{1/p^n}}$-linear and (\ref{eq:1}) is $\varinjlim_n \F\bbra{N_0^{1/p^n}}$-linear. Recall $\F\bbra{N_0^{p^{-\infty}}} \subset A_\infty$, consisting of series in the $Y_i$ with non-negative coefficients and total degree going to $\infty$. By smoothness, any element of $\pi$ is killed by terms of large enough degree, so $\pi$ becomes a $\F\bbra{N_0^{p^{-\infty}}}$-module. By continuity, (\ref{eq:1}) is $\F\bbra{N_0^{p^{-\infty}}}$-linear.
\end{proof}

Assume now moreover that $\pi$ has finite length and that all its irreducible constituents $\pi'$ are supersingular with $D_A(\pi')\ne 0$. Note that using the exactness of the functor $\pi''\mapsto D_A(\pi'')$ (\cite[Prop.~3.12]{BHHMS2}), all $D_A(\pi')$ are automatically \'etale as $D_A(\pi)$ is. Then since each $\pi'\vert_P$ is irreducible by \cite[Thm.~1.1(i)]{paskunas-restriction} (with the fact that $\pi$, hence $\pi'$, has a central character), we deduce that the morphism (\ref{eq:1}) is surjective in that case.

Finally, let $\pi$ as in (\ref{piglobal}). Since $\Hom_A(D_A^\otimes(\o r_v(1)),A)\cong D_A^\otimes(\o r_v^\vee(-1))$ (see \cite[end of \S~5]{YW3}) and $\o r_v^\vee\cong \o r_v\otimes {\det}(\o r_v)^{-1}$, it easily follows from \cite[Lemma 2.9.6]{BHHMS3} that the $(\vp,\oks)$-module $\Hom_A(D_A^\otimes(\o r_v(1))^{\ss},A)$ is isomorphic to $D_A^\otimes(\o r_v(1))^{\ss}$ up to a twist. Therefore, from the above results with (\ref{piglobal}), we obtain a continuous $P$-equivariant surjection
\begin{equation}\label{surj}
D_{A_\infty}^\otimes(\o r_v(1))^{\ss}\twoheadrightarrow \pi \otimes \eta
\end{equation}
for some explicit (local) character $\eta:P\rightarrow \F^\times$.

\begin{rem}
The morphism (\ref{eq:1}) is not surjective when $\pi$ is a principal series, see \S~\ref{sec:princ-seri-case2}.
\end{rem}

In the next sections we assume $f=2$ and our aim is to find a local smooth $P$-representation occurring as a quotient of the $P$-representation $D_{A_\infty}^\otimes(\o r_v(1))^{\ss}$.

\subsection{A smoothness result}\label{sec:smoothness-result}

{ For a continuous representation $\rhobar:\gK\rightarrow \GL_2(\F)$, recall $\cF_{\rhob}$ is the $(K\s)^f\rtimes S_f$-equivariant locally free sheaf on $\zlt$ defined in \S~\ref{geometricinter}. We define $D_A^\otimes(\rhob)^{\ss}$ as in (\ref{form}) and $\cF$ as in Definition \ref{cF}, replacing $\o r_v(1)$ by $\rhobar$ in \emph{loc.~cit.} In this section we assume $f=2$ and prove that the natural action of $P=\smatr{K^\times}{K}{0}{1}$ on $H^1(p^{\Z,\mathrm{a}},H^0(\zlt,\cF))^{G}$ is smooth (the ``a'' stands for ``antidiagonal''). We will use the latter space to construct a local smooth quotient of $D_{A_\infty}^\otimes(\rhob)^{\ss}=A_\infty\otimes_AD_{A}^\otimes(\rhob)^{\ss}$.}

\begin{lm}\label{lm:p-1-zlt}
  Let $\lambda \in \F\s$, $t \in \Z_{\ge 1}$, $r,s \in \Z[1/p]$, and $F = \sum_{n=0}^\infty \lambda_n T_0^{d_{0,n}} T_1^{d_{1,n}} \in H^0(\zlt)$ such that for all $n$ we have either $d_{0,n} > 0$ or $d_{1,n} > 0$.
  Then
  $$F \in ((p^t)^{\textnormal{a}}-\lambda^{-1}T_0^{-r}T_1^{-s})H^0(\zlt).$$
\end{lm}

{Here (and in what follows), given $x \in K^\times$, we write $x^{\mathrm{a}}$ for the antidiagonal element $(x,x^{-1}) \in G$.}

\begin{proof}
Let $I := \{ n \ge 0 : d_{0,n} \ge d_{1,n} \}$.
For $n \in I$ we let
\begin{equation}\label{eq:g_n}
  g_n := -\sum_{i=0}^{+\infty} (\lambda T_0^r T_1^s)^{i+1}(p^{it})^{\textrm{a}}(T_0^{d_{0,n}} T_1^{d_{1,n}}) = -\sum_{i=0}^{+\infty} \lambda^{i+1}T_0^{q^{it} d_{0,n}+(i+1)r} T_1^{q^{-it} d_{1,n}+(i+1)s}.
\end{equation}
Then $g_n$ converges in $H^0(\zlt)$ as $d_{0,n} > 0$ (by assumption).
Similarly for $n \notin I$ put
\begin{equation*}
  g_n := \sum_{i < 0} (\lambda T_0^rT_1^s)^{i+1}(p^{it})^{\textrm{a}}(T_0^{d_{0,n}} T_1^{d_{1,n}}) = \sum_{i=1}^{+\infty} (\lambda T_0^rT_1^s)^{1-i} T_0^{q^{-it} d_{0,n}} T_1^{q^{it} d_{1,n}}.
\end{equation*}
Again this converges in $H^0(\zlt)$, and in either case we have $((p^t)^{\textnormal{a}}-\lambda^{-1}T_0^{-r}T_1^{-s})g_n = T_0^{d_{0,n}} T_1^{d_{1,n}}$.

It remains to check that $g := \sum_{n=0}^\infty \lambda_n g_n \in \F\bbra{T_0^{q^{-\infty}},T_1^{q^{-\infty}}}$ converges in $H^0(\zlt)$ because this implies that $((p^t)^{\textnormal{a}}-\lambda^{-1}T_0^{-r}T_1^{-s})g = F$.
Without loss of generality we assume that $I = \Z_{\ge 0}$ -- as the argument for indices $n \notin I$ is analogous (by symmetry) -- so that in particular $d_{0,n} (\ge \frac 12(d_{0,n}+d_{1,n})) \to \infty$ as $n\to \infty$.

For $x, y \in \R_{> 0}$ we let $v_{x,y}$ denote the valuation on $H^0(\zlt)$ given by $v_{x,y}(F) := \min\{xd_{0,n}+yd_{1,n} : n \ge 0 \}$ (here we implicitly assume that all $\lambda_n$ are nonzero and that the $(d_{0,n},d_{1,n})$ are pairwise distinct).
From~\eqref{eq:g_n} we obtain, with $i$ running through $\Z_{\ge 0}$:
\begin{align*}
  v_{x,y}(g_n) &= \min\{ x[q^{it} d_{0,n} + (i+1)r]+y[q^{-it}d_{1,n}+(i+1)s] \} \\
  &\ge \min\{ x[(1+(q-1)it) d_{0,n}+(i+1)r] + y[q^{-it}d_{1,n}+(i+1)s] \} \\
  &= \min\{ [xd_{0,n}+yq^{-it}d_{1,n}+xr+ys] + i[x(q-1)td_{0,n}+xr+ys] \}.
\end{align*}
Choosing $n$ large enough such that $x(q-1)td_{0,n}+xr+ys \ge 0$, we get
\begin{equation*}
  v_{x,y}(g_n) \ge xd_{0,n} + \min\{ yq^{-it}d_{1,n} \} +xr+ys.
\end{equation*}
Hence if $d_{1,n} \ge 0$, then $v_{x,y}(g_n) \ge xd_{0,n}+xr+ys$, whereas if $d_{1,n} < 0$ then $v_{x,y}(g_n) \ge xd_{0,n}+yd_{1,n}+xr+ys$.
In either case we see that $v_{x,y}(g_n) \to \infty$ as $n \to \infty$.
\end{proof}

\begin{cor}\label{cor:smoothness-gp-ring}
For $\lambda \in \F\s$, $t \in \Z_{\ge 1}$, $r,s \in \Z[1/p]$ the action of $H^0(Z) = \F\bbra{T_0^{q^{-\infty}},T_1^{q^{-\infty}}}$ on $H^0(\zlt)/((p^t)^{\textnormal{a}}-\lambda^{-1}T_0^{-r}T_1^{-s})H^0(\zlt)$ is smooth, in the sense that for each element $F \in H^0(\zlt)/((p^t)^{\textnormal{a}}-\lambda^{-1}T_0^{-r}T_1^{-s})H^0(\zlt)$ there exists $N \in \Z_{\ge 0}$ such that the ideal $(T_0,T_1)^N$ kills $F$.
\end{cor}
\begin{proof}
Take any $F = \sum_{n=0}^\infty \lambda_n T_0^{d_{0,n}} T_1^{d_{1,n}} \in H^0(\zlt)$.
In particular, $d_{0,n}+d_{1,n} \to \infty$ as $n \to \infty$.
Choosing $N \ge 0$ such that $d_{0,n}+d_{1,n} >-N$ for all $n$, we see that $hF$ satisfies the conditions of Lemma \ref{lm:p-1-zlt} for any element $h\in (T_0,T_1)^N$.
\end{proof}

\begin{lm}\label{lm:smoothness-diagonal}
For $\lambda \in \F\s$, $t \in \Z_{\ge 1}$, $r,s \in \Z[1/p]$ the action of $\oks \times \oks$ on $H^0(\zlt)/((p^t)^{\textnormal{a}}-\lambda^{-1}T_0^{-r}T_1^{-s})H^0(\zlt)$ is smooth.
\end{lm}
\begin{proof}
For any $N \ge 1$, the Lubin--Tate action of $\oks$ on $\F\bbra{T}/T^N$ is smooth ($T$ is the Lubin--Tate variable as in \S~\ref{geometricinter}).
Indeed, by induction, note that if $[a](T) \in T + T^d \F\bbra{T}$ for $a \in 1+p\cO_K$ and $d > 1$, then $[a^p](T) \in T+T^{d+1} \F\bbra{T}$.
Choose then $i(N) \in \Z_{> 0}$ such that for $a \in 1+p^{i(N)} \ok$ we have $a(T) \in T(1+T^{N-1} \F\bbra{T})$.
By Lemma~\ref{lm:p-1-zlt} any element of $H^0(\zlt)/((p^t)^{\textnormal{a}}-\lambda^{-1}T_0^{-r}T_1^{-s})H^0(\zlt)$ is represented by a \emph{polynomial} in $T_0^{-q^{-n}}$, $T_1^{-q^{-n}}$ (for some $n$).
It thus suffices to check that the action of $(\oks)^2$ on $T_0^{-r/q^n}T_1^{-s/q^n} + ((p^t)^{\textnormal{a}}-\lambda^{-1}T_0^{-r}T_1^{-s})H^0(\zlt)$ is smooth, for any fixed $r,s \in \Z_{\ge 0}$ and $n \ge 0$.
By Lemma~\ref{lm:p-1-zlt} this is true, because $(1+p^{i(r+2)} \ok) \times (1+p^{i(s+2)} \ok)$ acts trivially on this coset.
\end{proof}

\subsubsection{The case of reducible \texorpdfstring{$\rhob$}{rho}}\label{sec:case-reducible-rhob}

Recall that $\omega_{nf}$ for $n \ge 1$ denotes Serre's niveau $n$ character of order $q^n-1$ (which is $\F$-valued via any embedding $\F_{q^n}\into \F$ that extends $\sigma_0:\Fq\into \F$, the choice of which does not matter).

\begin{lm}\label{lm:convenient-basis}\
\begin{enumerate}
  \item
Suppose that $\rhob \cong \smatr{\omega_2^{h_1} \mathrm{unr}(\lambda_1) }{0}0{\omega_2^{h_2}\mathrm{unr}(\lambda_2)}$ is split reducible. We have an isomorphism $H^0(\zlt, \cF_{\rhob}) \cong H^0(\zlt)^{\oplus 4}$ as $H^0(\zlt)$-modules such that the antidiagonal action of $p$, equivalently of $\vp_q$, corresponds to $(\vp_q, \lambda_1\lambda_2^{-1}\vp_q, \lambda_2\lambda_1^{-1}\vp_q, \vp_q)$ and the action of $(a,1) \in \oks \times \oks$ corresponds to $(\o a^{h_1} [a] \otimes 1,\o a^{h_1}[a] \otimes 1,\o a^{h_2}[a] \otimes 1,\o a^{h_2}[a] \otimes 1)$.
 \item
{Suppose that $\rhob \cong \smatr{\omega_2^{h_1} \mathrm{unr}(\lambda_1) }{*}0{\omega_2^{h_2}\mathrm{unr}(\lambda_2)}$ is reducible.} We have $H^0(\zlt,\cF) \cong H^0(\zlt)^{\oplus 2}$ \ as \ $H^0(\zlt)$-modules \ with \ antidiagonal \ action \ of \ $p$ \ given \ by $(\lambda_1\lambda_2^{-1}\vp_q, \lambda_2\lambda_1^{-1}\vp_q)$ and action of $(a,1) \in \oks \times \oks$ by $(\o a^{h_1}[a] \otimes 1,\o a^{h_2}[a] \otimes 1)$.
\end{enumerate}
\end{lm}
\begin{proof}
(i) Let $D_{\text{LT}}(\rhob)$ denote the Lubin--Tate $(\vp_q,\oks)$-module over $\F\ppar{T}$ associated to $\rhob$.
Then by changing the basis from $e_0$ to $T^h e_0$ in \cite[Lemma 2.1.5]{BHHMS3} {(applied with $d=1$)}, we see that $D_{\text{LT}}(\rhob) \cong \F\ppar{T}^{\oplus 2}$ such that $\vp_q$ (resp.\ $a \in \oks$) on the left corresponds to $(\lambda_1 \vp_q,\lambda_2 \vp_q)$ (resp.\ $(\o a^{h_1}[a],\o a^{h_2}[a])$) on the right. The result follows. (ii) also easily follows since, if $\rhob$ is nonsplit, one verifies that the sheaf $\cF$ is the \emph{same} as in the split case (for the semisimplification of $\o r_v(1)$).
\end{proof}

Recall from \S~\ref{geometricinter} that $G \cong K^\times \rtimes S_2$.
We recall how the group $P = \smatr{K\s}{K}{0}{1}$ acts on $H^1_{\mathrm{cts}}(G,H^0(\zlt,\cF_{\rhob}))$ and $H^1(p^{\Z,\textrm{a}},H^0(\zlt,\cF_{\rhob}))^{G}$.
The point is that both have actions of the ring $H^0(\zlt)^{G}$ and of the group $K\s = ((K\s)^2 \rtimes S_2)/G$.
From (\ref{globalsections2}) we have $H^0(\zlt)^{G} \cong \F\bbra{N_0^{p^{-\infty}}}$, thus we get an action of the group $\bigcup_{n \ge 0} N_0^{1/p^n} = \smatr{1}{K}{0}{1}$.
As $K\s$ acts semilinearly, we get an action of the semidirect product $P \cong \smatr{1}{K}{0}{1} \rtimes K\s$.

\begin{cor}\label{cor:smooth}\
\begin{enumerate}
  \item
Suppose that $\rhob : G_K \to \GL_2(\F)$ is split reducible. Then the action of $P$ on $H^1(p^{\Z,\mathrm{a}},H^0(\zlt,\cF_{\rhob}))^{G} = (H^0(\zlt,\cF_{\rhob})/(p^{\textnormal{a}}-1)H^0(\zlt,\cF_{\rhob}))^{G}$ is smooth.
  \item
{ Suppose \ that \ $\rhob : G_K \to \GL_2(\F)$ \ is \ reducible. \ Then \ the \ action \ of \ $P$ \ on $H^1(p^{\Z,\mathrm{a}},H^0(\zlt,\cF))^{G}$ is smooth.}
\end{enumerate}
\end{cor}
\begin{proof}
(i) It suffices to check separately that $N_0$ and $(1+p\cO_K) \times 1 \subset K\s\times 1$ act smoothly on $H^1(p^{\Z,\mathrm{a}},H^0(\zlt,\cF_{\rhob}))$ (hence the same is true on the $G$-invariant subspace).
  By Lemma~\ref{lm:convenient-basis} it suffices to show the same property for $H^0(\zlt)/(p^{\mathrm{a}}-\lambda)H^0(\zlt)$ for any $\lambda \in \F\s$.
  
First we show that the action of $N_0$ is smooth.
We have $\F\bbra{N_0} = \F\bbra{Y_0,Y_1}$, hence the $Y_i$ ($i = 0,1$) are topologically nilpotent.
So the image of $Y_i$ in $\F\bbra{T_0^{q^{-\infty}},T_1^{q^{-\infty}}}$ by \eqref{globalsections2} lies in the maximal ideal of $\F\bbra{T_0^{q^{-\infty}},T_1^{q^{-\infty}}}$.
Thus the smoothness follows from Corollary \ref{cor:smoothness-gp-ring} (with $t = 1$ and $r=s=0$).
Finally, the action of $(1+p\cO_K) \times 1$ is smooth by Lemma \ref{lm:smoothness-diagonal}.

{Part (ii) follows from part (i) by arguing as in Lemma \ref{lm:convenient-basis}(ii).}
\end{proof}

\subsubsection{The case of irreducible \texorpdfstring{$\rhob$}{rho}}\label{sec:case-irreducible-rhob}

We continue using the notation of \S~\ref{sec:case-reducible-rhob}.
Recall from \cite[\S~2.1]{BHHMS3} that $\rhob \cong (\ind \omega_{df}^h) \otimes \mathrm{unr}(\lambda)$ for some $d \ge 1$, $h \in \Z$ (not divisible by $\frac{q^d-1}{q^{d'}-1}$ for any proper divisor $d'|d$) and $\lambda \in \F\s$.
Here, $\ind \omega_{df}^h$ is determined by having determinant $\omega_f^h \cdot\mathrm{unr}(-1)^{d-1}$ and $(\ind\omega_{df}^h)\vert_{I_K}\cong \omega_{df}^h\oplus \omega_{df}^{qh}\oplus \cdots \oplus \omega_{df}^{q^{d-1}h}$.
For us, $d = f = 2$, so $h$ is not divisible by $q+1$.
As in \cite[\S~2.1]{BHHMS3} we let $f_{a,i}^{\LT} := \frac{\sigma_0(\o a)T_i}{a(T_i)}$ for $a \in \oks$ and $0 \le i \le f-1$.

\begin{lm}\label{lm:convenient-basis-irr}
Suppose that $\rhob \cong (\ind \omega_4^h) \otimes \mathrm{unr}(\lambda)$ is irreducible of dimension 2.
We have an isomorphism $H^0(\zlt, \cF_{\rhob}) \cong H^0(\zlt)^{\oplus 4}$ as $H^0(\zlt)$-modules such that the antidiagonal action of $p^2$, equivalently of $\vp_q^2$, corresponds to
$$((T_0T_1^{1/q^2})^{-h(q-1)} \vp_q^2, (T_0 T_1^{1/q})^{-h(q-1)}\vp_q^2, (T_0^q T_1^{1/q^2})^{-h(q-1)} \vp_q^2, (T_0^{q} T_1^{1/q})^{-h(q-1)}\vp_q^2)$$
and the action of $(a,1) \in \oks \times \oks$ corresponds to
$$((f_{a,0}^{\LT})^{h/(q+1)} [a] \otimes 1,(f_{a,0}^{\LT})^{h/(q+1)} [a] \otimes 1,(f_{a,0}^{\LT})^{hq/(q+1)} \o a^{h_2}[a] \otimes 1,(f_{a,0}^{\LT})^{hq/(q+1)} [a] \otimes 1).$$
\end{lm}
\begin{proof}
From \cite[Lemma 2.1.5]{BHHMS3} we see that $D_{\text{LT}}(\rhob) \cong \F\ppar{T}^{\oplus 2}$ such that $\vp_q^2$ (resp.\ $a \in \oks$) on the left corresponds to $(\lambda^2 T^{-h(q-1)} \vp_q^2,\lambda^2 T^{-hq(q-1)} \vp_q^2)$ (resp.\ $((f_a^{\LT})^{h/(q+1)} [a],(f_a^{\LT})^{hq/(q+1)} [a])$) on the right.
The result follows.
\end{proof}

\begin{cor}\label{cor:smooth-irr-case}
Suppose that $\rhob : G_K \to \GL_2(\F)$ is irreducible.
Then the action of $P$ on $H^1(p^{\Z,\mathrm{a}},H^0(\zlt,\cF_{\rhob}))^{G} = (H^0(\zlt,\cF_{\rhob})/(p^{\textnormal{a}}-1)H^0(\zlt,\cF_{\rhob}))^{G}$ is smooth.
\end{cor}
\begin{proof}
{As in the proof of Lemma \ref{cor:smooth}(i),} it suffices to check separately that $N_0$ and $\oks \times 1$ act smoothly on $H^1(p^{\Z,\mathrm{a}},H^0(\zlt,\cF_{\rhob}))$. As $H^1(p^{2\Z,\mathrm{a}},H^0(\zlt,\cF_{\rhob}))$ surjects onto \ $H^1(p^{\Z,\mathrm{a}},H^0(\zlt,\cF_{\rhob}))$, \ it \ suffices \ to \ show \ the \ same \ property \ for $H^1(p^{2\Z,\mathrm{a}},H^0(\zlt,\cF_{\rhob})) = H^0(\zlt,\cF_{\rhob})/((p^2)^{\textnormal{a}}-1)H^0(\zlt,\cF_{\rhob})$.

  For the smoothness of $N_0$, it again suffices to show that $\F\bbra{T_0^{q^{-\infty}},T_1^{q^{-\infty}}}$ acts smoothly on $H^0(\zlt,\cF_{\rhob})/((p^2)^{\textnormal{a}}-1)H^0(\zlt,\cF_{\rhob})$, and by Lemma~\ref{lm:convenient-basis-irr} this follows from Corollary \ref{cor:smoothness-gp-ring} (with $t = 2$, $\lambda = 1$ and suitable choices of $r$, $s$).

  Finally we show that $\oks \times 1$ acts smoothly.
  Let $e_0$, $e_1$ denote the basis of $D_{\text{LT}}(\rhob)$ used in the proof of Lemma~\ref{lm:convenient-basis-irr}.
  By Lemmas~\ref{lm:convenient-basis-irr} and~\ref{lm:convenient-basis} we can represent an element $\o f$ of $H^1(p^{2\Z,\mathrm{a}},H^0(\zlt,\cF_{\rhob}))$ by $f = \sum_{i,j=0}^1 f_{i,j}(T_0^{-q^{-n}},T_1^{-q^{-n}}) e_i \otimes e_j$ for some $n \ge 0$ and some \emph{polynomials} $f_{i,j}$.
  Choose $r,s \in \Z_{\ge 0}$ such that all exponents of $T_0^{r/q^n}T_1^{s/q^n} f_{i,j}(T_0^{-q^{-n}},T_1^{-q^{-n}})$ are nonnegative for all $i,j$.
  As in the proof of Lemma~\ref{lm:smoothness-diagonal} we see that $(1+p^{i(r+2)} \ok) \times 1$ acts smoothly on $\o f$, where the notation $i(N)$ is as in that proof.
  (Note that $a \in 1+p^{i(r+2)}\cO_K$ implies that $f_{a,0}^{\LT} \equiv 1 \pmod{T_0^{r+1}}$.)
\end{proof}

\subsection{A geometric candidate for \texorpdfstring{$A_\infty \otimes_A \Hom_A(D_A(\pi),A) \twoheadrightarrow \pi \otimes \chi$}{Ainfty otimes Hom\_A(D\_A(pi),A) -> pi otimes chi}}\label{sec:another-b-morphism}

{For $f=2$ and $\rhob$, $\cF$ as in the very beginning of \S~\ref{sec:smoothness-result}, using the geometric interpretation of $D_{A_\infty}^\otimes(\rhob)^{\ss}$ as in the last two isomorphisms of (\ref{globalinter}) we construct a $P$-equivariant map $D_{A_\infty}^\otimes(\rhob)^{\ss}\rightarrow H^1(p^{\Z,\textrm{a}},H^0(\zlt,\cF))^{G}$. By using the main result of \S~\ref{morphism} we prove that for $p>2$ its image is infinite-dimensional. In the global context of Definition \ref{cF} with $\rhobar=\o r_v(1)$ and $\pi$ as in (\ref{piglobal}) we hope that this image equals $\pi \otimes \eta$ (see (\ref{surj})).}

We suppose $f = 2$ and consider the rational open subsets $U_0 = \{ |Y_1| \le |Y_0| \ne 0 \}$, $U_1 = \{ |Y_0| \le |Y_1| \ne 0 \}$ of $\Spa \F\bbra{Y_0^{p^{-\infty}},Y_1^{p^{-\infty}}}$ in (\ref{eq:U_i}). Recall $U_0 \cup U_1 = \zok$ and $U_0 \cap U_1 = \zokg$ (see the proof of Lemma \ref{hartogs} and the paragraph below Remark \ref{enplus}).

As \ $\zlt$ \ is \ quasi-Stein, \ for \ any \ locally \ {finite} \ free \ $\cO_{\zlt}$-module \ $\cF$ \ we \ have $H^1(\zlt,\cF) = 0$ \cite[Thm.\ 2.6.5(c)]{KL2}, so by the (five-term exact sequence associated to the) Leray spectral sequence for the morphism $m:\zlt\rightarrow \zok$ in (\ref{m}) we have in particular $H^1(\zok, m_* \cF) = 0$. It follows that for any $(K^\times)^f\rtimes S_2$-equivariant locally {finite} free $\cO_{\zlt}$-module $\cF$ as in \S~\ref{geometricinter}, the Mayer--Vietoris sequence for the open cover $\zok = U_0 \cup U_1$ gives a short exact sequence of $H^0(\zok)$-modules with semilinear actions of $(K\s)^2 \rtimes S_2$:
{\small
\begin{equation}\label{MV}
0 \longrightarrow H^0(\zok,m_* \cF) \longrightarrow H^0(U_0,m_* \cF) \oplus H^0(U_1,m_* \cF) \longrightarrow H^0(\zokg,m_* \cF) \longrightarrow 0.
\end{equation}}%
(Note that the $(K^\times)^2\rtimes S_2$-action permutes the two direct summands, while the group $G \!= \!\{(a,a^{-1}) \!:\! a \in K\s\} \rtimes S_2$ preserves each summand.)
Taking $G$-invariants gives a long exact cohomology sequence by~\eqref{eq:H1cts},
\begin{multline}\label{connecting}
0 \longrightarrow H^0(\zok, (m_* \cF)^G) \longrightarrow H^0(U_0, (m_* \cF)^G) \oplus H^0(U_1, (m_* \cF)^G)\\
\longrightarrow H^0(\zokg, (m_* \cF)^G) \buildrel \delta \over \longrightarrow H^1_{\mathrm{cts}}(G,H^0(\zlt,\cF))
\end{multline}
which is $H^0(\zok)$-linear (as $G$ acts trivially on $H^0(\zok)$) and $K\s$-equivariant, in particular $P$-equivariant. (As $\cF$ is free of finite rank over $\zlt$, note that the injection in \eqref{MV} is identified with (several copies of) $0 \to H^0(\zlt,\cO_{\zlt}) \to H^0(m^{-1}(U_0),\cO_{\zlt}) \oplus H^0(m^{-1}(U_1),\cO_{\zlt})$, which is strict as $\cO_{\zlt}$ is a sheaf of topological rings.)
Finally, recall the $P$-equivariant inflation-restriction sequence for continuous $H^1$
{\small\begin{equation}\label{IR}
0\! \longrightarrow \! H^1_{\mathrm{cts}}(G/p^{\Z}, H^0(\zlt,\cF)^{p^{\Z,\textrm{a}}})\! \longrightarrow \! H^1_{\mathrm{cts}}(G, H^0(\zlt,\cF))\! \buildrel \res \over \longrightarrow \! H^1(p^{\Z,\textrm{a}},H^0(\zlt,\cF))^G
\end{equation}}%
from~\eqref{eq:infl-res} (the $P$-action is similar to the one above Corollary \ref{cor:smooth}).

Now, let $\cF$ be as in {the beginning of this section}. Using the last equality in (\ref{globalinter}) we consider the composition (with $\delta$ as in (\ref{connecting}))
\begin{equation}\label{composition}
D_{A _\infty}^\otimes(\rhob)^{\ss} \xrightarrow{\ \delta\ } H^1_{\mathrm{cts}}(G,H^0(\zlt,\cF)) \xrightarrow{\ \res\ } H^1(p^{\Z,\textrm{a}},H^0(\zlt,\cF))^{G},
\end{equation}
which is $P$-equivariant and $H^0(\zok)$-linear. Moreover the final term is smooth as a $P$-representation by Corollary~\ref{cor:smooth} and Corollary~\ref{cor:smooth-irr-case}.

Contrary to \S~\ref{perfI}, in this setting we at least have a non-trivial result of non-nullity:

\begin{thm}\label{main}
Assume $p>2$. Then the image of $D_{A _\infty}^\otimes(\rhob)^{\ss}$ in the composition (\ref{composition}) is an infinite-dimensional smooth representation of $P$.
\end{thm}

\begin{ques}\label{ques}
Is this image isomorphic to $\pi\vert_P\otimes \eta$ {for $\rhob=\o r_v(1)$, $\pi$ as in (\ref{piglobal}) and} $\eta$ as in (\ref{surj})?
\end{ques}

If the answer to Question \ref{ques} is yes, then $\pi\vert_P\otimes \eta$ is local, hence $\pi$ is local as $\GL_2(K)$-representation by \cite[Thm.~4.4]{paskunas-restriction}.

We {will} give the proof of Theorem \ref{main} \emph{assuming} Theorem \ref{AIM23} below. {To do this, we use the following result.} Let $I:=(Y_0,Y_1)$, a finitely generated ideal of the ring $\F\bbra{N_0^{p^{-\infty}}} \cong H^0(\zok)\cong \F\bbra{Y_0^{p^{-\infty}},Y_{1}^{p^{-\infty}}}$.

\begin{lem}\label{lem:invariant-complete}
    For $\lambda\in \F\s$, $t \in \Z_{\ge 1}$, $r,s \in \Z[1/p]$ the $H^0(\zok)$-module
    \begin{equation}\label{eq:invariant-complete}
      H^0(\zlt)^{(p,p^{-1})^t=\lambda^{-1}T_0^{-r}T_1^{-s}}
    \end{equation}
    is $(Y_0,Y_1)$-adically complete, and in particular is derived $(Y_0,Y_1)$-adically complete.
  \end{lem}
  Note that $H^0(\zok)$ acts via $H^0(\zok) \congto H^0(\zlt)^G = \F\bbra{T_0^{q^{-\infty}},T_1^{q^{-\infty}}}^G$.
\begin{proof}
  That the module is derived complete will be a consequence of the following two assertions by \cite[Tag \href{https://stacks.math.columbia.edu/tag/091T}{091T}]{stacks-project} and \cite[Tag \href{https://stacks.math.columbia.edu/tag/091U}{091U}]{stacks-project}:
  \begin{enumerate}
  \item There exists $N \in \Z_{\ge 0}$ such that
    \begin{equation}\label{eq:invt-contain}
      H^0(\zlt)^{(p,p^{-1})^t=\lambda^{-1}T_0^{-r}T_1^{-s}} \subset (T_0T_1)^{-N} \F\bbra{T_0^{q^{-\infty}},T_1^{q^{-\infty}}}.
    \end{equation}
  \item The module $(T_0T_1)^{-N} \F\bbra{T_0^{q^{-\infty}},T_1^{q^{-\infty}}}$ is $I$-adically complete.
  \item The module~\eqref{eq:invariant-complete} is the kernel of the map
    \begin{equation*}
      \lambda T_0^r T_1^s (p,p^{-1})-1 : (T_0T_1)^{-N} \F\bbra{T_0^{q^{-\infty}},T_1^{q^{-\infty}}} \to (T_0T_1)^{-N'} \F\bbra{T_0^{q^{-\infty}},T_1^{q^{-\infty}}}
    \end{equation*}
    where $N' \ge \max\{-Nq^t+r,-Nq^{-t}+s\}$ (arbitrary).
  \end{enumerate}
  The module~\eqref{eq:invariant-complete} is then $I$-adically complete by \cite[Tag \href{https://stacks.math.columbia.edu/tag/091T}{091T}]{stacks-project}, as it is separated by (i) and (ii).
  
  For (i), suppose that $f$ is contained in the space of invariants \eqref{eq:invariant-complete}.
  By invariance, if the monomial $T_0^a T_1^b$ occurs in $f$, then so does the monomial $T_0^{a_n} T_1^{b_n}$ for any $n \in \Z$, where $a_n := q^{nt}a + \frac{q^{nt}-1}{q^t-1}r$, $b_n := q^{-nt}b + \frac{q^{-nt}-1}{q^{-t}-1}s$.
  By Corollary~\ref{section} we deduce that $xa_n+yb_n \to \infty$ (or is constant) as $n \to \pm\infty$.
  As $a_n = q^{nt}(a+\frac r{q^t-1})-\frac r{q^t-1}$ and $b_n = q^{-nt}(b+\frac s{q^{-t}-1})-\frac s{q^{-t}-1}$ we deduce that $a+\frac r{q^t-1} \ge 0$ and $b+\frac s{q^{-t}-1} \ge 0$.
  We obtain the containment~\eqref{eq:invt-contain} for any $N \ge \max\{\frac r{q^t-1},\frac s{q^{-t}-1}\}$.

  For (ii) we may and will take $N = 0$.
  By continuity (Proposition~\ref{globalsections}) we see that $Y_i$ ($i=0,1$) is contained in the maximal ideal of $\F\bbra{T_0^{q^{-\infty}},T_1^{q^{-\infty}}}$.
  Choose $n \ge 0$ such that the minimal total degree (in $T_0$, $T_1$) of any monomial in any $Y_i$ is $\ge 2q^{-n}$.
  Hence $I = (Y_0,Y_1) \subset (T_0^{q^{-n}},T_1^{q^{-n}})$ inside $\F\bbra{T_0^{q^{-\infty}},T_1^{q^{-\infty}}}$.
  Hence $I^{2q^n} \subset (T_0,T_1)$.
  By \cite[Tag \href{https://stacks.math.columbia.edu/tag/090T}{090T}]{stacks-project} we obtain the first part of (ii).

  Part (iii) follows from (i) and (ii).
\end{proof}}

\begin{proof}[Proof of Theorem \ref{main}]
{By Corollary~\ref{cor:smooth}(ii) and Corollary~\ref{cor:smooth-irr-case} we only need to prove that the image of $D_{A _\infty}^\otimes(\rhob)^{\ss}$ is infinite-dimensional.} By combining Lemmas~\ref{lm:convenient-basis}, \ref{lm:convenient-basis-irr} with Lemma~\ref{lem:invariant-complete} we obtain that the $H^0(\zok)$-module $H^0(\zlt,\cF)^{p^{\Z,\textrm{a}}}$ is derived $I$-adically complete.
(If $\rhob$ is irreducible, note that $H^0(\zlt,\cF)^{p^{\Z,\textrm{a}}} \subset H^0(\zlt,\cF)^{p^{2\Z,\textrm{a}}}$ is a closed submodule.)
Now, as $p>2$, we have for $i\geq 0$ that
\[H^i_{\mathrm{cts}}(G/p^{\Z}, H^0(\zlt,\cF)^{p^{\Z,\textrm{a}}})\cong H^i_{\mathrm{cts}}((1+p\ok)^{\textrm{a}}, H^0(\zlt,\cF)^{p^{\Z,\textrm{a}}})^{[\Fq\s]^{\textrm{a}}\rtimes S_2}.\]
(see for instance from \cite[V.3.2.6.4]{lazard} or \cite[Lemma 3.3]{Kedlaya15} noting that $G/p^{\Z}$ is profinite $p$-analytic in the sense of \cite[\S~III.3.2.2]{lazard} and $1+p\ok$ is open in $G/p^{\Z}$). If $g_1,g_2$ are topological generators of $1+p\ok$ (using again $p>2$) and $M$ a separated complete $\F$-vector space endowed with an $H^0(\zok)$-module structure and an $H^0(\zok)$-linear continuous action of $1+p\ok$, recall from Lemma~\ref{lem:cts-coho-free-prop} that $H^i_{\mathrm{cts}}(1+p\ok, M)$ for $i\geq 0$ can be computed as the cohomology of the complex of $H^0(\zok)$-modules $[M\rightarrow M \oplus M \rightarrow M]$, where the first maps sends $m\in M$ to $(g_1-1)m \oplus (g_2-1)m$ and the second sends $m_1 \oplus m_2\in M\oplus M$ to $(g_2-1)m_1-(g_1-1)m_2$ (we actually only need $i=0,1$ in the sequel). If $M$ is moreover derived complete with respect to $I$ as $H^0(\zok)$-module, it follows from part (1) of \cite[Tag \href{https://stacks.math.columbia.edu/tag/091U}{091U}]{stacks-project} that so are the $H^i_{\mathrm{cts}}(1+p\ok, M)$ for $i\geq 0$. In particular (using \emph{loc.~cit.}~again) we deduce that the $H^0(\zok)$-modules $H^i_{\mathrm{cts}}(G/p^{\Z}, H^0(\zlt,\cF)^{p^{\Z,\textrm{a}}})$ are derived complete with respect to $I$ for $i\geq 0$.

Now, let $X:=\delta(D_{A _\infty}^\otimes(\rhob)^{\ss})\subseteq H^1_{\mathrm{cts}}(G, H^0(\zlt,\cF))$ and assume that the image $\overline X$ of $X$ in $H^1(p^{\Z,\textrm{a}},H^0(\zlt,\cF))^G$ via (\ref{IR}) is a finite-dimensional $\F$-vector space. Since $H^0(\zok)$-modules which are finite-dimensional over $\F$ are obviously $I$-adically complete, hence derived complete with respect to $I$ (part (1) of \cite[Tag \href{https://stacks.math.columbia.edu/tag/091R}{091R}]{stacks-project}), by the previous paragraph with part (1) of \cite[Tag \href{https://stacks.math.columbia.edu/tag/091U}{091U}]{stacks-project} the inverse image of $\overline X$ in $H^1_{\mathrm{cts}}(G, H^0(\zlt,\cF))$ is also derived complete with respect to $I$. Since $Y_0$ (or $Y_1$) acts invertibly on $D_{A _\infty}^\otimes(\rhob)^{\ss}$ and $X$ lies in this inverse image, it follows from part (2) of \cite[Tag \href{https://stacks.math.columbia.edu/tag/091P}{091P}]{stacks-project} that the map $\delta$ must be $0$, i.e.,~we have a short exact sequence of $H^0(\zok)$-modules
\begin{multline}\label{MVGbis}
0 \longrightarrow H^0(\zok,(m_* \cF)^G) \longrightarrow H^0(U_0,(m_* \cF)^G) \oplus H^0(U_1,(m_* \cF)^G) \\
\longrightarrow D_{A _\infty}^\otimes(\rhob)^{\ss} \longrightarrow 0.
\end{multline}

But $H^0(\zok,(m_* \cF)^G)=H^0(\zlt,\cF)^G$ is a closed $H^0(\zok)$-submodule of the $I$-adically complete $H^0(\zok)$-module $H^0(\zlt,\cF)^{p^{\Z,\textrm{a}}}$, hence is also $I$-adically complete, hence derived complete with respect to $I$. Since $Y_0$ (or $Y_1$) acts invertibly on $D_{A _\infty}^\otimes(\rhob)^{\ss}$, it then follows from part (2) of \cite[Tag \href{https://stacks.math.columbia.edu/tag/091P}{091P}]{stacks-project} applied with $A:=H^0(\zok)$, \ $f:=Y_0$ \ (or \ $Y_1$), \ $K:=H^0(\zok,(m_* \cF)^G)$ \ (in \ degree \ $0$) \ and \ $E:=D_{A _\infty}^\otimes(\rhob)^{\ss}[-1]$ that (\ref{MVGbis}) splits as a sequence of $H^0(\zok)$-modules. In particular there is an $H^0(\zok)$-linear surjection $H^0(U_0,(m_* \cF)^G) \oplus H^0(U_1,(m_* \cF)^G)\twoheadrightarrow H^0(\zok,(m_* \cF)^G)$. Since $Y_j$ acts invertibly on $H^0(U_j,(m_* \cF)^G)$ ($j=0,1$), by part (2) of \cite[Tag \href{https://stacks.math.columbia.edu/tag/091P}{091P}]{stacks-project} again, applied with $f=Y_j$, the map $H^0(U_j,(m_* \cF)^G) \rightarrow H^0(\zok,(m_* \cF)^G)$ is zero for $j=0,1$. Hence we must have $H^0(\zok,(m_* \cF)^G)=0$. From \ (\ref{MVGbis}) \ we \ deduce \ that \ the \ map \ $H^0(U_0,(m_* \cF)^G) \oplus H^0(U_1,(m_* \cF)^G) \rightarrow D_{A _\infty}^\otimes(\rhob)^{\ss}$ is an isomorphism, which contradicts Theorem \ref{AIM23} below. 

We conclude that $\overline X$ must be infinite-dimensional. As $H^1_{\mathrm{cts}}(G, H^0(\zlt,\cF))$ is a smooth $P$-representation and all maps are $P$-equivariant, this finishes the proof. 
\end{proof}

We can strengthen Theorem~\ref{main} as follows.

\begin{cor}\label{cor:main}
  Assume $p>2$. Then the image of $D_{A _\infty}^\otimes(\rhob)^{\ss}$ in the composition (\ref{composition}) is $Y_i$-divisible for $i\in \{0,1\}$, in particular it has unbounded $(Y_0,Y_1)$-torsion.
\end{cor}
\begin{proof}
As $D_{A _\infty}^\otimes(\rhob)^{\ss}$ is $Y_i$-divisible for $i\in \{0,1\}$, so is its image by any $H^0(\zok)$-linear morphism. Since the image of $D_{A _\infty}^\otimes(\rhob)^{\ss}$ in the composition (\ref{composition}) is nonzero by Theorem \ref{main}, it follows that it has unbounded $Y_i$-torsion for $i\in \{0,1\}$.
\end{proof}

Though we do not know if $H^0(\zok, (m_* \cF)^G)$ is nonzero, we can at least deduce from Theorem \ref{main} the following non-nullity statement:

\begin{cor}\label{non0bis}
Assume \ $p>2$. Then \ at \ least \ one \ of \ $H^0(\zok, (m_* \cF)^G)$, $H^1_{\mathrm{cts}}(G/p^{\Z}, H^0(\zlt,\cF)^{p^{\Z,\mathrm{a}}})$ is nonzero.
\end{cor}
\begin{proof}
Assume that both are zero. Then from (\ref{connecting}) and (\ref{IR}) we have a $P$-equivariant exact sequence
\begin{multline}\label{MVbis}
0 \longrightarrow H^0(U_0, (m_* \cF)^G) \oplus H^0(U_1, (m_* \cF)^G)\longrightarrow H^0(\zokg, (m_* \cF)^G)\\
\longrightarrow H^1(p^{\Z,\textrm{a}},H^0(\zlt,\cF))^G.
\end{multline}
Let $H$ be the image of the last map in $H^1(p^{\Z,\textrm{a}},H^0(\zlt,\cF))^G$, which is a nonzero smooth $P$-representation by Theorem \ref{main}. Let us first prove that $H^0(U_0, (m_* \cF)^G)$ is canonically an $\F\ppar{Y_1^{p^{-\infty}}}$-vector space (note that \emph{a priori} it is just an $\F\ppar{Y_0^{p^{-\infty}}}$-vector space and, using $\F\bbra{Y_1^{p^{-\infty}}}\hookrightarrow \F\ppar{Y_0^{p^{-\infty}}}\langle \big(\frac{Y_1}{Y_0}\big)^{p^{-\infty}}\rangle$, an $\F\bbra{Y_1^{p^{-\infty}}}$-module). 
{It suffices to show that $Y_1 : H^0(U_0, (m_* \cF)^G) \to H^0(U_0, (m_* \cF)^G)$ is bijective, and by~\eqref{MVGbis} we already know it is injective.} Denote by $R$ the injection in (\ref{MVbis}).
{Let $v\in H^0(U_0,(m_*\cF)^G)$ and note that $\frac{1}{Y_1}R(v)$ exists in $H^0(\zokg, (m_* \cF)^G)$.
Since $H$ is a smooth $\F\bbra{Y_0^{p^{-\infty}},Y_{1}^{p^{-\infty}}}$-module, there exists $N>0$ such that $\frac{Y_0^{N}}{Y_1}R(v)$ maps to $0$ in $H$, i.e.,~$\frac{Y_0^{N}}{Y_1}R(v)=R(v_{0}+v_{1})$ with $v_{i}\in H^0(U_i,(m_*\cF)^G)$. Multiplying by $Y_1$ we get \[Y_0^{N}R(v)=Y_1R(v_{0} + v_{1}),\text{ i.e., }R(Y_0^{N}v)=R(Y_1v_{0} + Y_1v_{1}).\]
Hence $Y_0^{N}v=Y_1v_{0} + Y_1v_{1}$, since $R$ is injective, and thus $Y_0^{N}v=Y_1v_{0}$, $Y_1v_{1}=0$.
Since $Y_0$ acts invertibly on $H^0(U_0,\cO_{U_0})$ we deduce that $v = Y_0^{-N}Y_1 v_0$ and so $Y_1$ is indeed surjective on $H^0(U_0, (m_* \cF)^G)$.
Hence $R$ becomes a morphism of $\F\ppar{Y_1^{p^{-\infty}}}$-vector spaces.
It follows that the cokernel $H$ of $R$ is an $\F\ppar{Y_1^{p^{-\infty}}}$-vector space and at the same time is a smooth $\F\bbra{Y_1^{p^{-\infty}}}$-module (by~\eqref{MVbis}).
Therefore $H = 0$, which contradicts Theorem \ref{main}.}
\end{proof}

\subsection{Aside on the trivial Galois representation}\label{sec:aside-trivial-galois}

For $f \geq 2$ we prove that the $P$-representation $H^{f - 1}(\zok,\cO_{\zok})$ is irreducible smooth. When \ $f = 2$ \ and \ $p>2$ \ we \ use \ this \ to \ study \ the \ connecting \ homomorphism $H^0(\zokg,(m_* \cO_{\zlt})^{G}) \xrightarrow{\ \delta\ } H^1_{\mathrm{cts}}(G,H^0(\zlt,\cO_{\zlt}))$ {in (\ref{connecting})}.

{Recall the affinoid perfectoid open cover $(U_i)_{i \in \{0,\ldots, f - 1\}}$ of $\zok$ from~\eqref{eq:U_i}: %
\[U_i:= \mathrm{Spa}\Big(\F\ppar{Y_i^{p^{-\infty}}}\left\langle \big(\frac{Y_j}{Y_i}\big)^{p^{-\infty}}, j\ne i\right\rangle, \F\bbra{Y_i^{p^{-\infty}}}\left\langle \big(\frac{Y_j}{Y_i}\big)^{p^{-\infty}}, j\ne i\right\rangle\Big).\]
Analogously to the proof of \cite[Lemma~2.4.2(i)]{BHHMS3}, for $0 \leq i_0 < i_1 < \ldots < i_k \leq f - 1$, we have
\begin{flushleft}
$\displaystyle{U_{i_0} \cap U_{i_1} \cap \ldots \cap U_{i_k} =}$
\end{flushleft}
\begin{flushleft}
$\displaystyle{\mathrm{Spa}\left(\F\ppar{Y_{i_0}^{p^{-\infty}}}\left\langle \big(\frac{Y_j}{Y_{i_0}}\big)^{\pm p^{-\infty}}\!\!, j = i_1, \ldots, i_k; \big(\frac{Y_j}{Y_{i_0}}\big)^{p^{-\infty}}\!\!, j \ne i_0, i_1, \ldots, i_k\right\rangle, \right.}$  
\end{flushleft}
\begin{flushright}
$\displaystyle{\left. \F\bbra{Y_{i_0}^{p^{-\infty}}}\left\langle \big(\frac{Y_j}{Y_{i_0}}\big)^{\pm p^{-\infty}}\!\!, j = i_1, \ldots, i_k; \big(\frac{Y_j}{Y_{i_0}}\big)^{p^{-\infty}}, j \ne i_0, i_1, \ldots, i_k\right\rangle \right)}$
\end{flushright}
In particular, each intersection $U_{i_0} \cap \ldots \cap U_{i_k}$ is affinoid perfectoid, and \cite[Thm.\ 2.6.5(c)]{KL2} implies that $H^j(U_{i_0} \cap \ldots \cap U_{i_k}, \cO_{\zok}) = 0$ for $j > 0$.  Consequently, by \cite[Tag \href{https://stacks.math.columbia.edu/tag/01ET}{01ET}]{stacks-project} (and \cite[Tag \href{https://stacks.math.columbia.edu/tag/01FM}{01FM}]{stacks-project}), we have for $i \geq 0$ an isomorphism
\[H^i(\zok, \cO_{\zok}) \cong \check{H}^i(\zok, \cO_{\zok})\]
where $\check{H}^i$ is the \v Cech cohomology associated to the open cover $(U_i)_{i \in \{0,\ldots, f - 1\}}$. Taking $i = f - 1$ in the above isomorphism gives
\begin{equation}\label{eq:cokerH^f-1}
H^{f - 1}(\zok,\cO_{\zok}) \cong \textnormal{coker}\left(\bigoplus_{i = 0}^{f - 1}H^0\big(\bigcap_{j \neq i}U_j, \cO_{\bigcap_{j \neq i}U_j}\big) ~\rightarrow~ H^0(\zok^{\textrm{gen}}, \cO_{\zok^{\textnormal{gen}}})\right)
\end{equation}
(recalling that $\bigcap_{i = 0}^{f - 1}U_i = \zok^{\textrm{gen}}$).  
Explicitly,
\begin{align*}
H^0(\zok^{\textrm{gen}}, \cO_{\zok^{\textnormal{gen}}}) & = A_\infty \\
&= \Big\{ \sum_{n=0}^\infty \lambda_n \un{Y}^{\un{d}_n} : ||\un{d}_n|| \to \infty \Big\}, \\
H^0\big(\bigcap_{j \neq i}U_j,\cO_{\bigcap_{j \neq i}U_j}\big) &= \F\ppar{Y_{i_0}^{p^{-\infty}}}\left\langle \big(\frac{Y_j}{Y_{i_0}}\big)^{\pm p^{-\infty}}, j \neq i, i_0; \big(\frac{Y_i}{Y_{i_0}}\big)^{p^{-\infty}}\right\rangle \\
&= \Big\{ \sum_{n=0}^\infty \lambda_n \un{Y}^{\un{d}_n}: d_{n,i} \ge 0,\ ||\un{d}_{n}|| \to \infty \Big\},
\end{align*}
where $\lambda_n \in \F$, $\un{d}_n \in \Z[1/p]^f$ (recalling the notation \eqref{Yk}), and $i_0$ is a fixed index different from $i$.
It follows by \eqref{eq:cokerH^f-1} that
\begin{equation}\label{eq:H^f-1-zok}
H^{f - 1}(\zok,\cO_{\zok}) \cong \bigoplus_{\un{d} \in (\Z[1/p]_{< 0})^f} \F \un{Y}^{\un{d}}.
\end{equation}

\begin{lem}\label{irr}
The $P$-representation $H^{f - 1}(\zok,\cO_{\zok})$ is irreducible and smooth.
\end{lem}
\begin{proof}
Let $V := H^{f - 1}(\zok,\cO_{\zok})$. From~\eqref{eq:H^f-1-zok} it is immediate that the action of $N_0$ on $V$ is smooth. We also deduce from the definition of $Y_i$ that $a(Y_i) \in Y_i + \m_{N_0}^{p^r}$ for all $a \in 1+p^r \cO_K$, $0 \le i \le f - 1$ and integers $r \ge 0$, which implies that the action of $\smatr{\oks}001$ on $V$ is smooth. This shows that $V$ is a smooth $P$-representation. To prove that $V$ is irreducible, suppose that $0 \ne v \in V$ and represent $v$ by a finite sum $\sum_{n=1}^N \lambda_n \un{Y}^{\un{d}_n} \in A_\infty$ with $\un{d}_n \in (\Z[1/p]_{< 0})^f$ and $\lambda_n \in \F\s$. We may assume that $\un{d}_1 < \un{d}_2 < \cdots < \un{d}_N$ in the lexicographic order. If $N > 1$ we multiply by $\un{Y}^{-\un{d}_2} \in \F\bbra{N_0}$ and obtain an element of $V$ that is of the same form but with $N = 1$. Thus we may assume that $v = \un{Y}^{\un{d}}$ for some $\un{d} \in (\Z[1/p]_{< 0})^f$.
Then $\smatr{p^{fn}}001 v = \un{Y}^{p^{fn}\un{d}}$, and by taking $n$ large enough and using the $N_0$-action we can obtain any monomial $\un{Y}^{\un{c}}$ with $\un{c} \in (\Z[1/p]_{< 0})^f$, proving that $V$ is irreducible.
\end{proof}}

{\begin{rem}
We can make Lemma \ref{irr} a bit more precise: we have an isomorphism of irreducible smooth $P$-representations $H^{f - 1}(\zok,\cO_{\zok}) \cong \cInd_{K^\times}^P(\chi|_{K^\times})$, where $\chi$ is as in \S~\ref{sec:surj-Ainfty-pi}.  See \S~\ref{sec:princ-seri-case2} for more details {for $f = 2$}.
\end{rem}}

We now assume $f=2$ and study the connecting morphism $H^0(\zokg,m_* \cF)^{G} \xrightarrow{\ \delta\ } H^1_{\mathrm{cts}}(G,H^0(\zlt,\cF))$ when $\cF=\cO_{\zlt}$. {From Lemma~\ref{lm:pushforward-sheaf} we know that ${(m_* \cO_{\zlt})^G} \cong \cO_{\zok}$.}
By (\ref{eq:cokerH^f-1}) for $f=2$ the Mayer--Vietoris sequence for the open cover $\zok=U_0\cup U_1$ gives a $P$-equivariant exact sequence
\begin{multline}\label{eq:MV}
0 \to H^0(\zok,\cO_{\zok}) \to H^0(U_0,\cO_{U_0})\oplus H^0(U_1,\cO_{U_1}) \to H^0(U_0 \cap U_1,\cO_{U_0 \cap U_1})\\
\to H^1(\zok,\cO_{\zok}) \to 0.
\end{multline}

From Lemma \ref{irr} and Proposition \ref{avoir} we deduce:

\begin{cor}\label{irrtriv}
{Assume $p>2$.} The $P$-representation $H^1_{\mathrm{cts}}(G,\!H^0(\zlt,\cO_{\zlt}))$ is irreducible, smooth and we have a $P$-equivariant exact sequence
\begin{multline}\label{MVG}
0 \to H^0(\zok,\cO_{\zok}) \to H^0(U_0,\cO_{U_0})\oplus H^0(U_1,\cO_{U_1}) \to H^0(U_0 \cap U_1,\cO_{U_0 \cap U_1})\\
\buildrel \delta \over \to H^1_{\mathrm{cts}}(G,H^0(\zlt,\cO_{\zlt})) \to 0
\end{multline}
with $\delta$ as in (\ref{connecting}) (for $\cF=\cO_{\zlt}$). In particular $\delta$ is surjective in that case.
\end{cor}

One can be a bit more specific. Recall from \S~\ref{sec:another-b-morphism} that we have a $P$-equivariant inflation-restriction sequence (writing again $H^0(\zlt)=H^0(\zlt,\cO_{\zlt})$):
\[0\longrightarrow H^1_{\mathrm{cts}}(G/p^{\Z, \mathrm{a}}, H^0(\zlt)^{p^{\Z,\textrm{a}}})\longrightarrow H^1_{\mathrm{cts}}(G, H^0(\zlt)) \buildrel \res \over \longrightarrow H^1(p^{\Z,\textrm{a}},H^0(\zlt))^G,\]
where the last representation is smooth (Corollary \ref{cor:smooth} for $\rhobar=1 \oplus 1$).

\begin{prop}\label{nonzero}
Assume $p> 2$. The composition
\[H^0(U_0 \cap U_1,\cO_{U_0 \cap U_1})\buildrel \delta \over \longrightarrow H^1_{\mathrm{cts}}(G, H^0(\zlt)) \buildrel \res \over \longrightarrow H^1(p^{\Z,\mathrm{a}},H^0(\zlt))^G\]
is nonzero. In particular we have $H^1_{\mathrm{cts}}(G/p^{\Z}, H^0(\zlt)^{p^{\Z,\mathrm{a}}})=0$ and an injection $H^1_{\mathrm{cts}}(G, H^0(\zlt)) \buildrel \res \over \hookrightarrow H^1(p^{\Z,\mathrm{a}},H^0(\zlt))^G$.
\end{prop}
\begin{proof}
The last sentence follows from the second with the irreducibility in Corollary \ref{irrtriv}. Let us prove the first statement. By surjectivity of the last map in (\ref{MVG}), if the composition is zero then there is a nonzero $P$-equivariant map $A_\infty \cong H^0(U_0 \cap U_1,\cO_{U_0 \cap U_1}) \rightarrow  H^1_{\mathrm{cts}}(G/p^{\Z}, H^0(\zlt)^{p^{\Z,\mathrm{a}}})$. 

Let \ $I:=(Y_0,Y_1)\subset H^0(\zok)$, as in the first paragraph of the proof of Theorem \ref{main}, we have that the $H^0(\zok)$-module $H^1_{\mathrm{cts}}(G/p^{\Z}, H^0(\zlt)^{p^{\Z,\mathrm{a}}})$ is derived complete with respect to $I$. Since $Y_0$ (or $Y_1$) is invertible in $A_\infty$, by part (2) of \cite[Tag \href{https://stacks.math.columbia.edu/tag/091P}{091P}]{stacks-project} any $H^0(\zok)$-linear map $A_\infty \rightarrow  H^1_{\mathrm{cts}}(G/p^{\Z}, H^0(\zlt)^{p^{\Z,\mathrm{a}}})$ is zero, a contradiction.
\end{proof}

\begin{rem}
Proposition \ref{nonzero} is probably still true when $p=2$. For $p>2$, even though the $P$-representation $H^1_{\mathrm{cts}}(G, H^0(\zlt))\cong H^1(\zok,\cO_{\zok})$ has an explicit description via (\ref{eq:H^f-1-zok}) and any element of $H^1(p^{\Z,\textrm{a}},H^0(\zlt))\cong H^0(\zlt)/(p^{\textrm{a}}-1)H^0(\zlt)$ is represented by a polynomial in $T_0^{-p^{-n}}$, $T_1^{-p^{-n}}$ for some $n$ (see the proof of Lemma \ref{lm:smoothness-diagonal}), it seems non-trivial to describe explicitly the $P$-equivariant injection $H^1_{\mathrm{cts}}(G, H^0(\zlt)) \buildrel \res \over \hookrightarrow H^1(p^{\Z,\textrm{a}},H^0(\zlt))$.
\end{rem}

\subsection{The morphism \texorpdfstring{$H^0(U_0,(m_*\cF)^G)\oplus H^0(U_1,(m_*\cF)^G) \rightarrow D_{A_\infty}(\pi)$}{H0(U0) oplus H0(U1) -> DAinfty(pi)}}\label{morphism}

For $f=2$ we prove that the map $H^0(U_0,(m_*\cF)^G)\oplus H^0(U_1,(m_*\cF)^G) \rightarrow H^0(U_0\cap U_1, (m_* \cF)^G)$ in (\ref{connecting}) cannot be an isomorphism.

\begin{thm}\label{AIM23}
  Suppose $f=2$ and $\cF$ any $G$-equivariant locally free sheaf of finite rank on $\zlt$.
  If $\cF \ne 0$, the map $$H^0(U_0,(m_*\cF)^G)\oplus H^0(U_1,(m_*\cF)^G) \rightarrow H^0(U_0\cap U_1, (m_* \cF)^G)$$ in (\ref{connecting}) cannot be an isomorphism.
\end{thm}

We \ start \ with \ a \ useful \ basic \ lemma \ about \ Newton \ polygons. \ We \ refer \ to \ \cite[\S~2.3]{BHHMS3} (and the references therein) for a reminder on ${\mathbf B}^+(C)$ for $C$ a perfect(oid) field containing $\F$.

\begin{lm}\label{lem:newton}
Suppose that $C$ is a perfect(oid) field containing $\F$ and $v$ a continuous rank $1$ additive valuation on $C$. If $x := \sum_{n \in \Z} [x_n]p^n$ converges in ${\mathbf B}^+(C)$, then the Newton polygon of $x$ is obtained from the maximal decreasing convex polygon that lies below the points $\{ (n,{v(x_n)}) : n \in \Z \}$ in $\R^2$ by removing all segments of slopes $0$ and $\infty$.
\end{lm}
\begin{proof}
By \cite[Ex.\ 1.6.22]{FF}, the Newton polygon of $x$ is obtained from the inverse Legendre transform of the function $\R_{> 0} \to \R$, $\lambda \mapsto v_\lambda(x)$ by removing all segments of slopes 0 and $\infty$. (Note that $\infty$ also needs to be removed! Compare with the sentence preceding Ex.\ 1.6.22 or Def.\ 1.6.18, 1.6.21 in \emph{loc.~cit.}) By continuity of $v_\lambda$ we have $v_\lambda(x) = \inf_{n \in \Z} \{v(x_n)+n\lambda\}$ for $\lambda > 0$. This is precisely the Legendre transform of the maximal decreasing convex polygon that lies below the points $(n,{v(x_n)})$ ($n \in \Z$), see \cite[\S~1.5.1]{FF}.
The lemma follows.
\end{proof}

\begin{cor}\label{cor:newton}
Keep the notation of Lemma~\ref{lem:newton}. Suppose that $x_0, x_1 \in \m_C$. Then the slopes of the Newton polygon of $x := \sum_{n\in \Z}[x_0^{p^{-2n}}]p^{2n}+\sum_{n\in \Z}[x_1^{p^{-1-2n}}]p^{1+2n}\in {\mathbf B}^+(C)^{\varphi_q=p^2}$ are given by the following.
\begin{enumerate}
  \item If $\frac{2p}{p^2+1} < v(x_0)/v(x_1) < \frac{p^2+1}{2p}$:
    \begin{multline*}
      \cdots > p^{-2n+1} v(x_1) - p^{-2n} v(x_0) > p^{-2n} v(x_0) - p^{-2n-1} v(x_1) \\
      > p^{-2n-1} v(x_1) - p^{-2n-2} v(x_0) > \cdots,
    \end{multline*}
    each with multiplicity 1.
  \item If $v(x_0)/v(x_1) \le \frac{2p}{p^2+1}$:
    \begin{equation*}
      \cdots > p^{-2n+2} v(x_0) - p^{-2n} v(x_0) > p^{-2n} v(x_0) - p^{-2n-2} v(x_0) > \cdots,
    \end{equation*}
    each with multiplicity 2.
  \item If $v(x_0)/v(x_1) \ge \frac{p^2+1}{2p}$:
    \begin{equation*}
      \cdots > p^{-2n+1} v(x_1) - p^{-2n-1} v(x_1) > p^{-2n-1} v(x_1) - p^{-2n-3} v(x_1) > \cdots,
    \end{equation*}
    each with multiplicity 2.    
\end{enumerate}
\end{cor}
(We negated slopes, just like in \cite[\S~1.5.1]{FF}.)
\begin{proof}
  By Lemma~\ref{lem:newton} we consider the points $v_{2n} := (2n,p^{-2n}v(x_0))$ and $v_{2n+1} := (2n+1,p^{-2n-1}v(x_1))$. \ The midpoint \ between \ $v_{2n}$ \ and \ $v_{2n+2}$ \ is \ given \ by $(2n+1,\frac 12 p^{-2n-2}(p^2+1)v(x_0))$ which lies strictly above $v_{2n+1}$ if and only if $v(x_0)/v(x_1) > \frac{2p}{p^2+1}$.
  Similarly, the midpoint between $v_{2n-1}$ and $v_{2n+1}$ lies above ${v_{2n}}$ if and only if $v(x_1)/v(x_0) > \frac{2p}{p^2+1}$.
  The result follows from Lemma~\ref{lem:newton}.
\end{proof}

We now prove Theorem \ref{AIM23}.

\begin{proof}
We recall the variables $X_0$, $X_1$ introduced above \cite[(37)]{BHHMS3} such that $\F\bbra{N_0} = \F\bbra{X_0,X_1}$.
For $\alpha\in (0,1]\cap \Q$ we let
\[V_\alpha:= \{|X_0| = |X_1|^\alpha \ne 0 \}\subset W_\alpha:=\{|X_1| \leq |X_0| \leq |X_1|^\alpha \ne 0 \}\]
which are open affinoid perfectoid subspaces of $U_0=\{|X_1| \leq |X_0| \ne 0 \}$ (note that $V_1=W_1=U_0\cap U_1$). For $\alpha\in [1,+\infty)\cap \Q$ we let 
\[V_\alpha:= \{|X_0| = |X_1|^\alpha \ne 0 \}\subset W_\alpha:=\{|X_1|^\alpha \leq |X_0| \leq |X_1| \ne 0 \}\]
which are open affinoid perfectoid subspaces of $U_1=\{|X_0| \leq |X_1| \ne 0 \}$. Note that $U_0\cap U_1\subset W_\alpha$ for all $\alpha$. We also remark that all $V_\alpha$, $W_\alpha$ are rational subsets of $\Spa \F\bbra{X_0^{p^{-\infty}},X_1^{p^{-\infty}}}$. (We remark that for $\frac 12 < \alpha < 2$, using \cite[(50)]{BHHMS3}, we still have $V_\alpha = \{|Y_0| = |Y_1|^\alpha \ne 0 \}$ and likewise for $W_\alpha$.) 

We divide the proof into several steps.

\textbf{Step 1.} We prove that for $\alpha\in (\frac{2p}{p^2+1},1]\cap \Q$ the map $m^{-1}(V_\alpha)\rightarrow V_\alpha$ is a pro-\'etale $G$-torsor between perfectoid spaces.

Let $C$ be a perfectoid field containing $\F$ and $v$ a continuous rank $1$ additive valuation on $C$ (hence valued in $\R \cup \{+\infty\}$). Let $x:=F(x_0,x_1):=\sum_{n\in \Z}[x_0^{p^{-2n}}]p^{2n}+\sum_{n\in \Z}[x_1^{p^{-1-2n}}]p^{1+2n}\in {\mathbf B}^+(C)^{\varphi_q=p^2}$ (where $x_i\in \cO_C$ with $v(x_i)>0$). Let $N(x) \subset \R^+\setminus\{0\}$ denote the set of slopes of the Newton polygon of $x$.
By Corollary~\ref{cor:newton} it follows that either $\frac{2p}{p^2+1} < v(x_0)/v(x_1) < \frac{p^2+1}{2p}$ and
\begin{equation}\label{eq:Nx}
N(x)=(pv(x_1)-v(x_0))p^{2\Z} \sqcup (v(x_0)-p^{-1}v(x_1))p^{2\Z}\ \subset\ \R^+\setminus\{0\},
\end{equation}
or $N(x) = cp^{2\Z}$ for some $c > 0$. Fix $\alpha\in (\frac{2p}{p^2+1},1]$. We consider the following set for $c\in \R^+\setminus\{0\}$:
\begin{equation}\label{c}
N_c:=(p\alpha^{-1}-1)cp^{2\Z} \cup (1-(p\alpha)^{-1})cp^{2\Z}\ \subset\ \R^+\setminus\{0\}.
\end{equation}
Note that the function $f : (-\infty,p) \to \R,\ x \mapsto \frac{1-(px)^{-1}}{px^{-1}-1}$ is strictly increasing, with $f(\frac{2p}{p^2+1}) = \frac 1{p^2}$, $f(1) = \frac 1p$, $f(\frac{p^2+1}{2p}) = 1$.
In particular the two subsets in (\ref{c}) are disjoint.
Moreover, we obtain
\begin{equation}\label{eq:fxy}
f(x)p^{2\Z} = f(y)p^{2\Z}\ \ \text{for $x,y \in \Big(\frac{2p}{p^2+1},\frac{p^2+1}{2p}\Big)$} \quad\Longrightarrow \quad x=y.
\end{equation}
If $x\in V_\alpha(C,\cO_C)$, then one has $N(x)=N_{v(x_0)}$, and if $x\in V_{\alpha^{-1}}(C,\cO_C)$, one has $N(x)=N_{p^{-1}v(x_1)}$. Conversely, if $x=F(x_0,x_1)\in {\mathbf B}^+(C)^{\varphi_q=p^2}$ is such that $N(x)=N_c$ for some $c>0$, then by above we first deduce that $\beta := v(x_0)/v(x_1)$ satisfies $\frac{2p}{p^2+1} < v(x_0)/v(x_1) < \frac{p^2+1}{2p}$.
By the equality of~\eqref{eq:Nx} and \eqref{c} we deduce that $f(\beta)p^{2\Z} = f(\alpha)p^{2\Z}$ or $f(\beta^{-1})p^{2\Z} = f(\alpha)p^{2\Z}$, hence $\beta = \alpha$ or $\beta = \alpha^{-1}$ by~\eqref{eq:fxy}. Therefore $x\in V_{\alpha}(C,\cO_C)\cup V_{\alpha^{-1}}(C,\cO_C)$ if and only if $N(x)$ is of the form (\ref{c}) for some $c>0$.

Let $(F(t_0),F(t_1))\in {\mathbf B}^+(C)^{\varphi_q=p}\times {\mathbf B}^+(C)^{\varphi_q=p}$ (where $F(t_i):=\sum_{n\in \Z}[t_i^{p^{-2n}}]p^{n}$ with $t_i\in \cO_C$ and $v(t_i)>0$) and recall $(F(t_0),F(t_1))\in m^{-1}(V_\alpha \cup V_{\alpha^{-1}})(C,\cO_C)$ if and only if $F(t_0)F(t_1)\in {\mathbf B}^+(C)^{\varphi_q=p^2}$ lies in $V_\alpha(C,\cO_C)\cup V_{\alpha^{-1}}(C,\cO_C)$. By the previous discussion $(F(t_0),F(t_1))\in m^{-1}(V_\alpha \cup V_{\alpha^{-1}})(C,\cO_C)$ if and only if the Newton polygon of $F(t_0)F(t_1)$ has slopes (\ref{c}) for some $c>0$. Since this Newton polygon has slopes
\begin{equation}\label{newton2}
(p^2-1)v(t_0)p^{2\Z} \cup (p^2-1)v(t_1)p^{2\Z}
\end{equation}
we obtain that $(F(t_0),F(t_1))\in m^{-1}(V_\alpha\cup V_{\alpha^{-1}})(C,\cO_C)$ if and only if there is $c>0$, $\sigma\in S_2$ and $m_0,m_1\in \Z$ such that 
\[(p^2-1)v(t_{\sigma(0)})=(p\alpha^{-1}-1)cp^{2m_0},\ (p^2-1)v(t_{\sigma(1)})=(1-(p\alpha)^{-1})cp^{2m_1},\]
if and only if there is $\sigma\in S_2$ and $m\in \Z$ such that
\[v(t_{\sigma(1)})=\frac{1-(p\alpha)^{-1}}{p\alpha^{-1}-1}p^{2m}v(t_{\sigma(0)})\]
(and recall $\frac{1-(p\alpha)^{-1}}{p\alpha^{-1}-1} = f(\alpha)\notin p^{2\Z}$ for $\alpha\in (\frac{2p}{p^2+1},1]$). 

For $\sigma\in S_2$ and $n\in \Z$, let $V_{\alpha,\sigma,n}$ be the open affinoid perfectoid subspace of $\zlt$ defined by $\vert T_{\sigma(1)}\vert = \vert T_{\sigma(0)}\vert^{\frac{1-(p\alpha)^{-1}}{p\alpha^{-1}-1}p^{n}}$. We claim that
\begin{equation}\label{alpha}
m^{-1}(V_\alpha\cup V_{\alpha^{-1}})=\bigcup_{\sigma\in S_2}\coprod_{m \in \Z}V_{\alpha,\sigma,2m}.
\end{equation}
Indeed, it follows from the previous discussion that the two perfectoid open spaces of $\zlt$ in (\ref{alpha}) have the same rank $1$ points. Then the equality (\ref{alpha}) follows by exactly the same arguments as in the proof of \cite[Prop.~2.4.4]{BHHMS3} (which treats the case $\alpha=1$ for $f\geq 1$), replacing the closed $V_{\sigma(\underline n_0)+f\underline m}$ there by $\kappa^{-1}((\frac{1-(p\alpha)^{-1}}{p\alpha^{-1}-1}p^{2m})^{-\text{sgn}(\sigma)})$, where $\kappa=\kappa_{0,1}$ is defined as in \emph{loc.~cit.}~and $\text{sgn}(\sigma) \in \{\pm 1\}$ is the signature of $\sigma\in S_2$. 

Let us prove that $m^{-1}(V_\alpha\cup V_{\alpha^{-1}})\rightarrow V_\alpha\cup V_{\alpha^{-1}}$ is a pro-\'etale $G$-torsor. If $\alpha=1$ (in which case $V_\alpha = V_{\alpha^{-1}}=V_1$), this is \cite[Cor.~2.4.5]{BHHMS3}, hence we can assume $\alpha\in (\frac{2p}{p^2+1},1)$. Then $V_\alpha$ and $V_{\alpha^{-1}}$ are two disjoint affinoid open subspaces of $\zok$, and since $\frac{1-(p\alpha)^{-1}}{p\alpha^{-1}-1}\in (p^{-2},p^{-1})$, one checks (since $f(\alpha) \notin p^{\Z}$ for $\alpha \ne 1$) that the decomposition (\ref{alpha}) is also disjoint, i.e.,~we have
\[m^{-1}(V_\alpha\sqcup V_{\alpha^{-1}})=\coprod_{\sigma\in S_2}\coprod_{m \in \Z}V_{\alpha,\sigma,2m}.\]
Moreover $S_2$ preserves each $\coprod_{\sigma\in S_2}V_{\alpha,\sigma,2m}$ (sending each open to the other one), $\mathcal{O}_K^{\times, \mathrm{a}}$ preserves each $V_{\alpha,\sigma,2m}$ (looking at rank $1$ points as in the proof of \cite[Prop.~2.4.4]{BHHMS3}) and $(p,p^{-1})$ sends $V_{\alpha,\sigma,2m}$ to $V_{\alpha,\sigma,2m-4\text{sgn}(\sigma)}$. Then one can again apply the same argument as in the last two paragraphs of the proof of \emph{loc.~cit.}\ to deduce that $m^{-1}(V_\alpha\sqcup V_{\alpha^{-1}})\rightarrow V_\alpha\sqcup V_{\alpha^{-1}}$ is a pro-\'etale $G$-torsor. Finally, since $V_\alpha$ (resp.~$m^{-1}(V_\alpha)$) is open in $V_\alpha\sqcup V_{\alpha^{-1}}$ (resp.~$m^{-1}(V_\alpha\sqcup V_{\alpha^{-1}})$), it follows that $m^{-1}(V_\alpha)\rightarrow V_\alpha$ is a pro-\'etale $G$-torsor between perfectoid spaces as well. 

We can determine $m^{-1}(V_\alpha)$, at least under slightly stronger conditions on $\alpha$. Let $(t_0,t_1)\in V_{\alpha,\id,0}(C,\cO_C)$, i.e.,~$(t_0,t_1)\in \cO_C\times \cO_C$ with $v(t_i)>0$ and $v(t_1)=\frac{1-(p\alpha)^{-1}}{p\alpha^{-1}-1}v(t_0)$, and let $x:=F(t_0)F(t_1)\in V_\alpha(C,\cO_C)\sqcup V_{\alpha^{-1}}(C,\cO_C)$ which we write as $x=\sum_{n\in \Z}[x_0^{p^{-2n}}]p^{2n}+\sum_{n\in \Z}[x_1^{p^{-1-2n}}]p^{1+2n}$. Assume that $\alpha$ is such that $\frac{1-(p\alpha)^{-1}}{p\alpha^{-1}-1}>\frac{1}{p^2-1}$ (instead of $\frac{1}{p^2}$ previously) or equivalently $\alpha\in (\frac{2}{p}-\frac{1}{p^3},1)$. Arguing as in the proof of \cite[Lemma 2.9.2]{BHHMS3} with $(C,\cO_C)$ instead of $(A'_\infty,(A'_\infty)^0)$ and with $\vert t_0 t_1\vert = \vert t_0\vert^{1+\frac{1-(p\alpha)^{-1}}{p\alpha^{-1}-1}}$ instead of $p^{-c}$ (where $\lvert \cdot \rvert$ is defined by $v(\cdot)=-\log(\lvert \cdot \rvert)$, note that $\vert t_i^{q^n}\vert<\vert t_0 t_1\vert$ for $i=0,1$ and $n\geq 1$ by our assumption on $\alpha$), the same argument gives in $\cO_C$:
\[\vert x_0-t_0t_1\vert < \vert t_0 t_1\vert\text{ \ and \ } \vert x_1^{p^{-1}}-(t_0^{q^{-1}}t_1+t_0t_1^{q^{-1}})\vert < \vert t_0 t_1\vert.\]
Since $\vert t_0t_1\vert < \vert t_0t_1^{q^{-1}}\vert < \vert t_0^{q^{-1}}t_1\vert$, we deduce $v(x_0)=v(t_0t_1)$ and $v(x_1)=v(t_0^{p^{-1}}t_1^p)$, equivalently $v(x_0)=\frac{p-p^{-1}}{p\alpha^{-1}-1}\alpha^{-1}v(t_0)$ and $v(x_1)=\frac{p-p^{-1}}{p\alpha^{-1}-1}v(t_0)$, hence $v(x_1)=\alpha v(x_0)$ i.e.,~$x\in V_{\alpha^{-1}}(C,\cO_C)$. We deduce that $m(V_{\alpha,\id,0}) \subset V_{\alpha^{-1}}$, and hence by using the $G$-action that $m(\coprod_{\sigma\in S_2}\coprod_{m \in 2\Z}V_{\alpha,\sigma,2m}) \subset V_{\alpha^{-1}}$.
By applying the action of $(p,1)$ we then deduce using \cite[(38)]{BHHMS3} that $m(\coprod_{\sigma\in S_2}\coprod_{m \in 1+2\Z}V_{\alpha,\sigma,2m}) \subset V_{\alpha}$. In particular, for $\alpha\in (\frac{2}{p}-\frac{1}{p^3},1]$ it follows that
\[m^{-1}(V_\alpha)=\coprod_{\sigma\in S_2}\coprod_{m \in 1+2\Z}V_{\alpha,\sigma,2m}\text{ \ and \ }m^{-1}(V_{\alpha^{-1}})=\coprod_{\sigma\in S_2}\coprod_{m \in 2\Z}V_{\alpha,\sigma,2m}\]
(note that when $\alpha=1$ these spaces are equal).

\textbf{Step 2.} We prove that for $\alpha\in (\frac{2}{p}-\frac{1}{p^3},1]\cap \Q$ the map $m^{-1}(W_\alpha)\rightarrow W_\alpha$ is a pro-\'etale $G$-torsor between perfectoid spaces.

When $\alpha=1$ this is \cite[Cor.~2.4.5]{BHHMS3}, hence we can assume $\alpha<1$. For $\sigma\in S_2$ and $n\in \Z$, let $W_{\alpha,\sigma,n}$ be the open affinoid perfectoid subspace of $\zlt$ defined by $\vert T_{\sigma(0)}\vert^{p^{n-1}} \leq \vert T_{\sigma(1)}\vert \leq  \vert T_{\sigma(0)}\vert^{\frac{1-(p\alpha)^{-1}}{p\alpha^{-1}-1}p^{n}}$. Using $\frac{1-(p\alpha)^{-1}}{p\alpha^{-1}-1}\in (p^{-2},p^{-1})$, one easily checks that the intervals
\[\left\{\Big[\frac{1-(p\alpha)^{-1}}{p\alpha^{-1}-1}p^{2m}, p^{2m-1}\Big], \Big[p^{1-2m},\frac{p\alpha^{-1}-1}{1-(p\alpha)^{-1}}p^{-2m}\Big], m\in 1+2\Z\right\}\]
are pairwise disjoint in $\R_{>0}$. Hence for $\sigma\in S_2$ and $m\in 1+2\Z$ the $W_{\alpha,\sigma,2m}$ are pairwise disjoint in $\zlt$ (the assumption $m\in 1+2\Z$ is important as the $W_{\alpha,\sigma,2m}$ for $m\in \Z$ are not pairwise disjoint, as then each $p^{2m-1}$ is contained in two of the intervals above). We claim that
\begin{equation}\label{alphaw}
m^{-1}(W_\alpha)=\coprod_{\sigma\in S_2}\coprod_{m \in 1+2\Z}W_{\alpha,\sigma,2m}.
\end{equation}
Indeed, let $C$ be a perfectoid field containing $\F$ and $v=-\log\lvert \cdot \rvert$ a continuous rank $1$ additive valuation on $C$ (valued in $\R \cup \{+\infty\}$). Let $x=F(x_0,x_1)\in {\mathbf B}^+(C)^{\varphi_q=p^2}$ with $x_i\in \cO_C$ and $v(x_i)>0$, and let $(t_0,t_1)\in \cO_C\times \cO_C$ with $v(t_i)>0$ such that $x=F(t_0)F(t_1)$. Replacing $\alpha$ by $\beta:=v(x_0)/v(x_1)\in \R^+\setminus\{0\}$, exactly the same arguments as in Step 1 (including its last paragraph) show that $\beta \in [\alpha,1]$ if and only if $f(\alpha) p^{2m} \le v(t_{\sigma(1)})/v(t_{\sigma(0)}) \le p^{2m-1}$ for some $\sigma\in S_2$ and $m\in 1+2\Z$  (recall that $f(x)$ is strictly increasing for $x < p$ and that $f(1) = 1/p$).
We deduce that the two perfectoid open subspaces of $\zlt$ in (\ref{alphaw}) have the same rank $1$ points. Then (\ref{alphaw}) follows by the same argument as in the proof of \cite[Prop.~2.4.4]{BHHMS3} (replacing the $V_{\sigma(\underline n_0)+f\underline m}$ there by the $\kappa^{-1}([f(\alpha)p^{2m}, p^{2m-1}])$ and $\kappa^{-1}([p^{1-2m},f(\alpha)^{-1}p^{-2m}])$ for $m\in 1+2\Z$ with $\kappa=\kappa_{0,1}$ as in \emph{loc.~cit.}). Finally, $G$ acts on the $W_{\alpha,\sigma,2m}$ similarly as on the $V_{\alpha,\sigma,2m}$ in Step 1, and by the same argument we obtain that $m^{-1}(W_\alpha)\rightarrow W_\alpha$ is a pro-\'etale $G$-torsor between perfectoid spaces.

\textbf{Step 3.} We assume that the map in the statement is an isomorphism and prove that, for $\alpha \in (\frac{2}{p}-\frac{1}{p^3},1]\cap \Q$ and moreover $\alpha^{-1} \in \Z[1/p]$, the restriction map $H^0(U_0,(m_* \cF)^G)\rightarrow H^0(V_\alpha,(m_* \cF)^G)$ is injective.

Assume $\alpha\in (\frac{2}{p}-\frac{1}{p^3},1]\cap \Q$. Then by Step $2$ and descent of vector bundles (\cite[Lemma 17.1.8]{SWBerkeley}) the $\cO_{W_\alpha}$-module $(m_* \cF)^G\vert_{W_\alpha}$ is locally free of finite rank. Since $W_\alpha$ is affinoid, we deduce that $((m_* \cF)^G)(W_\alpha)$ is a locally free $\cO_{\zok}(W_\alpha)$-module of finite rank. Moreover, \ since \ $V_\alpha\subset W_\alpha$ \ is \ also \ affinoid, \ we \ have \ $((m_* \cF)^G)(V_\alpha)=((m_* \cF)^G)(W_\alpha)\otimes_{\cO_{\zok}(W_\alpha)}\cO_{\zok}(V_\alpha)$.

{As $\alpha^{-1} \in \Z[1/p]$, by Lemma~\ref{lem:res-inj} below the restriction map} $\cO_{\zok}(W_\alpha)\to \cO_{\zok}(V_\alpha)$ {is injective}. Let $M$ be a free $\cO_{\zok}(W_\alpha)$-module of finite rank such that the $\cO_{\zok}(W_\alpha)$-module $((m_* \cF)^G)(W_\alpha)$ is a direct summand of $M$. Since the map $M\rightarrow M\otimes_{\cO_{\zok}(W_\alpha)}\cO_{\zok}(V_\alpha)$ is injective and since the (injective) composition
\[((m_* \cF)^G)(W_\alpha)\hookrightarrow M\hookrightarrow M\otimes_{\cO_{\zok}(W_\alpha)}\cO_{\zok}(V_\alpha)\]
factors through the map $((m_* \cF)^G)(W_\alpha)\rightarrow ((m_* \cF)^G)(W_\alpha)\otimes_{\cO_{\zok}(W_\alpha)}\cO_{\zok}(V_\alpha)$, we deduce that the latter is also injective, i.e.,~$((m_* \cF)^G)(W_\alpha)\hookrightarrow ((m_* \cF)^G)(V_\alpha)$. 

Now, the restriction map $((m_* \cF)^G)(U_0)\rightarrow ((m_* \cF)^G)(U_0\cap U_1)$ is injective as follows from the assumption in Step 3. Since $U_0\cap U_1\subset W_\alpha\subset U_0$, it factors through the restriction map $((m_* \cF)^G)(U_0)\rightarrow ((m_* \cF)^G)(W_\alpha)$ which is therefore also injective. Using the previous paragraph, we finally deduce that the composition
\[((m_* \cF)^G)(U_0)\hookrightarrow ((m_* \cF)^G)(W_\alpha)\hookrightarrow ((m_* \cF)^G)(V_\alpha)\]
is injective.

\textbf{Step 4.} We assume that the map in the statement is an isomorphism and prove that $H^0(U_j,(m_*\cF)^G)$ for $j=0,1$ is a {topological} $A_\infty$-module compatibly with its $H^0(U_j,\cO_{U_j})$-module structure {such that the restriction map $H^0(U_j,(m_*\cF)^G) \to H^0(U_0 \cap U_1,(m_*\cF)^G)$ is an $A_\infty$-linear closed embedding}.

By assumption the restriction maps induce an isomorphism
\begin{equation}\label{R}
R:H^0(U_0,(m_*\cF)^G)\oplus H^0(U_1,(m_*\cF)^G) \congto H^0(U_0\cap U_1, (m_* \cF)^G)
\end{equation}
of $\F\bbra{X_0^{p^{-\infty}},X_1^{p^{-\infty}}}$-modules.
As $X_0$, $X_1$ act invertibly on the right side, the same is true for each direct summand on the left side (cf.\ the proof of Corollary \ref{non0bis}), which means in particular that each $H^0(U_i,(m_*\cF)^G)$ is canonically a $\F\bbra{X_0^{p^{-\infty}},X_1^{p^{-\infty}}}[\frac1{X_0X_1}]$-module and that $R$ is $\F\bbra{X_0^{p^{-\infty}},X_1^{p^{-\infty}}}[\frac1{X_0X_1}]$-linear, {and moreover the restriction of $R$ to the first direct summand in~\eqref{R} is $\F\ppar{X_0^{p^{-\infty}}}\langle \big(\frac{X_1}{X_0}\big)^{p^{-\infty}}\rangle[\frac{1}{X_1}]$-linear}. In particular $H^0(U_0,(m_*\cF)^G)$ is an $\F\ppar{X_1^{p^{-\infty}}}$-vector space.

Now \ we \ prove \ that \ $R$ \ is \ a \ topological \ isomorphism. \ Recall \ first \ that $H^0(m^{-1}(U_0),\cF)\cong H^0(m^{-1}(U_0),\cO_{\zlt})^{\oplus 2}$ (forgetting the group actions) is a Fr\'echet space (over $\F\ppar{X_0^{p^{-\infty}}}$) by the discussion in \S~\ref{geometricinter}. As $H^0(U_0,(m_*\cF)^G)\cong H^0(m^{-1}(U_0),\cF)^G$ is a closed subspace of $H^0(m^{-1}(U_0),\cF)$, it is also a Fr\'echet space. In particular it has a countable fundamental system of neighborhoods of 0. Moreover $H^0(U_0,(m_*\cF)^G)$ is a topological $\F\ppar{X_1^{p^{-\infty}}}$-vector space, i.e.,~the natural map (from the previous paragraph)
\[\F\ppar{X_1^{p^{-\infty}}} \times H^0(U_0,(m_*\cF)^G)\rightarrow H^0(U_0,(m_*\cF)^G)\]
is continuous. {(It suffices to show that each element of $\F\ppar{X_1^{p^{-\infty}}}$ acts continuously on $H^0(U_0,(m_*\cF)^G)$. This is clear, as each $X_1^r$, $r \in \Z[1/p]_{\ge 0}$ induces a continuous bijection, hence homeomorphism, between Fr\'echet spaces.)}
Hence $H^0(U_0,(m_*\cF)^G)$ is a complete topological $\F\ppar{X_1^{p^{-\infty}}}$-vector space having a countable fundamental system of neighborhoods of 0. Likewise $H^0(U_1,(m_*\cF)^G)$ and $H^0(U_0\cap U_1, (m_* \cF)^G)$ are complete topological $\F\ppar{X_1^{p^{-\infty}}}$-vector spaces, each having a countable fundamental system of neighborhoods of 0 (though here we directly have the $\F\ppar{X_1^{p^{-\infty}}}$-action). As the map $R$ is moreover $\F\ppar{X_1^{p^{-\infty}}}$-linear we deduce by \cite[Thm.\ II.4.1.1]{Morel} that $R$ is a topological isomorphism. By the first paragraph of the proof, $\F\bbra{X_0^{p^{-\infty}},X_1^{p^{-\infty}}}[\frac1{X_0X_1}]$ stabilizes the direct sum decomposition of $H^0(U_0\cap U_1, (m_* \cF)^G)$ given by (the image of) $R$. Note that $\F\bbra{X_0^{p^{-\infty}},X_1^{p^{-\infty}}}[\frac1{X_0X_1}]$ \ is \ dense \ in \ $A_\infty$, \ which \ acts \ continuously \ on $H^0(U_0\cap U_1, (m_* \cF)^G)$.
As the direct summands are closed by the previous paragraph, we deduce that $A_\infty$ stabilizes the direct sum decomposition of $H^0(U_0\cap U_1, (m_* \cF)^G)$.
Via restriction, we therefore obtain a topological action of $A_\infty$ on $H^0(U_i, (m_* \cF)^G)$, which extends the given action of $H^0(U_j,\cO_{U_j})$.

\textbf{Step 5.} We assume that the map in the statement is an isomorphism and prove that, for $\alpha\in (\frac{2p}{p^2+1},1)\cap \Q$, the restriction map $H^0(U_0,(m_* \cF)^G)\rightarrow H^0(V_\alpha,(m_* \cF)^G)$ {is zero}, yielding a contradiction with Step 3 {when $\alpha^{-1} \in \Z[1/p]$}.

Write res for the restriction map $\res:H^0(U_0,(m_* \cF)^G)\rightarrow H^0(V_\alpha,(m_* \cF)^G)$. Note that $X_1$ acts invertibly on both sides, using Step 4 for the left side, so $\res$ is $H^0(U_0,\cO_{U_0})[\frac 1{X_1}]$-linear.
Let $v$ be a nonzero element of $H^0(U_0,(m_* \cF)^G)$ and assume that its image $\res(v)$ is nonzero. Let $N,{n}$ be positive integers. Since
\[X_1^n\Big(\frac{X_0}{X_1}\Big)^{Nn}=X_0^n\Big(\frac{X_1}{X_0}\Big)^{n}\Big(\frac{X_0}{X_1}\Big)^{Nn}\rightarrow 0\text{ in }A_\infty\text{ as }n\rightarrow +\infty,\]
by Step $4$ we deduce that $X_1^n(\frac{X_0}{X_1})^{Nn}v\rightarrow 0$ in $H^0(U_0,(m_*\cF)^G)$ as $n\rightarrow +\infty$. By continuity of the restriction map res, we deduce for any positive integer $N$ that
\begin{equation}\label{towards0}
\!\!\res\Big(X_1^n\Big(\frac{X_0}{X_1}\Big)^{Nn}v\Big)=X_1^n\Big(\frac{X_0}{X_1}\Big)^{Nn}\!\!\res(v)\rightarrow 0\text{ in }H^0(V_\alpha,(m_* \cF)^G)\text{ as }n\rightarrow +\infty.
\end{equation}
{Let $A_{\alpha,\infty}:=\Gamma(V_\alpha,\cO_{\zok})$.}
As in Step $3$, $H^0(V_\alpha,(m_* \cF)^G)$ is a locally free $A_{\alpha,\infty}$-module of finite rank (using Step $1$). We fix a free $A_{\alpha,\infty}$-module $M$ of finite rank containing $H^0(V_\alpha,(m_* \cF)^G)$ as a direct summand. Replacing $\res(v)$ by one of its nonzero coordinates in a basis of $M$ over $A_{\alpha,\infty}$, (\ref{towards0}) implies there is a nonzero element $a\in A_{\alpha,\infty}$ such that $X_1^n\big(\frac{X_0}{X_1}\big)^{Nn}a\rightarrow 0$ in $A_{\alpha,\infty}$ as $n\rightarrow +\infty$. 

Recall a point $x$ of $V_\alpha$ corresponds to an equivalence class of continuous multiplicative seminorm $\lvert \cdot \rvert_x:A_{\alpha,\infty}\rightarrow \Gamma \cup \{0\}$ where $\Gamma$ is a totally ordered abelian group (written multiplicatively).
{Also recall that $V_\alpha$ is perfectoid, hence uniform and analytic.}
Since $a\ne 0$ there exists $x \in V_\alpha$ such that $|a|_x \ne 0$ by \cite[Thm.\ 5.2.1]{SWBerkeley}. Replacing $x$ by its maximal generization $\wt x$ as in \cite[Prop.~4.2.5]{SWBerkeley} we may assume that $x$ has rank 1 and that $\lvert \cdot\rvert_x$ is valued in $\R_{\ge 0}$. Since $\lvert \cdot \rvert_x$ is continuous, from the previous paragraph we have $\vert X_1^n\big(\frac{X_0}{X_1}\big)^{Nn}\vert_x\vert a\vert_x\rightarrow 0$ in $\R$ as $n\rightarrow +\infty$ for any positive integer $N$. Moreover, by definition of $V_\alpha$ and using $\alpha<1$, we have $\vert X_1 \vert_x=\vert X_0\vert_x^{\alpha^{-1}}<\vert X_0\vert_x\ (<1)$. Taking $N \gg 0$ such that $(\vert \frac{X_0}{X_1} \vert_x)^N>\vert X_1\vert_x^{-1}$, which is possible as $\vert \frac{X_0}{X_1} \vert_x>1$, since $\vert a\vert_x\ne 0$ we cannot have $\vert X_1^n\big(\frac{X_0}{X_1}\big)^{Nn}\vert_x\vert a\vert_x\rightarrow 0$ in $\R$. It follows that we must have $\res(v)=0$, contradiction. Hence $\res = 0$.

Suppose now moreover that $\alpha^{-1} \in \Z[1/p]$.
  Then res is injective by Step 3, so $H^0(U_0,(m_* \cF)^G) = 0$ and by symmetry $H^0(U_1,(m_* \cF)^G) = 0$.
  By our assumption we deduce that $H^0(U_0 \cap U_1,(m_* \cF)^G) = 0$, but this is a free $A_\infty$-module of the same (positive) rank as $\cF$.
  This contradiction completes the proof of Theorem \ref{AIM23}.
\end{proof}

\begin{lem}\label{lem:res-inj}
  In Step 3, if $\alpha^{-1} \in \Z[1/p] \cap \R_{\ge 1}$, the restriction map $H^0(W_\alpha) \to H^0(V_\alpha)$ is injective.
\end{lem}
\begin{proof}
  First note: if $(A,A^+)$ is any Huber pair of characteristic $p$, then the perfection $(A\perf,A^{+\mathrm{perf}})$ is a Huber pair.
  (If $A_0$ with $I$-adic topology is a ring of definition of $A$, then $A_0\perf$ with $IA_0\perf$-adic topology is a ring of definition of $A\perf$.)
  Then $\Spa(A\perf,A^{+\mathrm{perf}}) \to \Spa(A,A^+)$ is a homeomorphism that identifies rational open subsets (cf.\ the proof of \cite[Prop.\ 6.11(i)]{Scholzeperfectoid}) and likewise for completions \cite[Prop.\ 2.3.9]{SWBerkeley}.
  Moreover a rational open subset $\Spa(B,B^+)$ in $\Spa(A,A^+)$ has preimage $\Spa(\wh {B\perf},\wh {B^{+\mathrm{perf}}})$ in $\Spa(\wh {A\perf},\wh {A^{+\mathrm{perf}}})$ (cf.\ the proof of \cite[Prop.\ 6.11(ii)]{Scholzeperfectoid}).

  Write $\alpha = m/d$ for some positive integers $m \le d$.
  If $W_\alpha^0$ (resp.\ $V_\alpha^0$) denotes the rational subset $\{|X_1| \leq |X_0| \leq |X_1|^\alpha \ne 0 \}$ (resp.\ $\{|X_0| = |X_1|^\alpha \ne 0 \}$) in $\Spa \F\bbra{X_0,X_1}$, then its preimage in $\Spa \F\bbra{X_0^{p^{-\infty}},X_1^{p^{-\infty}}}$ is given by $W_\alpha$ (resp.\ $V_\alpha$).
  Hence $H^0(W_\alpha)$ is given by the completed perfection of
  \begin{equation}
    H^0(W_\alpha^0) \cong \F\ppar{X_0}\left\langle \frac {X_1}{X_0},\frac {X_0^d}{X_1^{m}} \right\rangle \cong \F\ppar{X_0}\langle Z,W \rangle/(Z^{m}W-X_0^{d-m})
  \end{equation}
  and $H^0(V_\alpha)$ is the completed perfection of
  \begin{equation}
    H^0(V_\alpha^0) \cong \F\ppar{X_0}\left\langle \frac {X_1}{X_0},\Big(\frac {X_0^d}{X_1^{m}}\Big)^{\pm 1} \right\rangle \cong \F\ppar{X_0}\langle Z,W^{\pm 1} \rangle/(Z^{m}W-X_0^{d-m}).
  \end{equation}

  Using crucially that $m = p^r$ one verifies that after perfection we have a sequence of topological abelian groups
  \begin{equation*}
    0 \to \F\ppar{X_0}\perf \to \{\F\ppar{X_0}\langle Z\rangle\}\perf \oplus \{\F\ppar{X_0}\langle W\rangle\}\perf \to H^0(W_\alpha^0)\perf \to 0
  \end{equation*}
  where the first map is antidiagonal and the second is addition, and that this sequence is in fact \emph{strict}.
  Hence after completion we obtain
  \begin{equation*}
    0 \to \F\ppar{X_0^{p^{-\infty}}} \to \F\ppar{X_0^{p^{-\infty}}}\langle Z^{p^{-\infty}}\rangle \oplus \F\ppar{X_0^{p^{-\infty}}}\langle W^{p^{-\infty}}\rangle \to H^0(W_\alpha) \to 0.
  \end{equation*}
  Likewise, but more easily, we obtain that the natural map
  \begin{equation*}
    \F\ppar{X_0^{p^{-\infty}}}\langle W^{\pm p^{-\infty}}\rangle \to H^0(V_\alpha)
  \end{equation*}
  is a homeomorphism.
  The restriction map $H^0(W_\alpha) \to H^0(V_\alpha)$ is then injective because the $\F\ppar{X_0^{p^{-\infty}}}$-linear continuous map
  \begin{equation*}
    \F\ppar{X_0^{p^{-\infty}}}\langle Z^{p^{-\infty}}\rangle \oplus \F\ppar{X_0^{p^{-\infty}}}\langle W^{p^{-\infty}}\rangle \to \F\ppar{X_0^{p^{-\infty}}}\langle W^{\pm p^{-\infty}}\rangle,
  \end{equation*}
  sending $W^{1/p^n}$ to $W^{1/p^n}$ and $Z^{1/p^n}$ to $(X_0^{d-p^r}W^{-1})^{1/p^{r+n}}$, has kernel $\F\ppar{X_0^{p^{-\infty}}}$.
\end{proof}

We finally provide a slight strengthening to Theorem~\ref{AIM23}.

\begin{cor}\label{cor:AIM23}
  Suppose $f=2$ and $\cF$ any $G$-equivariant locally free sheaf of finite rank on $\zlt$.
  If $\cF \ne 0$, the map $$H^0(U_0,(m_*\cF)^G)\oplus H^0(U_1,(m_*\cF)^G) \rightarrow H^0(U_0\cap U_1, (m_* \cF)^G)$$ in (\ref{connecting}) cannot be injective with cokernel that is smooth as $H^0(\zok)$-module.
\end{cor}
\begin{proof}
  The argument of the proof of Corollary~\ref{non0bis} applies (replacing the last term in~\eqref{MVbis} by the cokernel) to show that the cokernel is zero,
  hence the result follows from Theorem~\ref{AIM23}.
\end{proof}

\subsection{The principal series case II}\label{sec:princ-seri-case2}

We briefly discuss the analog of Questions~\ref{ques} in the much simpler principal series case.
Namely, we assume {$\o r_v(1)\cong \smatr{\chi_1}{*}0{\chi_2}$, $\pi = \Ind_B^G(\chi_2^{-1}\omega \otimes \chi_1^{-1})$ and consider the rank 1 subsheaf $\cF_{\chi_1}$} of $\cF_{\o r_v(1)}$. We keep the same notation as in \S~\ref{sec:princ-seri-case} {and write $J:=\pi$}.

We let $\sigma := \ker(J \onto \F(\chi_2^{-1}\omega \otimes \chi_1^{-1}))$ denote its unique irreducible $P$-subrepresen\-tation, cf.\ \eqref{eq:5} (uniqueness follows from the assumption that $\o r_v$ is sufficiently generic, so that $\chi_2\chi_1^{-1}\neq \omega$). Recall that we have an isomorphism of $(\vp,\oks)$-modules $D_A(J) \cong D_A^\otimes(\chi_1)$ over $A$, which is unique up to $\F\s$ by \cite[Cor.\ 3.9]{BHHMS2}.
We claim that the image of the $P$-equivariant map
\begin{equation}\label{eq:first-map}
  \epsilon : D_{A_\infty}^\otimes(\chi_1^{-1}) \cong \Hom_{A_\infty}(D_{A_\infty}(J),A_\infty) \to J \otimes \chi
\end{equation}
constructed in \S~\ref{sec:surj-Ainfty-pi} is isomorphic to the image of the $P$-equivariant map
\begin{equation}\label{eq:second-map}
  \epsilon' : D_{A_\infty}^\otimes(\chi_1^{-1}) \to H^1(p^{\Z,\mathrm{a}},H^0(\cF_{\chi_1^{-1}}))^G
\end{equation}
constructed in \S~\ref{sec:another-b-morphism}, and that both are isomorphic to $\sigma \otimes \chi$.

We first justify this claim assuming the following two assertions:
\begin{enumerate}
\item the image of $\epsilon$ equals $\sigma \otimes \chi$;
\item the kernel of $\epsilon$ equals $A_\infty^+ e$, where $e$ is the basis of $D_{A_\infty}^\otimes(\chi_1^{-1})$ dual to the basis $\o\kappa$ of $D_{A_\infty}(J)$ and where
\begin{equation*}
  A_\infty^+ := \Big\{ \sum_{n=0}^\infty \lambda_n Y_0^{a_n} Y_1^{b_n} \in A_\infty : \text{$a_n \ge 0$ or $b_n \ge 0\ \forall n$} \Big\}.
\end{equation*}  
\end{enumerate}

For short, let $\cG_{\chi_1^{-1}} := (m_* \cF_{\chi_1^{-1}})^G$.
Recall that $\cF_{\chi_1^{-1}} \cong \cO_{\zlt}$ if we only remember the $G$-action and hence $\cG_{\chi_1^{-1}} \cong \cO_{\zok}$ as sheaves (without $K\s$-actions).

By combining Corollary~\ref{irrtriv} and Proposition~\ref{nonzero} {(and Proposition \ref{avoir})} we get an exact sequence
\begin{equation}\label{eq:2}
  \begin{split}
    H^0(U_0,\cO_{U_0})\oplus H^0(U_1,\cO_{U_1}) \to H^0&(U_0 \cap U_1,\cO_{U_0 \cap U_1}) \to H^1(\zok, \cO_{\zok}) \\ & \cong H^1_{\mathrm{cts}}(G,H^0(\zlt)) \into H^1(p^{\Z,\mathrm{a}},H^0(\zlt))^G.
  \end{split}
\end{equation}
and hence, taking into account residual $K\s$-actions, an exact sequence of $P$-represen\-tations
\begin{multline*}
H^0(U_0,\cG_{\chi_1^{-1}})\oplus H^0(U_1,\cG_{\chi_1^{-1}}) \to H^0(U_0 \cap U_1,\cG_{\chi_1^{-1}}) \to H^1(\zok, \cG_{\chi_1^{-1}}) \\ \cong H^1_{\mathrm{cts}}(G,H^0(\zlt,{\cF}_{\chi_1^{-1}})) \into H^1(p^{\Z,\mathrm{a}},H^0(\zlt,{\cF}_{\chi_1^{-1}}))^G.
\end{multline*}
Thus the kernel of $\epsilon'$ equals the image of $H^0(U_0,\cG_{\chi_1^{-1}})\oplus H^0(U_1,\cG_{\chi_1^{-1}})\to H^0(U_0 \cap U_1,\cG_{\chi_1^{-1}}) = D_{A_\infty}(\chi_1^{-1})$.
By the explicit description of the first map in~\eqref{eq:2} given in~\eqref{eq:cokerH^f-1}, \eqref{eq:H^f-1-zok} (and twisting by $\chi_1^{-1}$), the kernel of $\epsilon'$ is given by $A_\infty^+ e$, i.e.,\ it is equal to the kernel of $\epsilon$, which completes the proof of the claim.

It remains to prove assertions (i) and (ii).
For (i), by construction the map $\epsilon$ is obtained from
\begin{equation}\label{eq:seq1}
  \Hom_A(D_A(J),A) \to \Hom_\F^{\cont}(D_A(J),\F) \otimes \chi \to \Hom_\F^{\cont}(J^\vee,\F) \otimes \chi = J \otimes \chi
\end{equation}
by passing to $\ilim_\vp$ and then completing, where $\vp$ acts as precomposition by $\psi$ on the middle term and as $\smatr p001$ on the right.
Passing to $\ilim_\vp$ gives
\begin{multline}\label{eq:seq2}
  A_{(\infty)} \otimes \Hom_A(D_A(J),A) \to \Hom_\F^{\cont}(\plim_\psi D_A(J),\F) \otimes \chi \\
  \to \Hom_\F^{\cont}(J^\vee,\F) \otimes \chi = J \otimes \chi,
\end{multline}
where the second map is obtained by dualizing the map $\Theta$ of Lemma~\ref{kernel}, so that the image of~\eqref{eq:seq2} is contained in the unique subrepresentation of $J \otimes \chi$ of codimension 1, namely $\sigma \otimes \chi$.
By continuity, the image of $\epsilon$ is contained in $\sigma \otimes \chi$ and hence is equal to $\sigma \otimes \chi$, as $\epsilon$ is nonzero by \eqref{eq:lim-eps} and $\sigma$ is irreducible.

We prove assertion (ii).
By (i) it suffices to show that the kernel of $\epsilon$ contains $A_\infty^+ e$, as $A_\infty/A_\infty^+$ is an irreducible $P$-representation by (the proof of) Lemma~\ref{irr} and $\epsilon$ is nonzero.
Factor \eqref{eq:seq1} as
\begin{equation*}
  \Hom_A(D_A(J),A) \to \Hom_\F^{\cont}(D_A(J),\F) \otimes \chi \to \Hom_\F^{\cont}(D_A(J)^\natural,\F) \otimes \chi \to J \otimes \chi.
\end{equation*}
Using the basis $\o\kappa$ of $D_A(J)$ we can identify $D_A(J) \cong A$ and $D_A(J)^\natural \cong \F\bbra{N_0}$, cf.\ \S~\ref{sec:princ-seri-case}.
It thus suffices to ignore the final map and show that the composition
\begin{equation*}
  A = \Hom_A(A,A) \to \Hom_\F^{\cont}(\F\bbra{N_0},\F)
\end{equation*}
kills $A^+ := A \cap A_\infty^+$.
Unwinding definitions, we need to show that $\mu(A^+ \cdot \F\bbra{N_0}) = 0$, or equivalently $\mu(A^+) = 0$, where $\mu : A \to \F$ is the map of \cite[\S~3.3]{BHHMS3}.
Write $\F\bbra{N_0} = \F\bbra{Z_0,Z_1}$ as in \emph{loc.~cit.}, so that $A$ is the completed localization of $\F\bbra{N_0}$ at $Z_0Z_1$.
By \cite[(80)]{BHHMS3} we see that $\mu$ kills $A'^+ := \{\sum_{n=0}^\infty \lambda_n Z_0^{a_n}Z_1^{b_n} \in A : \text{$a_n \ge 0$ or $b_n \ge 0$ for all $n$}\}$ (note that the term $\prod_j (1+T_j)^{-1}$ in \emph{loc.~cit.} is contained in $\F\bbra{N_0}$).
By using that $Z_i \in Y_i + \m_{\F\bbra{N_0}}^2$ it easily follows that $A'^+ = A^+$.

We have finished proving (i) and (ii) and hence the claim that the maps~\eqref{eq:first-map} and \eqref{eq:second-map} have isomorphic images.

\appendix

\section{Digression on continuous group cohomology}\label{sec:digr-cont-coho}

We make a couple of elementary observations about continuous group cohomology, which can be verified from the definitions {and which we use in \S~\ref{perfII}}.

We assume that $G$ is a topological group and consider the category of topological $G$-modules, having as objects topological abelian groups endowed with a continuous action of $G$ and as morphisms continuous group homomorphisms (as in \cite{tate1976}).
For any topological $G$-module $A$ we have
\begin{equation*}
  H^0\cts(G,A) = H^0(G,A), \qquad H^1\cts(G,A) \subset H^1(G,A).
\end{equation*}
If $0 \to A \to B \to C \to 0$ is an exact sequence of topological $G$-modules with the injection $A \to B$ \emph{strict} (i.e.,\ a closed embedding) {it is an easy exercise (left to the reader) to check that} we have a long exact sequence
\begin{multline}\label{eq:H1cts}
  0 \to H^0(G,A) \to H^0(G,B) \to H^0(G,C) \to H^1\cts(G,A) \to H^1\cts(G,B) \\
  \to H^1\cts(G,C).
\end{multline}
If $N$ is a closed subgroup of $G$, then {one can again easily check that} we have an exact inflation-restriction sequence
\begin{equation}
  \label{eq:infl-res}
  0 \to H^1\cts(G/N,A^N) \to H^1\cts(G,A) \to H^1\cts(N,A)^{G/N}.
\end{equation}

\begin{lem1}\label{lem:cts-coho-free-prop}
  Suppose that $\Gamma \cong \Z_p^n$ and that $M$ is a topological $\zp$-module equipped with a continuous action of $\Gamma$.
  If $M$ is separated and complete with respect to a ($\Z$-)linear topology, then $H^i\cts(\Gamma,M)$ is computed by the Koszul complex defined by $\gamma_1-1,\dots,\gamma_n-1$ acting on $M$, where $\gamma_1,\dots,\gamma_n$ denote topological generators of $\Gamma$.
  In particular, $H^i\cts(\Gamma,M) = 0$ if $i > n$ and $H^n\cts(\Gamma,M) \cong M_\Gamma$ ($\Gamma$-coinvariants).
\end{lem1}
\begin{proof}
  We pass to the condensed world \cite{ClausenS} by using the condensation functor $X \mapsto \un{X}$ of \emph{loc.~cit.} Then $\un M$ is a condensed $\un{\Z_p}$-module equipped with a $\un{\Gamma}$-action.
  By our assumption on the topology of $M$, $\un M$ is solid as a condensed abelian group (\cite[Lemma 3.2.1]{BSSW}) and hence solid as $\un{\Z_p}$-module (\cite[Rk.\ 2.14]{RodriguesJacinto2022}). Hence, by \cite[Prop.\ B.3]{bosco} we deduce that the condensed group cohomology $R\un{\Hom}_{\Z[\un\Gamma]}(\un{\Z},\un{M})$ is quasi-isomorphic to the Koszul complex defined by $\gamma_1-1,\dots,\gamma_n-1$ acting on $\un M$. Evaluating at a point $*$ and taking cohomology we obtain the result by \cite[Prop.\ B.2]{bosco}.
\end{proof}

\section{Sheaf cohomology versus continuous group cohomology}\label{sec:appendix}

We put ourselves in the setting of \S\ref{sec:aside-trivial-galois}.
The goal of this appendix is to prove Proposition \ref{avoir} using Proposition \ref{prop:group_cohomology_trivial_bundle} below. As it is going to play a pivotal role, we start by recalling in this context what it means for a locally profinite group to act freely and how we think of the quotient space.

Let $\Perf$ be the category of perfectoid spaces of characteristic
$p$ (we have $\Perf=\Perf_{\Fp}$ with the notation of \S~\ref{geometricinter}). Let $X$ be an object of $\Perf$. Let $H$ be a locally profinite
group acting on $X$. This means that there is a map of perfectoid
spaces $\alpha : \underline{H}\times X\rightarrow X$ which is a group
action. ({Recall that $\un H$ is the sheaf defined by $\un H(W) = \cC^0(|W|,H)$ and} $\underline{H}\times X$ is a perfectoid space as
a consequence of \cite[Ex.~11.12]{ScholzeEtcohDia}). We assume that
$\underline{H}$ acts \emph{freely} on $X$, i.e.,~that for any pair $(k,k^+)$
where $k$ is perfectoid and $k^+$ is an open and bounded valuation
ring of $k$, the action of $H$ on $X(k,k^+)$ has trivial stabilizers.  It
follows from \cite[Prop.~5.3]{ScholzeEtcohDia} that the map of
perfectoid spaces $\underline{H}\times X\rightarrow X\times X$ is
an injection. As $\underline{H}\times X$ is an equivalence relation in
$X\times X$ and the projection map
$\underline{H}\times X\rightarrow X$ is pro\'etale {(it suffices to check one projection, by symmetry of the equivalence relation)}, the quotient $Y=X/\un{H}$
of $X$ by the equivalence relation $\underline{H}\times X$ in the
category of sheaves over $\Perf$ with the pro\'etale topology is a
diamond, the map $X\rightarrow Y$ is quasipro\'etale surjective and we
have an isomorphism of perfectoid spaces
$\underline{H}\times X\xrightarrow{\sim} X\times_YX$ \cite[Prop.~11.3]{ScholzeEtcohDia} so that the map
$X\rightarrow Y$ is a pro\'etale $H$-torsor \cite[Def.~10.12, {Lemma 10.13}]{ScholzeEtcohDia}.

\begin{prop1}\label{avoir}
For all $p>f$ and $i\geq 0$ we have a canonical $P$-equivariant isomorphism
\[H^i(\zok,\cO_{\zok}) \congto H^i_{\textnormal{cts}}(G,\cO_{\zlt}(\zlt)).\]
\end{prop1}
\begin{proof}
We recall that we have an action $\un G \times \zlt \to \zlt$ and we claim that the subgroup $H := \ker((K^\times)^f \to K\s) = (K^\times)^f \cap G$ acts freely on $\zlt$. %
    It suffices to show that $K\s$ acts freely on $\Spa \F\ppar{T^{q^{-\infty}}}$, namely (e.g., see \cite[(27)]{BHHMS3}) that $K\s$ has trivial stabilizers
on $k^{\circ\circ} \setminus \{0\}$ for pairs $(k,k^+)$ as above; this is true because the $K\s$-action is the restriction of a ring action of $K$ on $k^{\circ\circ}$. (Here $k^{\circ\circ}$ denotes the topologically nilpotent
elements in $k$.)

    Let $H' := (p^\Z)^f \cap G \subset H$, so that $G/H'$ is profinite and we observe that $\zlt/\un{H'}$ is a qcqs perfectoid space.
    (It is quasi-compact perfectoid by (the proof of) Lemma~\ref{lem:preimage-m} and it is quasi-separated, since $\zlt$ is quasi-separated, as it is an increasing union of affinoid open subsets, and $\zlt \to \zlt/\un{H'}$ is a local isomorphism.)
    We need to check that $m' : \zlt/\un{H'} \to \zok$ is quasipro\'etale.
    Note that the $H'$-torsor $\zlt \to \zlt/\un{H'}$ is quasipro\'etale by \cite[Lemma 10.13]{ScholzeEtcohDia} and the composition $m : \zlt\rightarrow\zok$ is quasipro\'etale by \cite[Lemma~7.6]{FarguesAJ}.
    On the other hand, we note that $m'$ is separated.
    (It is qs since both domain and target are qs, and then separatedness follows from \cite[Prop.~10.9]{ScholzeEtcohDia} using that any $\Spa(k,\cO_k)$-point of $\zlt/\un{H'}$, where $k$ is any perfectoid field, lifts to $\zlt$, and $\zlt$ is partially proper.)
    Then $m'$ is quasipro\'etale by \cite[Prop.~11.30]{ScholzeEtcohDia}.
    Therefore the result is a consequence of Proposition \ref{prop:group_cohomology_trivial_bundle} below (where we need $p>f$).
\end{proof}

\begin{cor1}\label{vaninshingi>1}
For $p>f$ and $i>f-1$ we have $H^i_{\textnormal{cts}}(G,\cO_{\zlt}(\zlt))=0$.
\end{cor1}
\begin{proof} 
  Note first that $\zok$ is a locally spectral topological space %
  {that is also quasi-compact \ and \ quasi-separated, \ cf.\ \S~\ref{geometricinter}.}
  \ Hence \ it \ is \ spectral \ (e.g.~see \ \cite[Prop.~I.2.2.9(iii)]{Morel}). But the cohomology of any spectral topological spaces vanishes in degree greater than its dimension (see for instance \cite[Tag \href{https://stacks.math.columbia.edu/tag/0A3G}{0A3G}]{stacks-project}). Since $\zok$ has dimension $f-1$, the statement follows from Proposition \ref{avoir}.
\end{proof}

We will work toward the remaining goal of verifying Proposition \ref{prop:group_cohomology_trivial_bundle}. We discussed free quotients above. Now we consider quotients of perfectoid spaces by (possibly) non-free actions of
locally profinite groups. Let us give some general
results.

\begin{lem1}\label{lemm:quotient_top_space}
  Let $X$ be a perfectoid space of characteristic $p$ and $G$ a
  locally profinite group acting on $X$. Assume that the quotient
  $X/\un{G}$ is representable by a perfectoid space $Z$. Then the
  topological space $\vabs{Z}$ is homeomorphic to the quotient space
  $\vabs{X}/G$.
\end{lem1}

\begin{proof}
{Let $R := X\times_Z X \in \Perf$.
  Then the quotient v-stack $X/R$ is in fact a v-sheaf, which is isomorphic to $Z$.
  On the other hand, we have a surjection of v-sheaves $\widetilde{R}:=\underline{G}\times X \to R$, $(g,x) \mapsto (x,gx)$, and hence by \cite[Prop.~12.7]{ScholzeEtcohDia} we deduce that $|X|/|\wt R| \cong |Z|$ as topological spaces.
  Finally observe that $|\wt R| \cong G \times |X|$, so that $|X|/|\wt R| \cong |X|/G$.
  (Alternatively we may observe that \cite[Prop.\ 11.13]{ScholzeEtcohDia} applies with the equivalence relation $R$ to deduce that $|X|/|R| \cong |Z|$.
  Here we observe that the proof of \emph{loc.~cit.} applies despite not assuming that the projections $R \to X$ are pro\'etale since $R$ is of the form $X \times_Z X$ for some $Z \in \Perf$.
  Finally we use that $|\wt R| \onto |R|$, as any map in $\Perf$ that is surjective as v-sheaves is set-theoretically surjective, as v-covers are set-theoretically surjective.)
}
\end{proof}

\begin{prop1}\label{prop:quotient_by_finite_group_have_perf_fibers}
  Let $X$ be a {qcqs} perfectoid space of characteristic $p$ and $G$ a {profinite} group acting on $X$.
  Assume that the quotient $X/\un{G}$ is representable by a {quasi-separated} perfectoid space $Z$ and that the map $X\rightarrow Z$ is quasipro\'etale.
  Let $H\subset G$ be an open normal subgroup acting freely on $X$ and let $Y=X/\un{H}$ which is a diamond.
  Then, {for any totally disconnected perfectoid space $Z'$ over $Z$, the diamond $Y\times_Z Z'$ is affinoid perfectoid.}
\end{prop1}

Before initiating the proof of the proposition, we start with a lemma.

\begin{lem1}\label{lem:open-subdiam}
  Suppose that $f: Y \to Z$ is a map of diamonds.
  Then for any open subdiamond $V \subset Z$, the preimage $Y \times_Z V$ is the open subdiamond whose underlying topological space is $|Y \times_Z V| = |Y| \times_{|Z|} |V|$.
\end{lem1}

By an open subdiamond of a diamond $Y$ we mean an open subfunctor $Y' \to Y$ (i.e.,\ for any $T \in \Perf$ the base change $T \times_Y Y' \to T$ is an open immersion in $\Perf$).

\begin{proof}
  It is clear that $Y \times_Z V$ is an open subdiamond, so that in particular $|Y \times_Z V|$ is an open subspace of $|Y|$ by \cite[Prop.~11.15]{ScholzeEtcohDia}.
  It remains to show that the natural map $|Y \times_Z V| \to |Y| \times_{|Z|} |V|$ is bijective, which follows from \cite[Prop.~11.13]{ScholzeEtcohDia}.
\end{proof}

We also briefly recall connected components of spatial diamonds.
If $Y$ is a spatial diamond \cite[Def.~11.17]{ScholzeEtcohDia}, then $|Y|$ is spectral and open subfunctors $U \subset Y$ correspond bijectively to open subsets $|U| \subset |Y|$.
Such a $U$ can be recovered from the open subset $|U|$ via $U(T) = \{ T \to Y : |T| \to |Y| \text{ factors through $|U|$}\}$ \cite[Prop.\ 11.15]{ScholzeEtcohDia}.
Fix now a connected component $|Y_c| \subset |Y|$.
Then we can define $Y_c$ as a subdiamond of $Y$ as follows: $Y_c := \plim_U U$, where $U$ runs over all open subfunctors of $Y$ such that $|Y_c| \subset |U|$ is a clopen subset.
Indeed, $Y_c$ is a diamond and its underlying topological space is $|Y_c| = \plim_U |U|$ (\cite[Lemma 11.22]{ScholzeEtcohDia}\footnote{The transition maps are qcqs, since each $U$ is qcqs by \cite[Prop.~11.19(iii)]{ScholzeEtcohDia}.}), which is the subspace $|Y_c|$ of $|Y|$ that we started with, since in any spectral space a connected component is the intersection of all clopen subsets containing it.
Moreover, by construction,
\begin{equation}\label{eq:Yc}
  Y_c(T) = \{ T \to Y : |T| \to |Y| \text{ factors through $|Y_c|$}\}.
\end{equation}
If $Y$ is perfectoid, then $Y_c$ is perfectoid by \cite[Prop.\ 6.5]{ScholzeEtcohDia} (checking over each affinoid open of $Y$, noting also that a clopen subspace of an affinoid perfectoid space $\Spa(A,A^+)$ is affinoid perfectoid, since it corresponds to an idempotent of $A^+$).

\begin{proof}[Proof of Proposition~\ref{prop:quotient_by_finite_group_have_perf_fibers}]
  {First, it is easy to check that if $W \to Z$ is any map in $\Perf$ with $W$ qcqs, then $X \times_Z W \to W$ satisfies the assumptions of the lemma, as $(X \times_Z W)/\un{G} \cong (X/\un{G}) \times_Z W \cong W$ and $X \times_Z W$ is qcqs.
    Hence we may assume that $Z$ is totally disconnected, so that $Z$ is in particular affinoid \cite[Lemma 7.5]{ScholzeEtcohDia}.
  }

  {We suppose first that $Z$ is moreover connected, i.e.,\ $Z = \Spa(k,k^+)$ for some perfectoid field $k$ and some open bounded valuation ring $k^+$ (cf.\ \cite[Prop.\ 1.15]{ScholzeEtcohDia}).}
  As $H$ acts freely on $X$, the quotient $Y=X/\un{H}$ is a diamond and
  the map $X\rightarrow Y$ is quasipro\'etale {(see the discussion before Lemma~\ref{lemm:quotient_top_space})}. Let
  $\widetilde{Z}=\Spa(C,C^+)\rightarrow Z$ be a pro\'etale
  surjective map with $(C,C^+)$ an algebraically closed perfectoid
  field. Let $\widetilde{X}\coloneqq X\times_Z\widetilde{Z}$ and
  $\widetilde{Y}\coloneqq Y\times_Z\widetilde{Z}$. As the map
  $X\rightarrow Z$ is quasipro\'etale and $\widetilde{Z}$ is strictly
  totally disconnected {(\cite[Def.\ 1.14]{ScholzeEtcohDia})}, the map
  $\widetilde{X}\rightarrow\widetilde{Z}$ is pro\'etale by
  definition (\cite[Def.\ 1.20]{ScholzeEtcohDia}). As $Z$ represents the quotient of $X$ by $G$, the space
  $\widetilde{Z}$ represents the quotient of $\widetilde{X}$ by
  $G$. It follows from Lemma \ref{lemm:quotient_top_space} that
  $\vabs{\widetilde{Z}}$ is the quotient of
  $\vabs{\widetilde{X}}$ by the action of $G$. Let
  $x\in \widetilde{X}$ be in the preimage of the closed point of
  $\widetilde{Z}$. There exists an affinoid perfectoid neighborhood
  $U$ of $x$ in $\widetilde{X}$ such that
  $U\rightarrow\widetilde{Z}$ is affinoid pro\'etale. Therefore $U$ is
  isomorphic to $\widetilde{Z}\times\underline{S}$ for some
  profinite set $S$ so that the localization of $\widetilde{X}$ at
  $x$ is isomorphic to $\Spa(C,C^+)$ {(cf.\ for instance \cite[Rk.\ 8.2.3]{SWBerkeley})}. Let $H'\subset G$ be the
  stabilizer of $x$, which is a finite group, as $H$ acts freely on
  $X$. As the localization of $\widetilde{X}$ at $x$ is a totally
  ordered set of generizations of $x$ {\cite[Lemma III.5.1.8]{Morel}}, it follows that $H'$ is also
  the stabilizer of any point of {$\wt X_x$}. %
  Therefore we have a $\un{G}$-equivariant map of perfectoid spaces
  $f : \underline{G/H'}\times\Spa(C,C^+)\rightarrow
  \widetilde{X}$. As $\widetilde{Z}$ is exactly the quotient of
  $\widetilde{X}$ by $G$, {and by using Lemma~\ref{lemm:quotient_top_space}}, it follows that the map $f$ is
  bijective {on the underlying topological spaces}. It follows from \cite[Lemma~5.4]{ScholzeEtcohDia} that
  $f$ is an isomorphism, as all residue fields of $\widetilde{X}$
  are isomorphic to $C$.

  It follows that
  $\widetilde{Y}\simeq\underline{G/(H'\cdot H)}\times
  \widetilde{Z}$ is therefore an affinoid perfectoid space. Moreover
  the map $\widetilde{Y}\rightarrow\widetilde{Z}$ is clearly finite \'etale. As {$Z$} is totally
  disconnected and $Z$, $\widetilde{Z}$, $\widetilde{Y}$ are
  all affinoid perfectoid, it follows from
  \cite[Prop.~9.3, Cor.~9.11(iii)]{ScholzeEtcohDia} (see also
  \cite[Thm.~9.1.3(1),(4)]{SWBerkeley}) that $\widetilde{Y}$ descends to
  an affinoid perfectoid space {finite \'etale} over $Z$, i.e.,~that $Y$ is
  affinoid perfectoid and that the map $Y\rightarrow Z$ is
  {finite \'etale.
  Hence $Y = \Spa(A,A^+)$ with $k(z) \to A$ a finite \'etale ring homomorphism.}

  Now suppose that $Z$ is an arbitrary totally disconnected space in $\Perf$.
  We check that the diamond $Y=X/\un{H}$ is spatial.
  The map $X \to Y$ is surjective, quasipro\'etale, and universally open by \cite[Lemma 10.13]{ScholzeEtcohDia}.
  By \cite[Prop.~11.24]{ScholzeEtcohDia} it is enough to show that $Y$ is qcqs as v-sheaf (as defined in \cite[\S~8]{ScholzeEtcohDia}).
  As $X$ is qc and $X \to Y$ is surjective, it follows that $Y$ is qc.
  To check that $Y$ is qs, by the discussion preceding \cite[Prop.\ 8.3]{ScholzeEtcohDia} (using $X$ qcqs and $X \to Y$ surjective), it suffices to check that $X \times_Y X$ is qc.
  But $X \times_Y X \cong \un H \times X$, which is qc, as $H$ profinite and $X$ is qc.

  To show the result, we will show that $Y$ is affinoid perfectoid by showing that any connected component of $Y$ is affinoid perfectoid and applying \cite[Lemma 11.27]{ScholzeEtcohDia}.
  Fix now a connected component $Y_c$ of $Y$ and let $Z_c$ denote the connected component of $Z$ that contains the image of $Y_c$, so that $Z_c = \Spa(k,k^+)$ as in the first part of the proof.
  Using~\eqref{eq:Yc} we obtain a map of diamonds $Y_c \to Z_c$ and hence a commutative diagram
  \begin{equation*}
    \xymatrix{Y_c\ar[rd]\ar[rrd]\ar[rdd] \\
      & Y_c' \ar[r]\ar[d] & Z_c\ar[d] \\
      & Y \ar[r] & Z}
  \end{equation*}
  where $Y_c' := Y \times_Z Z_c$ is affinoid perfectoid by the first part of the proof.
  Note that we have subfunctors $Y_c \subset Y_c' \subset Y$ and $Z_c \subset Z$.
  For any clopen subfunctor $V \subset Z$ we have $|Y \times_Z V| = |Y| \times_{|Z|} |V|$ by Lemma~\ref{lem:open-subdiam}, which is qcqs, as it is open and closed in $|Y|$ and $|Y|$ is qcqs by above.
  Moreover, $|Y_c'| = |\plim_{z \in V} (Y \times_Z V)| = \plim_{z \in V} |Y \times_Z V| = \plim_{z \in V} (|Y| \times_{|Z|} |V|) = |Y| \times_{|Z|} \plim_{z \in V} |V| = |Y| \times_{|Z|} |Z_c|$ (using \cite[Lemma 11.22]{ScholzeEtcohDia}, noting that the transition maps in $\plim_{z \in V} (Y \times_Z V)$ are qcqs, as each term is qcqs).
  It follows that
  \begin{equation}\label{eq:Yz}
    Y_c'(T) = \{ T \to Y : |T| \to |Y| \text{ factors through $|Y_c'|$}\}.
  \end{equation}

  In particular, $|Y_c| \subset |Y_c'| \subset |Y|$ carry subspace topologies, so that $|Y_c|$ is also a connected component of $|Y_c'|$ (as topological spaces).
  Hence $Y_c$ is a connected component of $Y_c'$ (as diamonds) by comparing \eqref{eq:Yc}, \eqref{eq:Yz} and using the discussion preceding \eqref{eq:Yc}.
  We conclude as any connected component of an affinoid perfectoid space is affinoid perfectoid \cite[Prop.\ 6.5]{ScholzeEtcohDia}.
\end{proof}

Now we want to make precise what we mean by equivariant vector bundles over
perfectoid spaces.

Let $\cF$ be a vector bundle over $X \in \Perf$, i.e.,~a locally finite free
sheaf of $\cO_X$-modules. {We keep ourselves in the setting where the locally profinite group $H$ acts freely on $X$}. An \emph{$H$-equivariant structure on $\cF$} is
the data of $\cO_X$-linear isomorphisms
$c_h : \cF\xrightarrow{\sim}h^*\cF$ for all $h\in H$ such that, for
$h,h'\in H$, we have $c_{hh'}=(h')^*(c_h)\circ c_{h'}$. Giving a
map $c_h$ is equivalent to giving a family of isomorphisms
$c_{h,U} : \cF(U)\rightarrow\cF(h(U))$ for $U$ open in $X$ such that,
for $s\in\cF(U)$ and $f\in\cO_X(U)$, we have $c_h(fs)=h(f)c_h(s)$, with
$h(f)=f(h^{-1}\cdot)$. We require moreover a continuity condition which is
the following: for any quasi-compact open subset $U\subset X$ and any
compact open subgroup $H'\subset H$ such that $H'U=U$, the action of $H'$
on {the Banach space} $\cF(U)$ is continuous.
{(Note that such $H'$ always exist by continuity of the map $|\alpha| : |\un H \times X| \cong H \times |X| \to |X|$.)}
An $H$-equivariant structure on a vector
bundle over $X$ is equivalent to an isomorphism
$\theta : \pr_2^*\cF\xrightarrow{\sim}\alpha^*\cF$ of
$\cO_{\underline{H}\times X}$-modules over $\underline{H}\times X$
satisfying $\beta^*(\theta)\circ\pr_{2,3}^*(\theta)=\gamma^*(\theta)$
over $\underline{H}\times\underline{H}\times X$ where $\pr_{2,3}$ is
the projection $(\pr_2,\pr_3)$ on $\underline{H}\times X$, $\beta$ is
$(\pr_1,\alpha)$, $\gamma$ is $(\mu,\pr_3)$ (where $\mu$ is the group
law on $\underline{H})$. Indeed, pulling back such an isomorphism
$\theta$ along the section $X\rightarrow\underline{H}\times X$
corresponding to $x\mapsto (h,x)$ gives an isomorphism $c_h$ and the
family $(c_h)$ satisfies the properties of an $H$-structure on
$\cF$. The continuity condition comes from the fact that, for any
$U\subset X$ quasi-compact open subset and any compact open $H'\subset H$ 
subgroup such that $H'U=U$, {by taking sections over $\underline{H'}\times U$,} the isomorphism $\theta$ induces an
endomorphism of ${(\pr_2^*\cF)(\underline{H'}\times U) \cong (\alpha^*\cF)(\underline{H'}\times U) \simeq {}}
\cC^0(H',\cF(U))$ {\cite[(2.1)]{Heuer_LB}}
which is given by $f\mapsto (h\mapsto h\cdot f(h))$. The existence of this
isomorphism proves the continuity of the action. Conversely assume
that we have an $H$-equivariant structure over $\cF$ satisfying our
continuity condition. For any $x\in X$, there is a quasi-compact open
neighborhood $U$ of $x$ and a compact open subgroup $H'\subset H$ such
that $H'U=U$ and such that $\cF$ is free over $U$. The group $H'$ acts
continuously on $\cF(U)$ which allows us to construct the isomorphism
$\theta : (\pr_2^*\cF)(\underline{H'}\times
U)\xrightarrow{\sim}(\alpha^*\cF)(\underline{H'}\times U)$ using the
formula above.

{If $X \in \Perf$ we denote by $X_{\proet}$ its pro\'etale site \cite[Def.\ 8.1]{ScholzeEtcohDia} and if $Y$ is a diamond we denote by $Y_{\qproet}$ its quasipro\'etale site \cite[Def.\ 14.1]{ScholzeEtcohDia}.
  We let $Y_{\qproetp}$ denote the full subcategory of $Y_{\qproet}$ consisting of all perfectoid objects in $Y_{\qproet}$ and covers consisting of all families of maps that are covers in $Y_\qproet$, equivalently jointly surjective maps of v-sheaves, so that we have an equivalence of topoi $Y_{\qproetp}^\sim \cong Y_{\qproet}^\sim$ by \cite[Tag \href{https://stacks.math.columbia.edu/tag/03A0}{03A0}]{stacks-project} (applied with $u$ the inclusion $Y_{\qproetp} \to Y_{\qproet}$; note that any object of $Y_\qproet$ is a diamond \cite[Prop.\ 11.7]{ScholzeEtcohDia} and hence is covered by an object of $\Perf$).
  We may then tacitly work with $Y_{\qproetp}$ instead of $Y_{\qproet}$.
  In particular, we let $\cO$ denote the structure sheaf on either $X_{\proet}$ or $Y_{\qproet}$, sending $W$ in the site to $\cO_W(W)$, which is a sheaf by \cite[Cor.~8.6, Thm.\ 8.7]{ScholzeEtcohDia}.
  We also note that if $X \in \Perf$ the map of sites $X_\qproet \to X_\proet$ induces an equivalence of topoi by \cite[Tag \href{https://stacks.math.columbia.edu/tag/03A0}{03A0}]{stacks-project}.
  {(This is similar to above, noting first that covers consist of jointly surjective maps of v-sheaves on both sides.
    By \cite[Lemma 7.14]{ScholzeEtcohDia} there exists a pro\'etale cover $\{X_i \to X\}$ such that each $X_i$ is strictly totally disconnected, hence by \cite[Def.~10.1]{ScholzeEtcohDia} for any quasipro\'etale $Y \to X$ the pullback cover $\{X_i \times_X Y \to Y\}$ is such that $X_i \times_X Y$ is perfectoid and $X_i \times_X Y \to X_i \to X$ is pro\'etale for all $i$.)}
}

If $\cF$ is a vector bundle over $X$ we can extend it to a sheaf
$\cF_{\proet}$ of $\cO$-modules over $X_{\proet}$ by pullback, so that $\cF_{\proet}$ is a vector
bundle over $X_{\proet}$ {(cf.\ \cite[Thm.\ 3.5.8]{KL2})}. Moreover if $\cF$ has an $H$-equivariant
structure, then using the isomorphism $\theta$ in the previous paragraph we see that the
$H$-equivariant structure extends automatically to
$\cF_{\proet}$ {(as $\cF \mapsto \cF_{\proet}$ is compatible with pullbacks)}.

{Now suppose again we are in the setting of Proposition~\ref{prop:quotient_by_finite_group_have_perf_fibers} so that $H$ acts freely on $X$.}
As the map of diamonds $f : X\rightarrow Y {{}= X/\un{H}}$ is quasipro\'etale and
surjective, it follows from \cite[Lemma~17.1.8]{SWBerkeley} that, if
$\cF$ is an $H$-equivariant vector bundle over $X_{\proet}$
(equivalently over $X_{\qproet}$), then $(f_*\cF)^H$ is a vector
bundle over $Y_{\qproet}$. Namely the isomorphism
$\theta : \pr_2^*\cF\xrightarrow{\sim}\alpha^*\cF$ provides a descent
datum on $\cF$ along the map $X\rightarrow Y$.
{Conversely, if $\cG$ is a vector bundle on $Y_\qproet$, then $f^* \cG$ is an $H$-equivariant vector bundle on $X_\qproet$.
  (More precisely, $f : X_\qproet \to Y_\qproet$ is a morphism of ringed sites, with $(f_*\cF)(W) = \cF(W \times_Y X)$ for $W \in Y_\qproet$ and $(f^*\cG)(T) = \cG(T)$ for $T \in X_\qproet$.
  Then $(f_*\cF)^H$ is the equalizer of the maps $f_* \cF \rightrightarrows g_*\pr_2^* \cF \cong g_*\alpha^* \cF$ of $\cO_Y$-modules, where $g := f \circ \pr_2 = f \circ \alpha : \un H \times X \to Y$ and the isomorphism is given by $\theta$.)
}
We obtain in this way an
equivalence between the category of locally {finite} free sheaves over
$Y_{\qproet}$ and the category of $H$-equivariant locally {finite} free sheaves
over $X_{\proet}$.

\begin{lem1}\label{lem:coho-vb}
  Suppose that $X$ is a quasi-Stein perfectoid space.
  Then we have a canonical equivalence between vector bundles on $X_v$, $X_\proet$, $X_\an$. If $\cF$ is such a vector bundle, then $H^q(X_v,\cF) = H^q(X_\proet,\cF) = H^q(X_\an,\cF) = 0$ for all $q > 0$.
\end{lem1}

Here, by a slight abuse of notation, we identify an analytic vector bundle $\cF$ with its pullbacks to $X_\proet$ and $X_v$.

\begin{proof}
  The first part follows from \cite[Thm.~3.5.8]{KL2}, and $H^q(X_\an,\cF) = 0$ for $q > 0$ holds by \cite[Thm.~2.6.5(c)]{KL2}.
  By applying the proof of \cite[Thm.~2.6.5(c)]{KL2} with \cite[Thm.~3.5.6]{KL2} instead of \cite[Thm.~2.5.1]{KL2} we deduce that $H^q(X_v,\cF) = 0$ for $q > 0$.
  Then $H^q(X_\proet,\cF) = 0$ for $q > 0$ follows from \cite[Thm.~3.5.9]{KL2}.
\end{proof}

{Assuming that $X$ is quasi-Stein, }
we deduce from
\cite[{Cor.~2.9}]{Heuer_LB} {and Lemma~\ref{lem:coho-vb}} that we have a natural isomorphism, for any
$n\geq0$,
\begin{equation}
  \label{eq:Cartan_Leray}
  H^n(Y_{\qproet},(f_*\cF)^H) \simeq H^n\cts(H,\cF(X)).
\end{equation}
{(Note that $X \times \un{H^n}$ is quasi-Stein for all $n \ge 0$.)}

Note that if $\cF=\cO_X$, then $(f_*\cO_X)^H\simeq\cO_Y$ so that
(\ref{eq:Cartan_Leray}) becomes
\begin{equation}
  \label{eq:Cartan_Leray_OX}
  H^n(Y_{\qproet},\cO_Y) \simeq H^n\cts(H,\cO_X(X)).
\end{equation}

\begin{lem1}\label{lemm:push_forward_along_finite_quotient}
  Let $h : Y\rightarrow Z$ be a map of diamonds with $Z$ a perfectoid
  space. Assume that for any {totally disconnected perfectoid space $Z'$} over $Z$, the fiber product
  ${Y' :={}} Y\times_Z Z'$ is representable by an affinoid perfectoid
  space. %
  Then, for $q>0$, $R^qh_*\cO_Y=0$ {in the quasipro\'etale topology}.
\end{lem1}

\begin{proof}
  {Suppose \ that \ $Z'$ \ is \ totally \ disconnected \ in \ $\Perf$.} \ By \ \cite[Exp.~V, Prop.~5.1(3)]{SGA4vol2}, we have
  an isomorphism {$R^qh'_*\cO_{Y'}\simeq (R^qh_*\cO_Y)_{|Z'}$, where $h'$ denotes the base change $Y' \to Z'$ of $h$}. {As $Y'$, $Z'$ are affinoid perfectoid by assumption,} it follows from
 {Lemma~\ref{lem:coho-vb} as well as \cite[Exp.~V, Prop.~5.1(1)]{SGA4vol2}} that $(R^qh_*\cO_Y)_{|Z'}=0$ for any
  $q>0$.
  {By taking a quasipro\'etale cover of $Z$ by totally disconnected perfectoid spaces (using \cite[Lemma 7.18]{ScholzeEtcohDia}), we deduce that} $R^qh_*\cO_Y=0$ for $q>0$.
\end{proof}

In what follows, for $Z$ a perfectoid space, we will denote
$\nu_Z : Z_{\proet}\rightarrow Z_{\an}$ the natural morphism {of sites, where $Z_{\an}$ denotes the analytic site of $Z$}.

\begin{prop1}\label{prop:group_cohomology_trivial_bundle}
  Let $X$ be a quasi-Stein perfectoid space of characteristic $p$ and
  $G$ a locally profinite group acting on $X$. Assume that the
  quotient $X/\un{G}$ is representable by a quasi-separated perfectoid space $Z$. Assume there exist closed
  normal subgroups $H' \subset H$ of $G$ acting freely on $X$ such that $X/\un{H'}$ is qcqs perfectoid {which is quasipro\'etale over $Z$}, $G/H'$ is profinite, and $G/H$ is
  finite of cardinality prime to $p$. Then we have an isomorphism, for
  every $q\geq0$,
  \[ H^q\cts(G,\cO_X(X))\simeq H^q(Z,\cO_Z). \]
\end{prop1}

\begin{proof}
  Let $\Gamma=G/H$ and let $Y:=X/\un{H}$. Denote $g : X\rightarrow Y$ and
  $h : Y\rightarrow Z$ so that $f=h\circ g$. It follows from
  (\ref{eq:Cartan_Leray_OX}), that we have
  $H^q(Y_{\qproet},\cO_Y)=H^q\cts(H,\cO_X(X))$ for every $q\geq0$.

{By the same proof as for Lemma~\ref{lm:pushforward-sheaf}} we have $(f_*\cO_X)^G\simeq (h_*\cO_Y)^{\Gamma}\simeq\cO_Z$.
  From Proposition~\ref{prop:quotient_by_finite_group_have_perf_fibers} and Lemma~\ref{lemm:push_forward_along_finite_quotient} {applied to $G/H'$ acting on $Y' := X/\un{H'}$}, we have
  $R^qh_*\cO_Y=0$ for $q>0$ so that, for every $q\geq0$,
  \[
    H^q(Z_{\proet},h_*\cO_Y)\simeq H^q(Y_{\qproet},\cO_Y). \] As the
  group $\Gamma$ has order prime to $p$, the sheaf
  $\cO_Z\simeq(h_*\cO_Y)^{\Gamma}$ is a direct factor of $h_*\cO_Y$ and we have
    \[
      H^q(Z_{\proet},\cO_Z)\simeq
      H^q(Y_{\qproet},\cO_Y)^{\Gamma}\simeq H^q\cts(H,\cO_X(X))^{\Gamma}. \] Finally, using again that
    $\Gamma$ has order prime to $p$, we deduce
    $H^q\cts(H,\cO_X(X))^{\Gamma}\simeq
    H^q\cts(G,\cO_X(X))$. Finally as $Z$ is
  perfectoid, it follows from \cite[Prop.~8.5(iii)]{ScholzeEtcohDia}
  that $R^q\nu_{Z,*}\cO_Z=0$ for $q>0$ so that
  \[ H^q(Z_{\proet},\cO_Z)\simeq H^q(Z,\cO_Z), \] which allows us to
  conclude.
\end{proof}

\bibliography{Biblio}

\newcommand{\etalchar}[1]{$^{#1}$}
\providecommand{\bysame}{\leavevmode\hbox to3em{\hrulefill}\thinspace}
\providecommand{\MR}{\relax\ifhmode\unskip\space\fi MR }
\providecommand{\MRhref}[2]{%
  \href{http://www.ams.org/mathscinet-getitem?mr=#1}{#2}
}
\providecommand{\href}[2]{#2}
\begin{thebibliography}{BHH{\etalchar{+}}25b}

\bibitem[BD14]{BD}
Christophe Breuil and Fred Diamond, \emph{Formes modulaires de {H}ilbert modulo
  {$p$} et valeurs d'extensions entre caract\`eres galoisiens}, Ann. Sci.
  \'{E}c. Norm. Sup\'{e}r. (4) \textbf{47} (2014), no.~5, 905--974.
  \MR{3294620}

\bibitem[Ber]{lucrezia}
Lucrezia Bertoletti, \emph{Finite {Length} for {Unramified} $\mathrm{GL_2}$:
  {Beyond} {Multiplicity} {One}}, \url{https://arxiv.org/pdf/2505.20134},
  preprint (2025).

\bibitem[BHH{\etalchar{+}}a]{BHHMS4}
Christophe Breuil, Florian Herzig, Yongquan Hu, Stefano Morra, and Benjamin
  Schraen, \emph{Finite length for unramified {$\mathrm{GL}_2$}},
  \url{https://arxiv.org/pdf/2501.03644}, preprint (2025).

\bibitem[BHH{\etalchar{+}}b]{BHHMS5}
\bysame, \emph{On the constituents of the mod {$p$} cohomology of {S}himura
  curves}, \url{https://arxiv.org/pdf/2506.16293}, preprint (2025).

\bibitem[BHH{\etalchar{+}}23]{BHHMS1}
\bysame, \emph{Gelfand-{K}irillov dimension and {${\rm mod}\, p$} cohomology
  for {$\rm GL_2$}}, Invent. Math. \textbf{234} (2023), no.~1, 1--128.
  \MR{4635831}

\bibitem[BHH{\etalchar{+}}25a]{BHHMS2}
\bysame, \emph{Conjectures and {R}esults on {M}odular {R}epresentations of
  {${\rm GL}_n({K})$} for a {$p$}-{A}dic {F}ield {$K$}}, Mem. Amer. Math. Soc.
  \textbf{315} (2025), no.~1598, v+163. \MR{5003478}

\bibitem[BHH{\etalchar{+}}25b]{BHHMS3}
\bysame, \emph{Multivariable {$(\varphi,\mathcal {O}_K^\times )$}-modules and
  local-global compatibility}, Math. Ann. \textbf{392} (2025), no.~2,
  2709--2801. \MR{4906333}

\bibitem[Bos]{bosco}
Guido Bosco, \emph{On the {$p$}-adic pro-\'etale cohomology of {D}rinfeld
  symmetric spaces}, \url{https://arxiv.org/pdf/2110.10683}, preprint (2023).

\bibitem[BP12]{BP}
Christophe Breuil and Vytautas Pa{\v s}k{\=u}nas, \emph{Towards a modulo {$p$}
  {L}anglands correspondence for {${\rm GL}_2$}}, Mem. Amer. Math. Soc.
  \textbf{216} (2012), no.~1016, vi+114. \MR{2931521}

\bibitem[BSSW]{BSSW}
Tobias Barthel, Tomer~M. Schlank, Nathaniel Stapleton, and Jared Weinstein,
  \emph{On the rationalization of the {$K(n)$}-local sphere},
  \url{https://arxiv.org/pdf/2402.00960}, preprint (2025).

\bibitem[CDP14]{CDP}
Pierre Colmez, Gabriel Dospinescu, and Vytautas Pa{\v s}k{\=u}nas, \emph{The
  {$p$}-adic local {L}anglands correspondence for {${\rm
  GL}_2(\mathbb{Q}_p)$}}, Camb. J. Math. \textbf{2} (2014), no.~1, 1--47.
  \MR{3272011}

\bibitem[Col10a]{Colmez}
Pierre Colmez, \emph{Repr\'{e}sentations \ de \ {${\rm GL}_2({\mathbb Q}_p)$} \
  et \ {$(\varphi,\Gamma)$}-modules}, Ast\'{e}risque (2010), no.~330, 281--509.
  \MR{2642409}

\bibitem[Col10b]{Colmez2}
\bysame, \emph{{$(\varphi,\Gamma)$}-modules et repr\'{e}sentations du
  mirabolique de {${\rm GL}_2(\mathbb{Q}_p)$}}, Ast\'{e}risque (2010), no.~330,
  61--153. \MR{2642405}

\bibitem[CS]{ClausenS}
Dustin Clausen and Peter Scholze, \emph{Lectures on condensed mathematics}, \
  \url{https://arxiv.org/abs/2605.03658}, 2026.

\bibitem[Eme]{emerton-local-global}
Matthew Emerton, \emph{Local-global \ \ compatibility \ \ in \ \ the \ \
  $p$-adic \ \ {L}anglands \ \ program \ \ for \ \ $\mathrm{GL}_{2/{\mathbb
  Q}}$}, \url{ http://www.math.uchicago.edu/~emerton/pdffiles/lg.pdf}, preprint
  (2011).

\bibitem[Far]{Fargues}
Laurent Fargues, \emph{Premiers pas en {L}anglands $p$-adique g\'eom\'etrique},
  draft (July 2023).

\bibitem[Far20]{FarguesAJ}
\bysame, \emph{Simple connexit\'{e} des fibres d'une application
  d'{A}bel-{J}acobi et corps de classes local}, Ann. Sci. \'{E}c. Norm.
  Sup\'{e}r. (4) \textbf{53} (2020), no.~1, 89--124. \MR{4093441}

\bibitem[FF18]{FF}
Laurent Fargues and Jean-Marc Fontaine, \emph{Courbes et fibr\'es vectoriels en
  th\'eorie de {H}odge {$p$}-adique}, no. 406, 2018, With a preface by Pierre
  Colmez. \MR{3917141}

\bibitem[Heu22]{Heuer_LB}
Ben Heuer, \emph{Line bundles on rigid spaces in the {$v$}-topology}, Forum
  Math. Sigma \textbf{10} (2022), Paper No. e82, 36. \MR{4487475}

\bibitem[Hu12]{yongquan-jussieu}
Yongquan Hu, \emph{Diagrammes canoniques et repr\'{e}sentations modulo {$p$} de
  {${\rm GL}_2(F)$}}, J. Inst. Math. Jussieu \textbf{11} (2012), no.~1,
  67--118. \MR{2862375}

\bibitem[Jen72]{jensen}
C.~U. Jensen, \emph{Les foncteurs d\'{e}riv\'{e}s de {$\varprojlim$} et leurs
  applications en th\'{e}orie des modules}, Lecture Notes in Mathematics, Vol.
  254, Springer-Verlag, Berlin-New York, 1972. \MR{407091}

\bibitem[Ked16]{Kedlaya15}
Kiran~S. Kedlaya, \emph{The {H}ochschild-{S}erre property for some {$p$}-adic
  analytic group actions}, Ann. Math. Qu\'e. \textbf{40} (2016), no.~1,
  149--157. \MR{3512526}

\bibitem[KL]{KL2}
Kiran~S. Kedlaya and Ruochuan Liu, \emph{Relative {{\(p\)}}-adic {Hodge}
  theory, {II}: Imperfect period rings},
  \url{https://arxiv.org/pdf/1602.06899}, preprint (2016).

\bibitem[KL15]{KL1}
\bysame, \emph{Relative {$p$}-adic {H}odge theory: foundations}, Ast\'{e}risque
  (2015), no.~371, 239. \MR{3379653}

\bibitem[Laz65]{lazard}
Michel Lazard, \emph{Groupes analytiques {$p$}-adiques}, Inst. Hautes
  \'{E}tudes Sci. Publ. Math. (1965), no.~26, 389--603. \MR{209286}

\bibitem[LvO96]{LiOy}
Huishi Li and Freddy van Oystaeyen, \emph{\ {Z}ariskian \ filtrations},
  $K$-Monographs in Mathematics, vol.~2, Kluwer Academic Publishers, Dordrecht,
  1996. \MR{1420862}

\bibitem[Mor]{Morel}
Sophie Morel, \emph{\ {A}dic spaces}, \
  \url{http://perso.ens-lyon.fr/sophie.morel/adic_notes.pdf}, 2019.

\bibitem[Pa{\v s}07]{paskunas-restriction}
Vytautas Pa{\v s}k{\=u}nas, \emph{On the restriction of representations of
  {${\rm GL}_2(F)$} to a {B}orel subgroup}, Compos. Math. \textbf{143} (2007),
  no.~6, 1533--1544. \MR{2371380}

\bibitem[RJRC22]{RodriguesJacinto2022}
Joaqu\'in Rodrigues~Jacinto and Juan~Esteban Rodr\'iguez~Camargo, \emph{Solid
  locally analytic representations of {$p$}-adic {L}ie groups}, Represent.
  Theory \textbf{26} (2022), 962--1024. \MR{4475468}

\bibitem[Sch]{ScholzeEtcohDia}
Peter Scholze, \emph{{\'E}tale cohomology of diamonds}, Ast\'{e}risque, to
  appear.

\bibitem[Sch12]{Scholzeperfectoid}
\bysame, \emph{Perfectoid spaces}, Publ. Math. Inst. Hautes \'Etudes Sci.
  \textbf{116} (2012), 245--313. \MR{3090258}

\bibitem[Sch15]{Benj}
Benjamin Schraen, \emph{Sur la pr\'{e}sentation des repr\'{e}sentations
  supersinguli\`eres de {${\rm GL}_2(F)$}}, J. Reine Angew. Math. \textbf{704}
  (2015), 187--208. \MR{3365778}

\bibitem[Sch18]{ScholzeLT}
Peter Scholze, \emph{On the {$p$}-adic cohomology of the {L}ubin-{T}ate tower},
  Ann. Sci. \'Ec. Norm. Sup\'er. (4) \textbf{51} (2018), no.~4, 811--863, With
  an appendix by Michael Rapoport. \MR{3861564}

\bibitem[SGA72]{SGA4vol2}
\emph{Th\'eorie des topos et cohomologie \'etale des sch\'emas. {T}ome 2},
  Lecture Notes in Mathematics, vol. Vol. 270, Springer-Verlag, Berlin-New
  York, 1972, S\'eminaire de G\'eom\'etrie Alg\'ebrique du Bois-Marie
  1963--1964 (SGA 4), Dirig\'e{} par M. Artin, A. Grothendieck et J. L.
  Verdier. Avec la collaboration de N. Bourbaki, P. Deligne et B. Saint-Donat.
  \MR{354653}

\bibitem[{Sta}]{stacks-project}
The {Stacks Project Authors}, \emph{\textit{Stacks Project}},
  \url{https://stacks.math.columbia.edu}.

\bibitem[SW20]{SWBerkeley}
Peter Scholze and Jared Weinstein, \emph{Berkeley lectures on {$p$}-adic
  geometry}, Annals of Mathematics Studies, vol. 207, Princeton University
  Press, Princeton, NJ, 2020. \MR{4446467}

\bibitem[Tat76]{tate1976}
John Tate, \emph{Relations between {$K\sb{2}$} and {G}alois cohomology},
  Invent. Math. \textbf{36} (1976), 257--274. \MR{429837}

\bibitem[Vie12]{BVi}
Mathieu Vienney, \emph{Repr\'{e}sentations modulo {$p$} d'un sous-groupe de
  {B}orel de {$GL_2(\mathbb{Q}_p)$}}, C. R. Math. Acad. Sci. Paris \textbf{350}
  (2012), no.~13-14, 651--654. \MR{2971374}

\bibitem[Vig08]{GAFAVigneras}
Marie-France Vign\'eras, \emph{S\'erie principale modulo {$p$} de groupes
  r\'eductifs {$p$}-adiques}, Geom. Funct. Anal. \textbf{17} (2008), no.~6,
  2090--2112. \MR{2399093}

\bibitem[Wana]{YW2}
Yitong Wang, \emph{Diamond \ diagrams \ and \ multivariable \
  {$(\varphi,\mathcal{O}_K^\times)$}-modules},
  \url{https://arxiv.org/pdf/2504.09270}, preprint (2025).

\bibitem[Wanb]{YW3}
\bysame, \emph{Lubin-{T}ate and multivariable
  {$(\varphi,\mathcal{O}_K^\times)$}-modules in dimension $2$},
  \url{https://arxiv.org/pdf/2404.00396}, preprint (2024).

\bibitem[Wanc]{YW}
\bysame, \emph{On the rank of the multivariable
  {$(\varphi,\mathcal{O}_K^\times)$}-modules associated to mod {$p$}
  representations of {$\rm{GL}_2(K)$}}, \url{https://arxiv.org/pdf/2404.00389},
  preprint (2024).

\bibitem[Wan23]{YW0}
\bysame, \emph{On the {${\rm mod} \, p$} cohomology for {${\rm GL}_2$}}, J.
  Algebra \textbf{636} (2023), 20--41. \MR{4637601}

\bibitem[Wed]{Wedhorn:adic}
Torsten Wedhorn, \emph{Adic spaces}, \url{https://arxiv.org/pdf/1910.05934},
  2019.

\bibitem[Wu21]{Wu}
Zhixiang Wu, \emph{A note on presentations of supersingular representations of
  {$\textrm{GL}_2(F)$}}, Manuscripta Math. \textbf{165} (2021), no.~3-4,
  583--596. \MR{4280498}

\end{thebibliography}

\bibliographystyle{amsalpha}

\end{document}